\documentclass[3p,times]{elsarticle_modified}
\usepackage{amsmath,amssymb,epsfig}

\usepackage{xcolor}
\usepackage{hyperref}
\mathchardef\ordinarycolon\mathcode`\:
\mathcode`\:=\string"8000\def\R{\mathbb{R}}
\begingroup \catcode`\:=\active
  \gdef:{\mathrel{\mathop\ordinarycolon}}\def\N{\mathbb{N}}
\endgroup

\newcommand{\N}{\mathbb{N}}

\renewcommand{\R}{\mathbb{R}}

\usepackage{mathrsfs}

\newcommand{\Ec}{\mathcal{E}}

\newcommand{\Ic}{\mathcal{I}}

\newcommand{\Rc}{\mathcal{R}}

\newcommand{\Sc}{\mathcal{S}}

\newcommand{\bdot}{\boldsymbol{\cdot}}                
\newcommand{\id}{\mathrm{id}} 						  %
\newcommand{\norm}[1]{\left\lVert#1\right\rVert}      
\newcommand{\abs}[1]{\left|#1\right|}                 
\newcommand{\paren}[1]{\left(#1\right)}               
\newcommand{\sparen}[1]{\left\{#1\right\}}		      
\newcommand{\h}[1]{\left\langle#1\right\rangle}       
\newcommand{\ran}[1]{\mathrm{ran}\left(#1\right)}     
\newcommand{\dom}[1]{\mathrm{dom}\left(#1\right)}     

\newcommand{\ceil}[1]{\left.\lceil #1\rceil\right.}
\renewcommand{\d}{\,\mathrm{d}}						  

\DeclareMathOperator*{\supp}{\mathrm{supp}}           
\DeclareMathOperator*{\sgn}{\mathrm{sgn}}

\DeclareMathOperator*{\argmin}{\mathrm{arg\:min}}

\renewcommand{\Im}{\mathrm{Im}}
\newcommand{\tvector}[2]{\begin{pmatrix}#1\\ #2\end{pmatrix}}

\renewcommand{\epsilon}{\varepsilon}
\renewcommand{\rho}{\varrho}
\renewcommand{\phi}{\varphi}
\renewcommand{\bar}[1]{\overline{#1}}

\usepackage[amsmath,standard,thmmarks]{ntheorem}

\makeatletter
\newtheoremstyle{mythm}%
{\item[\hskip\labelsep \theorem@headerfont ##1\ ##2\theorem@separator]}%
{\item[\hskip\labelsep \theorem@headerfont ##1\ ##2\ \normalfont{(##3)}\theorem@separator]}
\makeatother

\theoremstyle{mythm}
\theoremseparator{.}

\renewtheorem{Proof}{Proof}

\renewtheorem{theorem}{Theorem}[section]
\renewtheorem{lemma}[theorem]{Lemma}
\renewtheorem{corollary}[theorem]{Corollary}
\renewtheorem{proposition}[theorem]{Proposition}

\renewtheorem{definition}[theorem]{Definition}

\theorembodyfont{\rmfamily}
\newtheorem{algorithm}[theorem]{Algorithm}

\theoremstyle{nonumberbreak}
\theoremheaderfont{\itshape}
\renewtheorem{example}{Example}
\renewtheorem{remark}{Remark}

\theoremstyle{nonumberbreak}

\makeatletter
\newtheoremstyle{myproof}%
{\item[\hskip\labelsep \theorem@headerfont ##1\theorem@separator]}%
{\item[\hskip\labelsep \theorem@headerfont ##1\ ##3\theorem@separator]}
\makeatother

\theoremstyle{myproof}
\theoremheaderfont{\bfseries}
\theoremseparator{.}
\theoremsymbol{\hbox{}\hfill {\tiny$\blacksquare$}}
\renewtheorem{Proof}{Proof}
\renewtheorem{proof}{Proof}

\usepackage[babel]{csquotes}
\usepackage[labelfont=bf,font=small]{caption}
\usepackage[font=small]{subfig}
\usepackage{pgfplots}
\usepackage{enumitem}
\usepackage{booktabs}
\usepackage[ruled]{algorithm}
\usepackage{algorithmic}
\usepackage{float}

\begin{document}

\begin{frontmatter}
	\title{Sparse regularization in limited angle tomography}
	\author[a,b]{J\"urgen Frikel}
	\address[a]{ Institute of Biomathematics and Biometry, Helmholtz Zentrum M\"unchen, German Research Center for Environmental Health, Ingolst\"adter Landstra\ss e 1, D-85764, Germany}
	\address[b]{Zentrum Mathematik, M6, Technische Universit\"at M\"unchen, Germany}
	\date{\today}
	\begin{abstract}
		We investigate the reconstruction problem of limited angle tomography. Such problems arise naturally in applications like digital breast tomosynthesis, dental tomography, electron microscopy etc. Since the acquired tomographic data is highly incomplete, the reconstruction problem is severely ill-posed and the traditional reconstruction methods, such as filtered backprojection (FBP), do not perform well in such situations. 
	
		To stabilize the reconstruction procedure additional prior knowledge about the unknown object has to be integrated into the reconstruction process. In this work, we propose the use of the sparse regularization technique in combination with curvelets. We argue that this technique gives rise to an edge-preserving reconstruction. Moreover, we show that the dimension of the problem can be significantly reduced in the curvelet domain. To this end, we give a characterization of the kernel of limited angle Radon transform in terms of curvelets and derive a characterization of solutions obtained through curvelet sparse regularization.  In numerical experiments, we will present the practical relevance of these results.
	\end{abstract}

\end{frontmatter}

\section{Limited angle tomography: Introduction, Organization and Notations}
\label{sec:intro}

\subsection{Introduction}
\label{subsec:introduction}
Limited angle tomography problems arise naturally in many practical applications, such as di\-gital breast tomosynthesis, dental tomography, etc. The underlying principle of these imaging techniques consists in two steps: First, the data acquisition step, where a few x-ray projections of an object are taken from different view angles (within a limited angular range). Second, the reconstruction step, where the attenuation coefficient $f:\R^2\to\R$ of the object is approximately reconstructed from the given projection data. In this paper we are concerned with the second step, i.e., with the development of an appropriate (adapted) reconstruction technique which takes into account the special structure of the limited angle tomography.

To this end, we consider the \emph{Radon transform} as  mathematical model for the acquisition process which is defined by
\begin{equation}
\label{eq:radon transform}
	\Rc f(\theta,s)=\int_{L(\theta,s)} f(x)\d S(x),
\end{equation}
where $L(\theta,s)=\sparen{x\in\R^2:\;x_1\cos\theta+x_2\sin\theta=s}$ denotes the line with normal direction $(\cos\theta,\sin\theta)^T$ and a signed distance from the origin $s\in\R$. Furthermore, we assume $f$ to lie in the natural domain of the Radon transform, i.e., $f$ is such that \eqref{eq:radon transform} exists for all $(\theta,s)$. Whenever we write $\Rc_\theta f(s)$ instead of $\Rc f(\theta,s)$, we consider $\Rc f(\theta,s)$ as an univariate function of the second argument $s$ with a fixed angular parameter $\theta$. In this case, we will call the function $\Rc_\theta f$ a projection of $f$ at angle $\theta$.

In contrast to the classical computed tomography, in limited angle tomography the data $\Rc f(\theta,s)$ is known only  within a limited angular range, that is, for $\theta\in[-\Phi,\Phi]$ with $\Phi<\pi/2$. To emphasize that the Radon transform $\Rc f$ of a function is defined only on a limited angle domain $[-\Phi,\Phi]\times \R$, we will write $\Rc_\Phi f$ and call it the \emph{limited angle Radon transform}. As a consequence of the limited angular range, the reconstruction problem $y = \Rc_\Phi f$ becomes severely ill-posed \cite{Natterer86, Davison83}. Thus, small measurement errors can cause huge reconstruction errors. 

This is a serious drawback for practical applications since the acquired data is (to some extent) always corrupted by noise. The practical reconstruction problem is therefore given by the equation
\begin{equation}\label{eq:reconstruction problem}
	y^\delta = \Rc_\Phi f + \eta,
\end{equation}
where $\eta$ denotes the noise, $\delta>0$ is the noise level, i.e., $\norm{\eta}<\delta$. The aim is to find an approximation to $f$ from the noisy measurements $y^\delta$. 

It is well-known that classical reconstruction methods, such as filtered backprojection (FBP), do not perform well in such situations meaning that they are sensitive to noise \cite{Natterer86}. To stabilize the inversion additional prior knowledge about the solution has to be integrated into the reconstruction procedure \cite{Regularization_Engl96}. Usually, variational methods are used to obtain a regularized solution $f_\alpha$  of the reconstruction problem which is given as a minimizer of the so-called Tikhonov type variational functional
\begin{equation}\label{eq:tikhonov}
	T_\alpha (f) = \norm{\Rc_\Phi f-y^\delta}_2^2 + \alpha\Lambda(f),
\end{equation}
where $\alpha>0$ denotes a regularization parameter and $\Lambda:\dom\Lambda\to[0,\infty]$ is a convex and proper functional \cite{VariationalMethodsScherzer2009}. The first term  in \eqref{eq:tikhonov} - the data fidelity term - controls the data error, whereas the second term - the so-called penalty or prior term - encodes the prior information about the object. 

The choice among the various prior terms and, thus, regularization techniques depends on the specific object (which is imaged) and, to some extent, on the desire to preserve or emphasize particular features of the unknown object. Usual choices for $\Lambda$ are any kind of a smoothness (semi-) norms \cite{VariationalMethodsScherzer2009}. For instance, the Besov norm allows to adjust the smoothness of the solution at a very fine scale \cite{Kolehmainen03,Rantala2006,LorenzTrede2008}. Another prominent example in image reconstruction is the total variation (TV) norm which is used in particular for edge-preserving reconstruction, \cite{Fadili:cs,Hansen:2011wf}. 

Indeed, to preserve edges is an important issue for medical imaging. However, it was pointed out in \cite{HermanDavidi2008}, that TV reconstruction may be not an appropriate choice for medical imaging purposes. One reason for this is that TV regularization favors piecewise constant functions and, hence, produces staircase effects (cf. \cite{Ring:2000vq,Caselles:2007hk}) which may destroy relevant information. Hence, piecewise constant functions may be not appropriate for our purpose. To overcome this problem, higher order total variation priors were considered by some authors, see for example \cite{Bredies:2010gt}. In this work will use curvelets to avoid such problems while preserving edges of the reconstruction.

Another issue we are concerned with is the fact that in limited angle tomography one can not expect to get a perfect reconstruction (though the limited angle problem is uniquely solvable in some mathematical settings). Depending on the available angular range some structures of the unknown object can be reconstructed (are visible) and some can not be reconstructed (are invisible) \cite{Quinto93}. To our knowledge the information about the visible and invisible structures (which is encoded in the data set) is not exploited by any of the mentioned reconstruction methods.

In view of the above discussion, our goal in this work is to design a reconstruction method for limited angle tomography which is
\begin{enumerate}[label=(\roman{enumi})]
	\item stable, i.e., insensitive to noise,\label{enum:goal1}
	\item independent of acquisition geometry,\label{enum:goal2}
	\item edge-preserving,\label{enum:goal3}
	\item adapted to the limited angle setting, i.e., exploits information about visible and invisible structures.\label{enum:goal4}
\end{enumerate}
First thoughts on this topic have been formulated in an extended abstract (2 pages) that is submitted to the Proceedings in Applied Mathematics and Mechanics, \cite{FrikelPAMM2011}.

\subsection{Organization of this paper} 
In the first part of Section \ref{sec:stabilization} a brief description of the curvelet dictionary will be given. The second part of Section \ref{sec:stabilization} the technique of sparse regularization will be introduced as a stable reconstruction method. In Section \ref{sec:BCD reconstruction} we will discuss the relation between the curvelet sparse regularization technique and the curvelet thresholding which was proposed in \cite{Candes_Recovering_Edges_in_Illposed_problems02}. Based on this discussion, a characterization of the curvelet sparse regularization will be given for the full angular problem by using the biorthogonal curvelet decomposition (BCD) \cite{Candes_Recovering_Edges_in_Illposed_problems02}. Afterwards, the limitations of the BCD approach will be discussed. 

Our main results will be presented in Section \ref{sec:characterizations}. Here, we will first prove a characterization of the kernel of the limited angle Radon transform in terms of curvelets. As a consequence, a characterization of curvelet sparse regularizations will be derived. These results will be applied to a finite dimensional reconstruction problem in Section \ref{sec:adapted csr}. By performing dimensionality reduction of the reconstruction problem in the curvelet domain, an adapted curvelet sparse regularization approach will be introduced. In Section \ref{sec:discussion} we will discuss some of our results.

 We will conclude this paper by showing some numerical experiments in Section \ref{sec:results}. In particular, we will show that the execution times of the adapted curvelet sparse regularization significantly reduces while preserving the reconstruction quality.

\subsection{Notation}
\label{subsec:notation}
We state here some notations which will be used throughout this paper:

The inner product of $x,y\in\R^n$ will be denoted as $x\cdot y$ or simply $xy$. When not otherwise stated, inner product in a function space $X$ will be denoted by $\h{f,g}_{X}$. The norm of a vector $x\in\R^n$ will be denoted by $\abs{x}$ whereas the norm in a function space $X$ will be denoted by $\norm{f}_{X}$.

We will be using some classical function spaces, such as the space of Schwartz functions $\Sc(\R^n)$ and the spaces of measurable functions $L^p(\Omega)$, without reference since they can be found in every book on functional analysis. Same holds for the classical sequence spaces $\ell^p$. 

The Fourier transform $\hat f$ of a function $f\in\Sc(\R^n)$ is defined by
\begin{equation*}
	\hat f(\xi) = \paren{2\pi}^{-n/2}\int_{\R^n}f(x)e^{-ix\xi}\d x.
\end{equation*}
The inverse Fourier transform is given by $\check f(x)=\hat f(-x)$. Basic properties of the Fourier transform will be used without proof. For details about the Fourier transform we refer to \cite{SteinWeiss1971}.

For $\eta\in[0,2\pi]$, we define $\rho_\eta$ to be the rotation operator $\rho_\eta f(x)=f(R_\eta x)$, where the rotation matrix $R_\eta$ is defined by
\begin{equation*}
	R_\eta =
		\begin{pmatrix} 
			 \cos\eta &\sin\eta\\ 
			-\sin\eta &\cos\eta
		\end{pmatrix}.
\end{equation*}

Eventually, we refer to \cite{Natterer86} for notations and some basic facts about the Radon transform.

\section{Stabilization of limited angle reconstructions by sparsity in the curvelet domain}
\label{sec:stabilization}
In this section we are going to address our goals \ref{enum:goal1}-\ref{enum:goal3} stated at the end of Subsection \ref{subsec:introduction}.  To stabilize the inversion we need to incorporate some a priori information into the reconstruction which permits an edge-preserving reconstruction, or at least, does not smoothes edges in the reconstruction. We will do so by assuming that the functions we are going to reconstruct belong to a class $\Ec^2$ which consists of functions that are $C^2$ except from discontinuities along $C^2$ curves. To translate this qualitative information into a mathematical language we use the fact that functions in $\Ec^2$ are optimally sparse with respect to the curvelet frame \cite{Candes_New_tight_Frames_of_Curvelets04}. Hence, the technique of sparse regularization  \cite{Daubechies_Iterative_Thresholding_Algorithm_for_Inverse_Problems04} seems to be appropriate in our setting. 

To this end, we briefly recall the definition the curvelet frame \cite{Candes_CCTII} and collect some basic facts about technique of sparse regularization. 

\subsection{The curvelet dictionary}
\label{ssec:curvelet frame}
At scale $2^{-j}$, $j\in\mathbb{N}_0$, we first define the generating curvelets $\psi_{j,0,0}$, in the frequency domain using polar coordinates $(r,\omega)$ by
\begin{equation}\label{eq:generating curvelets}
	\widehat\psi_{j,0,0}(r,\omega)=2^{-3j/4}\cdot W(2^{-j}\cdot r)\cdot V\paren{\frac{2^{\ceil{j/2}+1}}{\pi}\cdot\omega},
\end{equation}
where $W(r)$ is a radial window and $V(\omega)$ denotes an angular window. The windows $W$ and $V$ are both real and smooth, i.e., $W,V\in C^\infty$. Furthermore, we assume that $\supp W \subset (1/2,2)$, $\supp V\subset (-1,1)$ and that the following admissibility conditions are satisfied,
\begin{align*}
	\sum_{j=-\infty}^\infty W^2(2^j r) &= 1,\quad r\in(3/4,3/2);\\
	\sum_{l=-\infty}^\infty V^2(\omega-l) &= 1,\quad \omega\in(-1/2,1/2).
\end{align*}

\begin{figure}[h]
\centering
	\includegraphics[width=\textwidth]{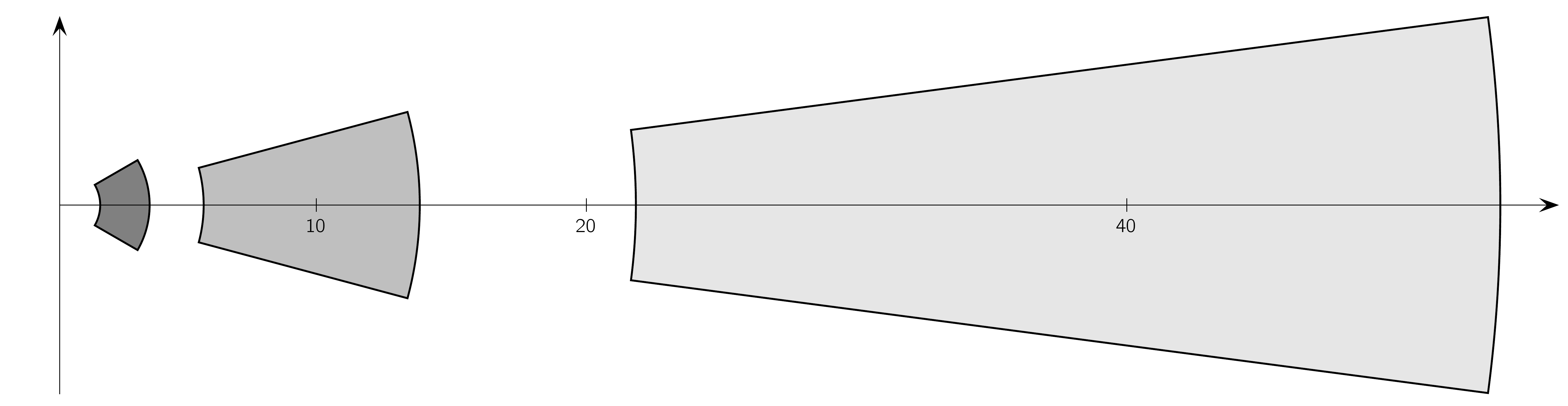}
	\caption{Support of curvelets in the Fourier domain for $j=1$ (dark gray), $j=3$ (gray) and $j=5$ (light gray).}
	\label{fig:curveletsupport}
\end{figure}

The family of curvelets $\sparen{\psi_{j,l,k}}_{j,l,k}$ is now constructed by translation and rotation of generating curvelets $\psi_{j,0,0}$. That is, at scale $2^{-j}$, the curvelet $\psi_{j,l,k}$ is defined via
\begin{equation}\label{eq:curvelets}
	\psi_{j,l,k}(x) = \psi_{j,0,0}(R_{\theta_{j,l}}(x-b^{j,l}_k)),
\end{equation}
where $R_{\theta_{j,l}}$ denotes the rotation matrix (cf. Section \ref{subsec:notation}) with respect to the scale-dependent rotation angles $\theta_{j,l}$ and scale-dependent locations $b^{j,l}_k$ which are define by
\begin{align*}
	\theta_{j,l} &= l\cdot \pi\cdot 2^{-\ceil{j/2}-1}, \quad -2^{\ceil{j/2}+1}\leq l<2^{\ceil{j/2}+1},\\
	b^{j,l}_k &=R_{\theta_{j,l}}^{-1}\paren{\frac{k_1}{2^j},\frac{k_2}{2^{j/2}}}, \quad k=(k_1,k_2)\in\mathbb{Z}^2.
\end{align*}

Since the window functions $W$ and $V$ are compactly supported, and in particular, since the support of $W(2^j\bdot)$ is contained in $(1/2,\infty)$, it follows from \eqref{eq:generating curvelets} and \eqref{eq:curvelets} that, in the Fourier domain, each curvelet is supported on a polar wedge which has a positive distance to the origin, see Figure \ref{fig:curveletsupport}. We have $\hat\psi_{j,l,k}(\xi)=0$ for all $\abs{\xi}<1/2$ and for all admissible indices $(j,l,k)$, i.e., the region $\bigcup_{(j,l,k)}\supp\hat\psi_{j,l,k}$ covers not all of the $\mathbb{R}^2$. Thus, the system $\sparen{\psi_{j,l,k}}$ does not contain any low-pass element. 

To complete the definition of the curvelet system we define the generating low-pass function $\psi_{-1,0,0}$ in the Fourier domain by
\begin{equation*}
	\widehat\psi_{-1,0,0}(r,\omega)= W_0(r),\quad W_0^2(r):=1-\sum_{j=0}^\infty W^2(2^{-j}r)
\end{equation*}
and complete the curvelet system by all of its translates $\sparen{\psi_{-1,0,k}}_{k\in\mathbb{Z}^2}$. 

\begin{remark*}
	In the spatial domain, the essential support of curvelets is an ellipse which is located near $b^{j,l}_k$ and oriented along the orthogonal direction $\theta_{j,l}^\bot=\theta_{j,l}+\pi/2$. The directional localization becomes higher when the scale parameter $j$ increases. Thus, curvelets are highly oriented at fine scales, see Figure \ref{fig:curvelets}.
\end{remark*}

\begin{figure}[H]
\centering
	\subfloat[$\psi_{500}$]{\includegraphics[height=4.5cm]{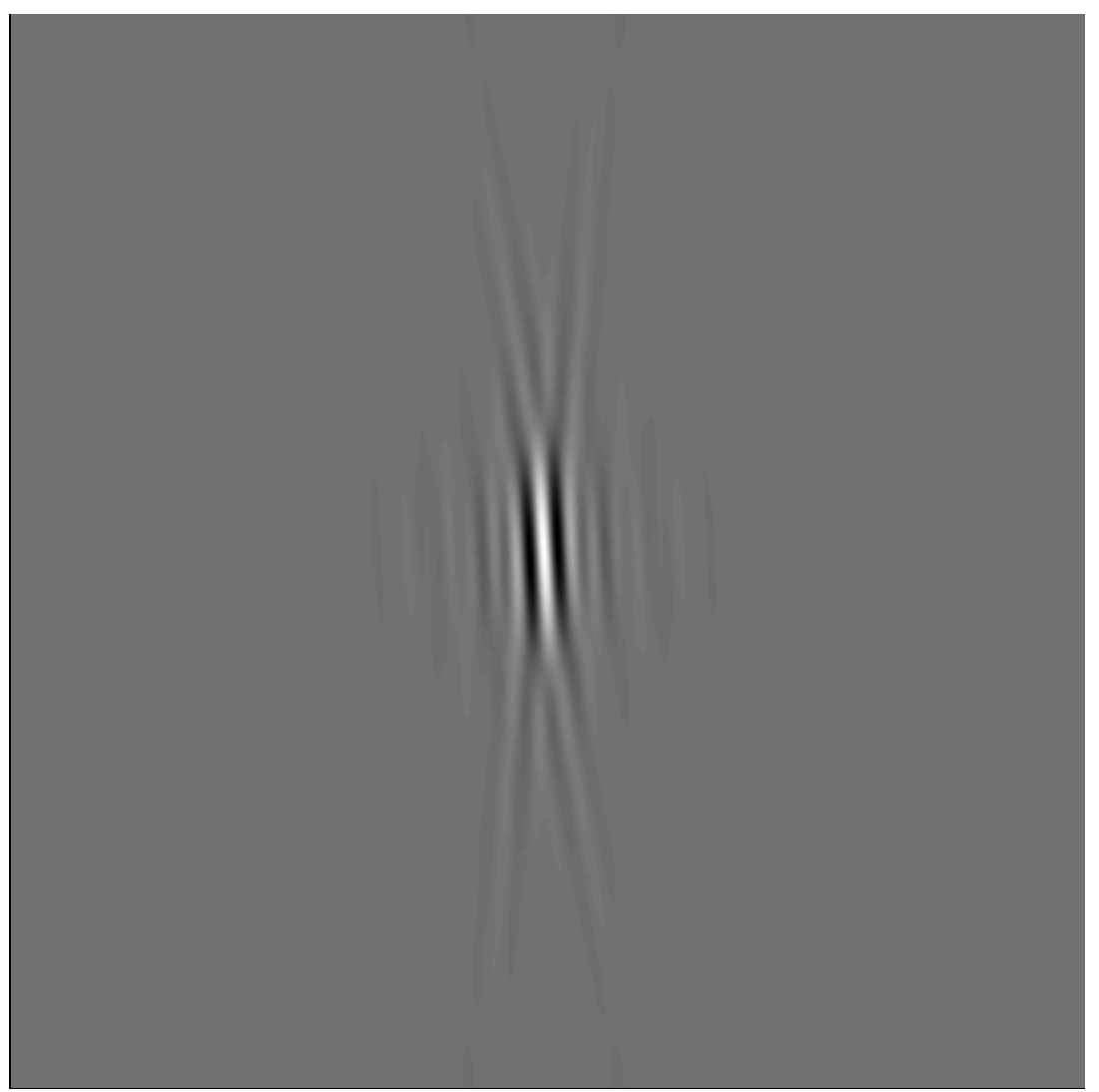}}\quad
	\subfloat[$\psi_{650}$]{\includegraphics[height=4.5cm]{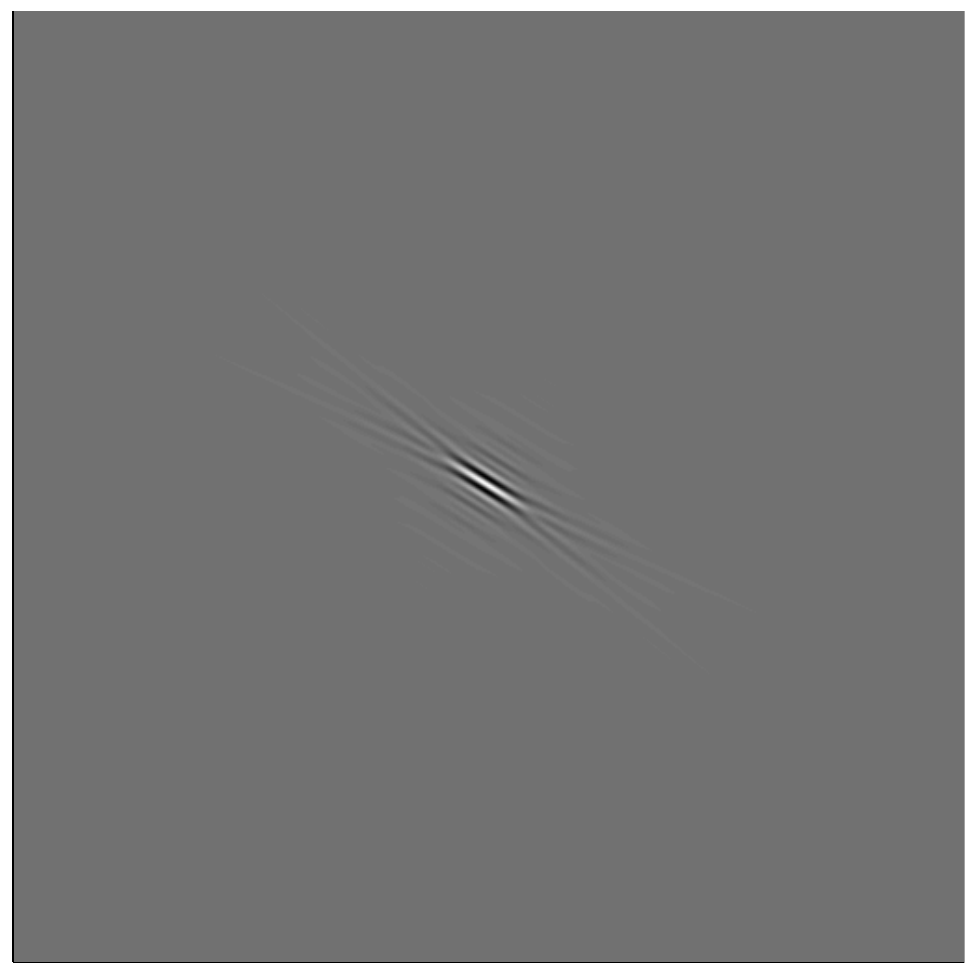}}
	\caption{Curvelets at different scales and different orientations. Left image shows a curvelet with orientation $\theta_{5,0}=0^\circ$ whereas the right image shows a curvelet with orientation $\theta_{6,5}=56.25^\circ$.}
	\label{fig:curvelets}
\end{figure}

The index set of the completed curvelet system is now given by
\begin{align}
	\mathcal{I} =\sparen{(-1,0,k):\,k\in\mathbb{Z}^2} \cup \sparen{(j,l,k):\,j\in\mathbb{N}_0,\,k\in\mathbb{Z}^2,\,-2^{\ceil{j/2}+1}\leq l<2^{\ceil{j/2}+1}}
	\eqqcolon \Ic_0 \; \cup \; \Ic_1.	
\end{align}
Note that each index $(j,l,k)\in \Ic$ has a 3 parameter structure, where $j$ denotes the \emph{scale}-parameter, $k=(k_1,k_2)$ is the \emph{location} parameter and $l$ is the \emph{orientation} parameter. The system $\sparen{\psi_{j,l,k}}_{(j,l,k)\in\mathcal{I}}$ is now complete in the sense that it constitutes a tight frame for $L^2(\R^2)$, \cite{Candes_CCTII}. For each $f\in L^2(\mathbb{R}^2)$ there is a curvelet representation
\begin{equation}
\label{eq:curvelet decomposition}
	f = \sum_{(j,l,k)\in\mathcal{I}}\h{\psi_{j,l,k},f}\psi_{j,l,k}
\end{equation}
and a Parseval relation holds,
\begin{equation*}
	\norm{f}_{L^2(\mathbb{R}^2)}^2 = \sum_{(j,l,k)\in\mathcal{I}}\abs{\h{\psi_{j,l,k},f}}^2.
\end{equation*}

We conclude this section by noting that curvelets can be understood as further development of wavelets \cite{Chui_IntroductionToWavelets}. Therefore, as it is well known from the theory of wavelets, the curvelet dictionary provides a sparse representation of a large class of functions. This property qualifies curvelets for the use within the framework of the sparse regularization. An even more important property of curvelets lies in the fact that they offer an optimally sparse representation of functions that are $C^2$ except form discontinuities along $C^2$ curves \cite{Candes_New_tight_Frames_of_Curvelets04}. This means, that curvelets encode edges in a very efficient way. Thus, sparse representations of functions with respect to the curvelet dictionary can be considered to be edge-preserving.

\subsection{Curvelet sparse regularization (CSR)}
\label{ssec:sparse regularization}
Our aim is to solve the limited angle reconstruction problem \eqref{eq:reconstruction problem} in a stable way such that the edges are preserved. To this end, we use sparse regularization of curvelet coefficients. The idea of sparse regularization is to determine a solution $f$ of the problem which is sparse or compressible with respect to the curvelet frame. Sparsity of $f$ means that the series expansion \eqref{eq:curvelet decomposition} of $f$ has only a very small number of curvelet coefficients $c_k$ which are non-zero, whereas compressibility of $f$ means that $f$ can be well approximated using a sparse series expansion. 

To this end, we have to formulate the reconstruction problem in the curvelet domain, i.e., we are interested in the recovery of the curvelet coefficients $c_{j,l,k}=\h{\psi_{j,l,k},f}$ of $f$ instead of the function $f$ itself. We assume that the unknown object $f$ can be represented by a finite linear combinations of curvelets, i.e., $f = \sum_{n=1}^N\h{\psi_n,f}\psi_n$ with $n=n(j,l,k)$. Further, we define the analysis operator $T$ and the synthesis operator $T^\ast$ as
\begin{equation*}
	Tf = \sparen{\h{\psi_n,f}}_{n=1}^N,\quad T^\ast c = \sum_{n=1}^N c_n\psi_n.
\end{equation*}
Then, the reconstruction problem \eqref{eq:reconstruction problem} can expressed in terms of curvelet coefficients via
\begin{equation}\label{eq:noisy tomography problem in curvelet domain}
	y^\delta = K c + \eta,\quad K:=\Rc_\Phi T^\ast.
\end{equation}

A solution to \eqref{eq:noisy tomography problem in curvelet domain} by sparse regularization of curvelet coefficients is given as a minimizer of the $\ell^1$-penalized Tikhonov type functional, i.e.,
\begin{equation}\label{eq:sparse regularization}
    \hat c=\argmin_{c\in\R^N}\sparen{\frac 1{2}\norm{Kc-y^\delta}^2_{L^2(S^1\times\mathbb{R})}+\norm{c}_{1,w}},
\end{equation}
where $\norm{c}_{1,w}=\sum_k w_k\abs{c_k}$ denotes the weighted $1$-norm with a weight sequence $w$ satisfying $w_k\geq w_0>0$.  A reconstruction for the original problem \eqref{eq:reconstruction problem} is then given by applying the synthesis operator to the regularized curvelet coefficients $\hat c$, i.e.,
\begin{equation}
\label{eq:curvelet sparse regularization reconstruction}
 	\hat f = \sum_{n=1}^N \hat c_n\psi_n.
\end{equation}
We note that sparse regularization is indeed a regularization method \cite{Daubechies_Iterative_Thresholding_Algorithm_for_Inverse_Problems04},\cite[Sec. 3.3]{VariationalMethodsScherzer2009}. Therefore, the computation of a reconstruction by \eqref{eq:sparse regularization} and \eqref{eq:curvelet sparse regularization reconstruction} is stable and favors sparse solutions \cite{Daubechies_Iterative_Thresholding_Algorithm_for_Inverse_Problems04,Lorenz_Semismooth_Newton_for_SparseTikhonov08}. We will refer to this method by the term \emph{curvelet sparse regularization} or CSR, respectively. In the previous subsection, we have discussed that sparse representation of functions with respect to the curvelet dictionary are edge-preserving. Consequently, curvelet sparse regularization gives rise to an edge-preserving reconstruction method.

In the following proposition we give a general characterization of minimizers of the $\ell^1$-penalized Tikhonov functional. Though the proof can be found in \cite{Lorenz_Semismooth_Newton_for_SparseTikhonov08}, we will recall it here for the sake of completeness. To this end, we define the so-called \emph{soft-thresholding operator} $\mathcal{S}_w:\R^N\to\R^N$  by 
\begin{equation}
\label{eq:soft-thresholding}
	(\mathcal{S}_w(x))_k=S_{w_k}(x_k):=\max\sparen{0,\abs{x_k}-w_k}\sgn(x_k).
\end{equation}

\begin{proposition}
\label{prop:sparse regularization solution}
	The set of minimizers of the $\ell^1$-penalized Tikhonov functional 
	\begin{equation*}\label{eq:ell1 penalized Tikhonov functional}
		\Psi(c)=\frac 1{2}\norm{Kc-y^\delta}^2_{L^2(S^1\times\R)}+\norm{c}_{1,w},
	\end{equation*}
	is non-empty. Furthermore, each minimizer $\hat c$ of $\Psi$ is characterized by
	\begin{equation}\label{eq:characterization of sparse solution}
		\hat{c}=\Sc_{\gamma w}\paren{\hat{c}-\gamma K^\ast(K\hat{c}-y^\delta)}
	\end{equation}
	for any $\gamma>0$.
\end{proposition}
\begin{proof}
	We follow the proof of \cite{Lorenz_Semismooth_Newton_for_SparseTikhonov08}. Since $\Psi$ is convex and coercive it follows that there is a minimizer $\hat c$ of $\Psi$. We denote by $\partial f(c)$ the subdifferential of a convex function $f$ at $x$. Each minimizer $\hat c$ is characterized by the requirement (cf. \cite{Rockafellar1970})
	\begin{equation*}
		0\in\partial \Psi(\hat c) = K^\ast(K\hat c-g) + \partial \norm{c}_{1,w}
	\end{equation*}
	which is equivalent to
	\begin{equation*}
		-K^\ast(K\hat c-g)\in\partial \norm{\hat c}_{1,w}.
	\end{equation*}
	Multiplying by $\gamma>0$ and adding $\hat c$ to both sides yields
	\begin{equation*}
		\hat c -\gamma K^\ast(K\hat c-g)\in\hat c + \gamma\partial \norm{\hat c}_{1,w}=(\id + \gamma\partial \norm{\cdot}_{1,w})\hat c.
	\end{equation*}
	Following the arguments in \cite{Lorenz_Semismooth_Newton_for_SparseTikhonov08} we get that $(\id + \gamma\partial\norm{\cdot}_{1,w})^{-1}$ exists and is single valued. Hence, the above inclusion is characterized by the equation
	\begin{equation*}
		\hat c = (\id + \gamma\partial\norm{\cdot}_{1,w})^{-1}(\hat c -\gamma K^\ast(K\hat c-g)).
	\end{equation*}
	A simple calculation shows that $(\id + \gamma\partial\norm{\cdot}_{1,w})^{-1}=\Sc_{\gamma w}$.
\end{proof}

\section{Relation to biorthogonal curvelet decomposition (BCD) for  the Radon transform}
\label{sec:BCD reconstruction}
In this section we will show that if the data is available from the full angular range, i.e., if we are dealing with the Radon transform $\Rc$ rather than the limited angle Radon transform $\Rc_\Phi$, an explicit formula for the minimizer of the $\ell^1$-penalized Tikhonov functional \eqref{eq:sparse regularization} can be derived using the biorthogonal curvelet decomposition (BCD) for the Radon transform \cite{Candes_Recovering_Edges_in_Illposed_problems02}. This formula is closely related to the BCD based reconstruction which was also proposed in \cite{Candes_Recovering_Edges_in_Illposed_problems02}. Afterwards we will discuss that the curvelet sparse regularization can be understood as a natural generalization of the BCD reconstruction.

\subsection{Full angular range}
We now briefly recall the definition of the BCD for the Radon transform. For details we refer to  \cite{Candes_Recovering_Edges_in_Illposed_problems02}. If not otherwise stated, we let $n=n(j,l,k)\in\Ic$ and denote the curvelet frame by $\sparen{\psi_n}$. In order to derive the BCD for the Radon transform a pair of frames $\sparen{U_n}$ and $\sparen{V_n}$ is constructed for $\ran{\Rc}\subset L^2(\R\times S^1)$ such that 
\begin{equation*}
	\Rc\psi_n=2^{-j}V_n, \quad \Rc^\ast U_n=2^{-j}\psi_n
\end{equation*}
and a quasi-biorthogonal relation 
$\h{V_n,U_{n'}}_{L^2(\R\times S^1)}=2^{j-j'}\h{\psi_n,\psi_{n'}}$, 
holds for all $n,n'\in\Ic$. In particular, there is an $L^2$-norm equivalence property
\begin{equation*}
	\sum_{n\in\Ic}\abs{\h{g,U_n}_{L^2(\R\times S^1)}}^2 \asymp\norm{g}_{L^2(\R\times S^1)}^2,
\end{equation*}
for all $g\in\ran{\Rc}$. Similar relations hold for $\sparen{V_n}$. Using these notations, the BCD of the Radon transform is given by the following reproducing formula
\begin{equation}\label{eq:BCD reproducing formula}
	f = \sum_{n\in\Ic} 2^j \h{\Rc f,U_n}_{L^2(\R\times S^1)}\psi_n,
\end{equation}
where $f\in L^2(\R^2)$ is assumed to be a finite sum of curvelets $\sparen{\psi_n}$ \cite{Candes_Recovering_Edges_in_Illposed_problems02}. Note that the curvelet coefficients of $f$ are computed from the Radon transform data $\Rc f$. 

We now use the above the frames $(U_n)$, $(V_n)$ and its relations with $\Rc$ and $\Rc^\ast$, respectively, to compute the minimizer of the $\ell^1$-penalized Tikhonov functional. We assume that $y^\delta\in\ran{\Rc}$ and denote $\h{\cdot,\cdot}=\h{\cdot,\cdot}_{L^2(\R^2)}$ and $[\cdot,\cdot] = \h{\cdot,\cdot}_{L^2(\R\times S^1)}$. With $y_n^\delta=[y^\delta,U_n]$, we get
\begin{align*}
	\norm{\Rc f - y^\delta}_{L^2(\R\times S^1)}^2 &\asymp \sum_{n\in\Ic} \abs{[\Rc f - y^\delta, U_n]}^2\\
	&= \sum_{n\in\Ic} \abs{[\Rc f, U_n] - [y^\delta,U_n]}^2\\
	&= \sum_{n\in\Ic} \abs{\h{f, \Rc^\ast U_n} - [y^\delta,U_n]}^2\\
	&= \sum_{n\in\Ic} \abs{\h{f, 2^{-j}\psi_n} - [y^\delta,U_n]}^2\\
	&= \sum_{n\in\Ic} \abs{2^{-j}c_n - y^\delta_n}^2.
\end{align*}
Using the definition of $\norm{\cdot}_{1,w}$ we see that
\begin{equation*}
	\norm{K c - y^\delta}_{L^2(\R\times S^1)}^2+\norm{c}_{1,w}\asymp \sum_{n\in\Ic} \paren{\abs{y^\delta_n-2^{-j}c_n}^2 + \alpha w_n\abs{c_n}},
\end{equation*}
with a suitably chosen constant $\alpha>0$. Hence, we have
\begin{align}
	\notag
	\hat c &=\argmin_{c\in\R^N}\norm{K c - y^\delta}_{L^2(\R\times S^1)}^2+\norm{c}_{1,w}\\
	\label{eq:sparse regularization via BCD}
	&= \argmin_{c\in\R^N} \sum_{n\in\Ic} \paren{\abs{2^{-j}c_n - y^\delta_n}^2 + \alpha w_n\abs{c_n}},
\end{align}
which can be minimized by minimizing each term in \eqref{eq:sparse regularization via BCD} separately. Note that each term in \eqref{eq:sparse regularization via BCD} is of the form $\abs{ax-b}^2+c\abs{x}$ and its minimum is given by $S_{c/(2a^2)}(b/a)$ where $S_{c/(2a^2)}$ is the soft-thresholding function from \eqref{eq:soft-thresholding} with the threshold $c/(2a^2)$. Therefore, the sparsity regularized curvelet coefficients $\hat c$ are given by 
\begin{equation*}
	\hat c_n= S_{2^{2j-1}\alpha w_n}(2^j y_n^\delta).
\end{equation*}
We have now proven the following Theorem.

\begin{theorem}
The solution of the full angular problem $y^\delta=\Rc f + \eta$ via curvelet sparse regularization is given by the (closed) formula
\begin{equation}\label{eq:sparse minimizer via BCD}
	\hat f=\sum_{n\in\Ic} S_{2^{2j-1}\alpha w_n}(2^j y_n^\delta)\psi_n.
\end{equation}
\end{theorem}

The relation between the reproducing formula \eqref{eq:BCD reproducing formula} and \eqref{eq:sparse minimizer via BCD} is now obvious. If the thresholding parameters $w_n=w_n(\delta)$ in \eqref{eq:sparse minimizer via BCD} are chosen such that $w_n(\delta)\to 0$ as $\delta\to 0$, i.e., they vanish if there is no noise present in the data, then \eqref{eq:sparse minimizer via BCD} reduces to \eqref{eq:BCD reproducing formula}. On the other hand, if the data is corrupted by noise, then, the curvelet sparse regularized solution is simply a thresholded version of the BCD reproducing formula. The stabilizing character of the curvelet sparse regularization is reflected by the inherent thresholding of the curvelet coefficients (see also \eqref{eq:characterization of sparse solution}).

In \cite{Candes_Recovering_Edges_in_Illposed_problems02} a very similar reconstruction rule was derived. Starting form the BCD reproducing formula the authors proposed to use soft-thresholding of  coefficients in \eqref{eq:BCD reproducing formula}, i.e.,
\begin{equation}
	\label{eq:BCD inversion of Radon data}
	\hat{f} = \sum_{n\in\Ic} S_{\tau_j}\paren{2^j\h{y^\delta,U_n}_{L^2(\R\times S^1)}}\psi_n,
\end{equation}
with a scale dependent threshold $\tau_j$. We see that this formula coincides with \eqref{eq:sparse minimizer via BCD} for a suitably chosen thresholding sequence $\tau=(\tau_j)$. 

\begin{remark}
	Note the ill-posed nature of the reproducing formula \eqref{eq:BCD reproducing formula}. This is evident because the coefficients $2^j\h{\Rc f,U_\mu}_{L^2(\R\times S^1)}$ corresponding to fine scales (large $j$) are amplified by the factor $2^j$. Since noise is a fine scale phenomenon, there will be very large reconstruction errors when the data is corrupted by noise. 
\end{remark}

\subsection{Limited angular range \& Limitations of the BCD}
\label{subsec:limitations of BCD reconstruction}
We have seen that there is an explicit expression \eqref{eq:sparse minimizer via BCD} for the CSR reconstruction in the case of full angular tomography.  We also noted that the thresholded BCD reconstruction \eqref{eq:BCD inversion of Radon data} leads (under certain conditions) to the same reconstruction. To extend this observation to the limited angle tomography a biorthogonal curvelet decomposition for the limited angle Radon transform would be needed.  To our knowledge there is no such BCD available for the limited angle Radon transform.  Consequently, in the case of limited angle tomography, the CSR reconstruction \eqref{eq:curvelet sparse regularization reconstruction} can not be expressed explicitly as it was done for the full angular range in \eqref{eq:sparse minimizer via BCD} and the BCD reconstruction of Cand{\`e}s and Donohod \cite{Candes_Recovering_Edges_in_Illposed_problems02} can not be applied in this situation. 

In contrast to the BCD method, a reconstruction of the limited angle problem can be computed using CSR. Hence, curvelet sparse regularization can be understood as the natural generalization of the thresholded BCD reconstruction.

Curvelet sparse regularization offers even more flexibility compared to the BCD method. For example, the implementation of the thresholded BCD method is difficult for acquisition geometries which are different from the parallel geometry. This is because the BCD method requires discretization of the functions $U_n$ which live in the Radon domain. The implementation of the curvelet sparse regularization approach, however, is independent of the acquisition geometry. One needs only to implement the system matrix. Moreover, the generalization to higher dimensions is also easier accessible via curvelet sparse regularization approach.

\section{Characterization of limited angle Radon transform}
\label{sec:characterizations}
In Section \ref{sec:stabilization} we presented curvelet sparse regularization as our method of choice for the limited angle tomography because it is stable, edge-preserving and flexible. In Proposition \ref{prop:sparse regularization solution} we noted the existence of a solution and showed that each minimizer of the $\ell^1$-penalized Tikhonov functional \eqref{eq:sparse regularization} is given as a fixed point of some operator (cf. \eqref{eq:characterization of sparse solution}).  This characterization is generic in the sense that it does not take into account the special structure of the underlying problem. 

The goal of this section is to give a characterization of the minimizer \eqref{eq:sparse regularization} which is adapted to the setting of limited angle geometry. In the following we will show that, depending on the available angular range, a big portion of the curvelet coeffiients of the CSR reconstruction are zero. 

We state our main results first and postpone the proofs to the end of this section.
\begin{theorem}
\label{thm:null space}
	Let $0<\Phi<\pi/2$. We define the polar wedge $W_{\Phi}$ by
	\begin{equation}
	\label{eq:scale dependent polar wedge}
		W_\Phi = \sparen{\xi\in\R^2:\,\xi=r(\cos\omega,\sin\omega),\,r\in\R,\abs{\omega}\leq\Phi}.
	\end{equation}
	Moreover, we define a proper subset of the curvelet index set by
	\begin{equation}
	\label{eq:invisible curvelet indices}
		\Ic_\Phi^{\mathrm{invisible}}=\sparen{(j,l,k)\in\Ic:\,\supp\hat\psi_{j,l,k}\cap W_\Phi=\emptyset},
	\end{equation}
	where $\psi_{j,l,k}$ denotes a curvelet and $\Ic$ is the curvelet index set (cf. Subsection \ref{ssec:curvelet frame}). Then, 
	\begin{equation}
	\label{eq:null space}
		\Rc_\Phi\psi_{j,l,k}\equiv 0 \text{ for all } (j,l,k)\in\Ic_\Phi^{\mathrm{invisible}}.
	\end{equation}
\end{theorem}

The above theorem characterizes a subspace of the kernel of the limited angle Radon transform in terms of curvelets. Using this Theorem \ref{thm:null space}, a characterization of curvelet sparse regularized solutions to the limited angle problem can be derived.

\begin{theorem}
\label{thm:characterization of sparse regularized curvelet coeffs}
	Let $0<\Phi<\pi/2$, $y^\delta\in\ran{\Rc_\Phi}$ and let $\Ic_\Phi^{\mathrm{invisible}}$ be defined by \eqref{eq:invisible curvelet indices}. Then,
	\begin{equation*}
		\hat c=\argmin_{c\in\R^N}\sparen{\frac 1{2}\norm{Kc-y^\delta}^2_{L^2(S^1\times\mathbb{R})}+\norm{c}_{1,w}}
	\end{equation*}
	satisfies
	\begin{equation*}
	\hat c_{j,l,k} = 0 \text{ for all }(j,k,l)\in\Ic_\Phi^{\mathrm{invisible}}.
	\end{equation*}
\end{theorem}

We start to develop the proof of Theorem \ref{thm:null space} first. To this end, we need some auxiliary results. Though the content of the following lemma is classical we will give a proof for the sake of completeness.

\begin{lemma}
\label{lem:delta approximation}
	Let $b>0$ and $\delta^b$ be a function defined by
	\begin{equation}
	\label{eq:approximate delta function}
		\delta^b(x)=\frac{1}{2\pi}\int_{-b}^b e^{\pm ix\cdot\xi}\d \xi.
	\end{equation}
	 Then, it holds that $\delta^b\to\delta$ pointwise in $\Sc'(\R)$ as $b\to\infty$, i.e., for $\phi\in\Sc(\R)$ we have
	\begin{equation}\label{eq:delta approximation}
		\lim_{b\to\infty}\int_{\R}\delta^b(x)\phi(x)\d x=\phi(0).
	\end{equation}
\end{lemma}
\begin{proof}
	A simple calculation shows that
	\begin{equation*}
		\delta^b(x) = \frac{1}{\pi}\frac{\sin(bx)}{x}.
	\end{equation*}
	Let $\phi\in\Sc(\R)$. We split the integral in \eqref{eq:delta approximation}
	\begin{align*}
		\frac{1}{\pi} \int_{-\infty}^\infty \frac{\sin(bx)}{x}\phi(x)\d x 
		&= \frac{1}{\pi} \int_{-\epsilon}^\epsilon \frac{\sin(bx)}{x}\phi(x)\d 
		x +  \frac{1}{\pi} \int_{\abs{x}>\epsilon }\frac{\sin(bx)}{x}\phi(x)\d 
		x\\[2ex]
		&=: \; I_1(b)\quad + \quad I_2(b).
	\end{align*}
	
	Observe that $g(x)=\chi_{\sparen{\abs{x}>\epsilon}}(x)\phi(x)/x$ is in $L^1(\R)$ and it follows from the Riemann-Lebesgue lemma \cite[Theorem 1.1 + 1.2]{SteinWeiss1971} that
	\begin{equation*}\label{aux:delta approximation}
		\lim_{b\to\infty} I_2(b) =\lim_{b\to\infty}\frac{1}{\pi}\int_\R\sin(bx)g(x)\d x = \sqrt{\frac{2}{\pi}}\lim_{b\to\infty}\Im(\,\hat g(b))=0.
	\end{equation*}

	To compute $\lim_{b\to\infty} I_1(b)$ we use the Taylor expansion of $\phi$ at $x=0$, $\phi(x)=\phi(0)+\phi\,'(\xi_x)x$ with $x\in(-\epsilon,\epsilon)$ and $\xi_x\in[-x,x]$. Note that $\phi'$ is bounded and, hence, $\phi'(\xi_x)$ is integrable on $(-\epsilon,\epsilon)$.
	\begin{align*}
		\lim_{b\to\infty} I_2(b) &=  \lim_{b\to\infty} \int_{-\epsilon}^\epsilon\delta^b(x)\phi(x)\d x \\
		&=\lim_{b\to\infty} \int_0^\epsilon\delta^b(x)(\phi(x)+\phi(-x))\d x \\\
		&=\frac{2\phi(0)}{\pi} \lim_{b\to\infty} \int_0^{\epsilon}\frac{\sin(bx)}{x}\d x + \frac{1}{\pi} \lim_{b\to\infty} \int_0^{\epsilon}\sin(bx)(\phi'(\xi_x)+\phi'(-\xi_x))\d x\\
		&=\frac{2\phi(0)}{\pi} \lim_{b\to\infty} \int_0^{b\cdot\epsilon}\frac{\sin(x)}{x}\d x\\
		&=\phi(0),
	\end{align*}
	where we have again applied the Riemann-Lebesgue lemma to the function $g(x)=\chi_{(0,\epsilon)}(x)(\phi'(\xi_x)+\phi'(-\xi_x))$ and have used the asymptotics for the sine integral, $\lim_{x\to\infty}\int_0^x\sin t/t\d t=\pi/2$.
\end{proof}

The key observation for the proof of Theorem \ref{thm:null space} is contained in the following lemma.

\begin{lemma}\label{lem:integration along lines}
	For $f\in\Sc(\R^2)$ and $\xi=(\cos\eta,\sin\eta)$, $\xi^\bot=(-\sin\eta,\cos\eta)$, we have
	\begin{equation}\label{eq:integration along lines}
		\int_\R f(t\cdot\xi)\d t = \int_\R \hat f(t\cdot\xi^\bot)\d t.
	\end{equation}
	That is, integration in the spatial domain along a line through the origin corresponds to the integration along a perpendicular line through the origin in the frequency domain.
\end{lemma}
\begin{proof}
	We first show that
	\begin{equation}\label{aux:integration along lines}
		\int_\R f(x_1,0)\d x_1 = \int_\R \hat f(0,\xi_2)\d \xi_2.
	\end{equation}
	Using the notation $\delta^b(x)=(2\pi)^{-1}\int_{-b}^b e^{ix\xi}\d\xi$ we compute
	\begin{align*}
		\int_\R f(x_1,0)\d x_1 &= \frac{1}{2\pi}\int_\R \int_{\R^2} \hat f(\xi)e^{ix_1\xi_1}\d\xi\d x_1\\
		&=  \frac{1}{2\pi}\lim_{b\to\infty}\int_{-b}^b \int_{\R^2} \hat f(\xi)e^{ix_1\xi_1}\d\xi\d x_1\\
		&= \frac{1}{2\pi}\lim_{b\to\infty}\int_{\R^2} \hat f(\xi) \paren{\int_{-b}^b  e^{ix_1\xi_1}\d x_1} \d\xi\\
		&= \lim_{b\to\infty}\int_{\R}\int_\R \hat f(\xi_1,\xi_2) \delta^b(\xi_1) \d\xi_1\d\xi_2\\
		&= \int_{\R}\paren{\lim_{b\to\infty}\int_\R \hat f(\xi_1,\xi_2) \delta^b(\xi_1) \d\xi_1}\d\xi_2\\
		&= \int_{\R}\hat f(0,\xi_2)\d\xi_2,
	\end{align*}
	where we have used Lemma \ref{lem:delta approximation}. In the above computation, the change of the integration order and the interchange of the limit process and integration are allowed due to Fubini's theorem and the dominated convergence theorem, respectively.
	
	Since Fourier transform and rotation commute, we get using notations in Subsection \ref{subsec:notation} and \eqref{aux:integration along lines} that
	\begin{align*}
		\int_\R f(t\cdot\xi)\d t &= \int_\R f\paren{R_{-\eta}\cdot\tvector{t}{0}}\d t\\
		&= \int_\R \rho_{-\eta} f(t,0)\d t\\
		&= \int_\R \widehat{\rho_{-\eta} f}(0,t)\d t\\
		&= \int_\R \hat f\paren{R_{-\eta}\cdot\tvector{0}{t}}\d t\\
		&= \int_\R \hat f(t\cdot\xi^\bot)\d t.
	\end{align*}
	{}
\end{proof}

Using Lemma \ref{lem:integration along lines} we are now able to derive a formula for the Radon transform of curvelets.
\begin{theorem}
\label{thm:radon transform of curvelets}
	Let $\psi_{j,l,k}$ be a curvelet (cf. \eqref{eq:curvelets}) and denote $\theta(\omega)=(\cos\omega,\sin\omega)^\intercal$. Then,  
	\begin{equation}
	 	\Rc\psi_{j,l,k}(\theta(\omega),s) = 2^{j/4} V\paren{\frac{2^{\ceil{j/2}+1}}{\pi}(\omega+\theta_{j,l})} \sqrt{2\pi}\,\widehat{W}\paren{2^j \h{ b^{j,l}_k,\theta(\omega+\theta_{j,l}) }-2^js},
	\end{equation}
	where $b^{j,l}_k$ and $\theta_{j,l}$ are defined in Section \ref{ssec:curvelet frame}.
\end{theorem}
\begin{proof}
	 Let $\theta:=\theta(\omega)=(\cos\omega,\sin\omega)^\intercal$, $\omega\in[-\Phi,\Phi]$, and $\tau_p f:=f(\bdot+p)$, $p\in\R^2$. First note that each curvelet $\psi_{j,l,k}$ is a Schwartz function since, per definition, its Fourier transform is $C^\infty$ and compactly supported.  Hence, we may apply Lemma \ref{lem:integration along lines}:
	\begin{equation}
	\label{aux:radon transform of curvelets1}
		\Rc_\Phi\psi_{j,l,k}(\theta,s) = \int_\R \psi_{j,l,k}(s\theta+t\theta^\bot)\d t = \int_\R (\tau_{s\theta}\psi_{j,l,k})(t\theta^\bot)\d t = \int_\R e^{i\h{s\theta,t\theta}} \hat \psi_{j,l,k}(t\theta)\d t.
	\end{equation}
	 
	 At scale $2^{-j}$, each curvelet $\psi_{j,l,k}$ is defined via translation and rotation of a generating curvelet $\psi_{j,0,0}$, cf. \eqref{eq:curvelets}. Using the relation of the Fourier transform and rotation as well as the relation $\widehat{\tau_p f}(\xi)=e^{i\h{p,\xi}}\hat f(\xi)$ we see that
	 \begin{equation}
	 \label{aux:radon transform of curvelets2}
		\hat\psi_{j,l,k}(\xi) = e^{-i\h{b^{j,l}_k,R_{\theta_{j,l}}\xi}}\hat\psi_{j,0,0}(R_{\theta_{j,l}}\xi).
	\end{equation}
	
	Now plugging \eqref{aux:radon transform of curvelets2} into  \eqref{aux:radon transform of curvelets1}, together with \eqref{eq:generating curvelets}, we deduce
	\begin{align*}
	\Rc_\Phi\psi_{j,l,k}(\theta,s) &= \int_\R e^{ist} e^{ -i \h{ b^{j,l}_k, R_{\theta_{j,l}}(t\theta) } } \hat \psi_{j,0,0}(R_{\theta_{j,l}}(t\theta))\d t,\\
	&=\int_\R e^{ist} e^{ -it \h{ b^{j,l}_k,\theta(\omega+\theta_{j,l}) } } \hat \psi_{j,0,0}(t,\omega+\theta_{j,l})\d t,\\
	&=2^{-3j/4} V\paren{\frac{2^{\ceil{j/2}+1}}{\pi}(\omega+\theta_{j,l})}\int_\R e^{ist} e^{ -it \h{ b^{j,l}_k,\theta(\omega+\theta_{j,l}) } }  W(2^{-j} t)\d t,\\
	&=2^{j/4} V\paren{\frac{2^{\ceil{j/2}+1}}{\pi}(\omega+\theta_{j,l})}\int_\R e^{ -ir [2^j \h{ b^{j,l}_k,\theta(\omega+\theta_{j,l}) } - s] }  W(r)\d r,\\
	&= 2^{j/4} V\paren{\frac{2^{\ceil{j/2}+1}}{\pi}(\omega+\theta_{j,l})} \sqrt{2\pi}\,\widehat{W}\paren{2^j \h{ b^{j,l}_k,\theta(\omega+\theta_{j,l}) }-2^js}.
\end{align*}

\end{proof}

The proof Theorem \ref{thm:null space} is now a simple consequence of Theorem \ref{thm:radon transform of curvelets}. 	
\begin{Proof}[of Theorem \ref{thm:null space}]
	Let $j\in\N$. To abbreviate the notation we let $a_j=\pi^{-1}2^{\ceil{j/2}+1}$ and denote by \[A_\Phi:=[-\pi,-\pi+\Phi]\cup[-\Phi,\Phi]\cup[\pi-\Phi,\pi]\] the symmetric (visible) angular range of the limited angle Radon transform $\Rc_\Phi$ (cf. Figure \ref{fig:symmetric angular range}). 
	
	According to Theorem \ref{thm:radon transform of curvelets}, we have to determine all $\theta_{j,l}$ such that $V(a_j(\bdot+\theta_{j,l}))|_{A_\Phi}\equiv 0$. Since $\supp V\subset (-1,1)$ and $\theta_{j,l}\in[-\pi,\pi]$, we have
	\begin{equation*}
		\theta_{j,l}\not\in (-a_j^{-1}-\omega,a_j^{-1}-\omega)\quad\Rightarrow\quad V(a_j(\omega+\theta_{j,l}))=0.
	\end{equation*}
	for all $\omega\in A_\Phi$. Therefore, by defining (cf. Figure \ref{fig:symmetric angular range})
	\begin{equation*}
	A_{\Phi,j}:= [-\pi,-\pi+(\Phi+a_j^{-1})]\cup[-(\Phi+a_j^{-1}),\Phi+a_j^{-1}] \cup[\pi-(\Phi+a_j^{-1}),\pi],
	\end{equation*}
	we see that $V(a_j(\bdot+\theta_{j,l}))|_{A_\Phi}\equiv 0$ holds whenever $\theta_{j,l}\not\in A_{\Phi,j}$. The assertion follows by defining the invisible index set of curvelet coefficients as
	\begin{equation*}
		\Ic_\Phi^\mathrm{invisible}=\sparen{(j,l,k)\in\Ic:\, \theta_{j,l}\not\in A_{\Phi,j}}.
	\end{equation*}

	\begin{figure}[h]
	\centering
		\includegraphics[width=11cm]{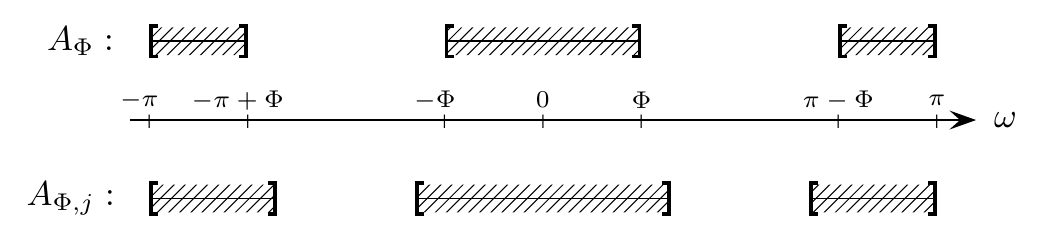}
		\caption{The symmetric (visible) angular range of the limited angle Radon transform, $A_\Phi$, and its scale-dependent version, $A_{\Phi,j}$. }
		\label{fig:symmetric angular range}
	\end{figure}	
\end{Proof}

We summarize the steps that were needed to prove Theorem \ref{thm:null space}. This procedure is also illustrated in Figure \ref{fig:idea of proof}.
To evaluate $\Rc_\Phi \psi_{j,l,k}(\xi,s)$ for a curvelet $\psi_{j,l,k}$, the integration was shifted to the Fourier domain according to Lemma \ref{lem:integration along lines}. This related the value $\Rc_\Phi \psi_{j,l,k}(\xi,s)$ to the integration of $\hat \psi_{j,l,k}$ along the line $L_\xi=\sparen{t\xi^\bot:\,t\in\R}$. Because of the limited angular range, the union of all such lines, 
\begin{equation}
\label{eq:angular wedge}
	W_\Phi=\bigcup_{\eta\in[-\Phi,\Phi]}L_{\xi(\eta)},
\end{equation}
covers not all of the $\R^2$. Above, we have again used the notation $\xi(\eta)=(\cos\eta,\sin\eta)$. To prove the assertion, we computed all those curvelet indices $(j,l,k)$ such that $\supp \hat\psi_{j,l,k}\cap W_\Phi =\emptyset$. 

\begin{figure}[h]
	\centering
		\includegraphics[width=11cm]{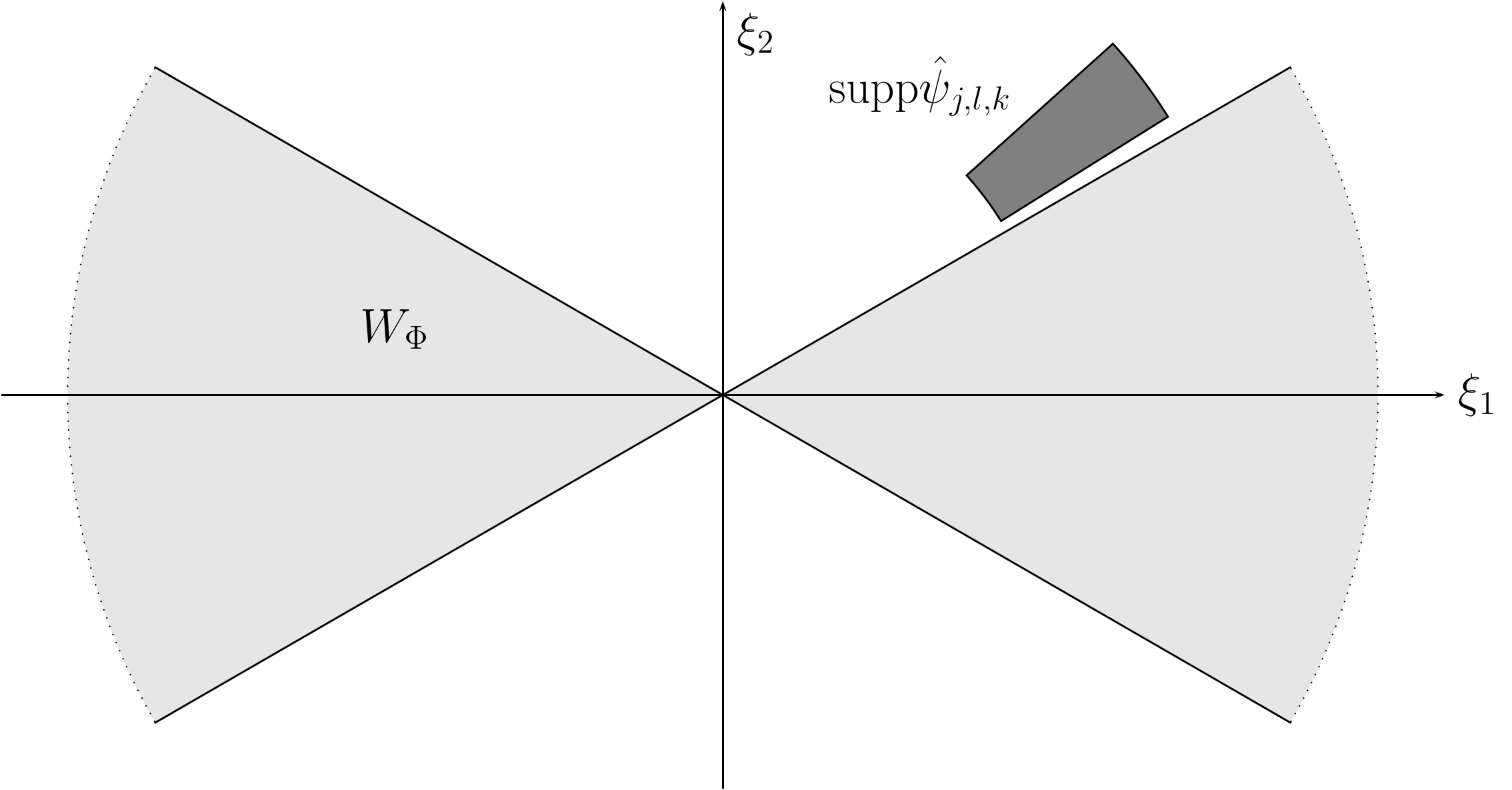}
		\caption{Theorem \ref{thm:null space} states that a curvelet $\psi_{j,l,k}$ lies in the kernel of the limited angle Radon transform whenever the the support of $\psi_{j,l,k}$ lies outside the \enquote{visibility region} $W_\Phi=\sparen{r(\cos\eta,\sin\eta):\, r\in\R, \eta\in[-\Phi,\Phi]}$. That is, $\Rc_\Phi\psi_{j,l,k}\equiv 0$ whenever $\supp \hat\psi_{j,l,k}\cap W_\Phi =\emptyset$.}
		\label{fig:idea of proof}
\end{figure}

Now, we turn the proof of Theorem \ref{thm:characterization of sparse regularized curvelet coeffs}. This will be a simple consequence of Theorem \ref{thm:null space} and the following lemma.

\begin{lemma}\label{lem:subdifferential of ell1 norm}
	Let $f:\ell^2(\Ic)\to [-\infty,\infty]$ be defined by $f(x)=\sum_{n\in\Ic}\phi(x_n)$, where $\phi:\R\to\R$ is a convex function such that $f$ is proper. Then, it holds that
	\begin{equation}
		y\in\partial f(x)\;\Leftrightarrow\;y_n\in\partial\phi(x_n)\text{ for all }n\in\Ic.
	\end{equation}
\end{lemma}
\begin{proof}
	First note that, since $f$ is proper, for $x\in\ell^2(\Ic)$ we have $\partial f(x)=\emptyset$ if $f(x)=\infty$. In what follows we therefore assume without loss of generality that $f(x)<\infty$.
	
	Suppose $y_n\in\partial \phi(x_n)$ for all $n\in\Ic$. Then, by definition of the subgradient we have
	\begin{equation*}
		\forall z_n\in\R:\; \phi(z_n)\geq \phi(x_n)+y_n(z_n-x_n).
	\end{equation*}
	Summing over $n$ implies
	\begin{equation*}
		\forall z\in\ell^2(\Ic):\; \sum_{n\in\Ic}\phi(z_n)\geq \sum_{n\in\Ic}\paren{\phi(x_n)+y_n(z_n-x_n)}=\sum_{n\in\Ic}\phi(x_n)+\h{y,z-x},
	\end{equation*}
	which is by definition of $f$ equivalent to $y\in\partial f(x)$. This proves the implication \enquote{$\Leftarrow$} of the statement.
	
	On the other hand, if $y\in\partial f(x)$, then $f(z)\geq f(x) + \h{y,z-x}$ for all $z\in\ell^2(\Ic)$. In particular, this holds for all $z=x+h e_n$ with $h\in\R$ and $n\in\Ic$, where $e_n=(\delta_{i,n})_{i\in\Ic}$ and $\delta_{i,n}$ denotes the Kronecker delta. Therefore we have
	\begin{align*}
		\forall h\in\R\,\forall n\in\Ic:\; &f(x+he_n)\geq f(x)+\h{y,h e_n}\\
		&\Leftrightarrow\; \sum_{i\in\Ic}\phi(x_i + \delta_{i,n}h) - \sum_{i\in\Ic}\phi(x_i) \geq y_n h\\
		&\Leftrightarrow\; \phi(x_n + h) - \phi(x_n) \geq y_n h\\
		&\Leftrightarrow\; y_n\in\partial\phi(x_n).
	\end{align*}
	{}
\end{proof}

\begin{Proof}[of Theorem \ref{thm:characterization of sparse regularized curvelet coeffs}]
	As in the proof of Proposition \ref{prop:sparse regularization solution} we see that $\hat c$ fulfills the following relation
	\begin{equation}
	\label{proof:aux1}
		-K^\ast(K\hat c - y^\delta)\in\partial\norm{\hat c}_{1,w}.
	\end{equation}
	
	Since $y^\delta\in\ran\Rc_\Phi$ we have that $y^\delta = Kc^\delta = \Rc_\Phi T^\ast c^\delta$ for some curvelet coefficient vector $c^\delta\in\R^N$. Thus,
	\begin{equation*}
		\hat x:=-K^\ast(K\hat c - y^\delta) = -K^\ast K (\hat c - c^\delta) = -K^\ast\paren{\sum_{j,l,k}(\hat c - c^\delta)_{j,l,k}\Rc_\Phi\psi_{j,l,k}}.
	\end{equation*}
	By Theorem \ref{thm:null space}, it holds that
	\begin{equation}
	\label{proof:aux2}
		\hat x_{j,l,k}=0 \text{ for all }(j,k,l)\in\Ic_\Phi^{\mathrm{invisible}}.
	\end{equation}
		
	Now let $(j,l,k)\in\Ic_\Phi^{\mathrm{invisible}}$. In view of \eqref{proof:aux1} and \eqref{proof:aux2}, Lemma \ref{lem:subdifferential of ell1 norm} implies that $0\in\partial\paren{w_{j,l,k}\abs{\hat c_{j,l,k}}}$. However, this means that $\hat c_{j,l,k}$ minimizes the function $w_{j,l,k}\abs{\,\cdot\,}$ and since $w_{j,l,k}>0$ we get that $\hat c_{j,l,k}=0$. {}
\end{Proof}

\section{Adapted curvelet sparse regularization}
\label{sec:adapted csr}
In this section we are going to apply the results from Section \ref{sec:characterizations} to a finite dimensional reconstruction problem. We will show that, in this setting, a significant dimensionality reduction can be performed. Based on this approach, we will formulate the adapted curvelet sparse regularization (A-CSR). 

\subsection{Discrete reconstruction problem}
We consider the discrete version of the reconstruction problem. To this end, we model $f$ as a finite linear combination of curvelets, i.e., $f=\sum_{n=1}^N c_n\psi_n$, where $n=n(j,l,k)$ is an enumeration of the curvelet index set $\Ic$ (cf. Subsection \ref{ssec:curvelet frame}). Moreover, we assume to be given a finite number of measurements $y_{m}=\Rc_\Phi(\theta_m,s_m)$, $1\leq m\leq M\in\N$.  Then, each measurement $y_m$ can be expressed as
\begin{equation}
\label{eq:discrete measurements}
 	y_m=\Rc_\Phi f(\theta_m,s_m)=\sum_{n=1}^N c_n\Rc_\Phi\psi_n(\theta_m,s_m).
\end{equation}
Now, let us define the so-called \emph{system matrix} $K$ by $K_{m,n}=\Rc_\Phi\psi_n(\theta_m,s_m)$ for $1\leq m\leq M$ and $n\in\Ic$. Then, the discrete version of the limited angle problem \eqref{eq:noisy tomography problem in curvelet domain} reads
\begin{equation}
\label{eq:discrete problem}
 	y = K c + \eta.
\end{equation}
Note the abuse of notation. In contrast to \eqref{eq:noisy tomography problem in curvelet domain}, where $K=\Rc_\Phi T^\ast$ denotes a continuous operator, here, $K$ is its discrete version. We want to point out, that the reconstruction problem \eqref{eq:discrete problem} is formulated in terms of all curvelet coefficients $c=T f$. That is, to solve \eqref{eq:discrete problem}, we need to compute $c_{n(j,l,k)}$ for all possible curvelet indices $(j,l,k)\in\Ic$. The dimension of the reconstruction problem, given by $N=\abs{\Ic}$, does not depend on the available angular range. In what follows, we will use the method curvelet sparse regularization to solve this problem.

\subsection{Dimensionality reduction \& Adapted curvelet sparse regularization (A-CSR)}
First, note that the results from Section \ref{sec:characterizations} are formulated only in terms of the angular range parameter $\Phi$. In turn, this parameter is completely determined by the acquisition geometry. Hence, it is known prior to the reconstruction and can be extracted from the given data by 
\begin{equation*}
	\Phi=\min\sparen{\phi:\exists \phi_0\in[-\pi,\pi]: \forall\,\,1\leq m\leq M:\, \theta_m\in[\phi_0-\phi,\phi_0+\phi]}. 
\end{equation*}
Knowing $\Phi$, we can use Theorem \ref{thm:characterization of sparse regularized curvelet coeffs} to identify those curvelets which lie in the kernel of the limited angle Radon transform $\Rc_\Phi$. Their index set can be precomputed according to \eqref{eq:invisible curvelet indices} or, equivalently, by
\begin{equation*}
	\Ic_\Phi^{\mathrm{invisible}}=\sparen{(j,l,k)\in\Ic: (\cos\theta_{j,l},\sin\theta_{j,l})^T\not\in W_{\Phi,j} },
\end{equation*}
where $W_{\Phi,j} = \sparen{\xi\in\R^2:\,\xi=r(\cos\omega,\sin\omega)^T,\,r\in\R,\abs{\omega}<\Phi+\pi 2^{-\ceil{j/2}-1}}$ is a polar wedge at scale $2^{-j}$ and $\theta_{j,l}$ is the orientation of the curvelet $\psi_{j,l,k}$. In what follows, curvelets $\psi_{j,l,k}$ as well as curvelet coefficients $c_{n(j,l,k)}$ with $(j,l,k)\in\Ic_\Phi^{\mathrm{invisible}}$ will be called \emph{invisible}\footnote{We adapted the term \emph{invisible} from \cite{Quinto06}.} from the given angular range. Accordingly, the \emph{index set of visible curvelet coefficients} is defined by
\begin{equation*}
	\Ic_\Phi^\mathrm{visible}=\Ic\setminus\Ic_\Phi^\mathrm{invisible}.
\end{equation*}

In view of Theorem \ref{thm:null space}, it holds that $K_{m,n}=\Rc_\Phi\psi_n(\theta_m,s_m)\equiv 0$ for $n\in\Ic_\Phi^{\mathrm{invisible}}$ and for all $1\leq m\leq M$, i.e., those columns of the system matrix $K$ which correspond to the invisible index set are identified to be actually zero. Therefore, we may define a new system matrix $K_\Phi$ with respect to the visible index set by \[(K_\Phi)_{m,n}=\Rc_\Phi \psi_n(\theta_m,s_m), \quad 1\leq m\leq M,\; n\in\Ic_\Phi^\mathrm{visible}.\] 
Such a \emph{reduced system matrix} $K_\Phi$ has the size $M\times \abs{\Ic_\Phi^\mathrm{visible}}$. Since $\abs{\Ic_\Phi^\mathrm{visible}}=\abs{\Ic}-\abs{\Ic_\Phi^\mathrm{invisible}}$, the number of columns is reduced by the quantity $\abs{\Ic_\Phi^\mathrm{invisible}}$. Using the reduced system matrix we formulate the \emph{adapted (or reduced) limited angle problem} as 
\begin{equation}
\label{eq:adapted problem}
	y^\delta = K_\Phi c + \eta.
\end{equation}

The dimension of the adapted problem, given by $N_\Phi=\abs{\Ic}-\abs{\Ic_\Phi^\mathrm{invisible}}$, now depends on the angular range parameter $\Phi$. From the definition of $\Ic_\Phi^\mathrm{visible}$ it is clear that as the angular range becomes larger the number of visible curvelets increases. Hence, the dimension of the adapted problem $N_\Phi$ increases as the angular range increases and vice versa. 

Apply the technique of curvelet sparse regularization (cf. Subsection \ref{ssec:sparse regularization}) to the reduced problem \eqref{eq:adapted problem} we formulate the \emph{adapted curvelet sparse regularization} (A-CSR) as
\begin{equation}
\tag{A-CSR}
\label{eq:adapted csr}
	\hat c_\Phi=\argmin_{c\in\R^{N_\Phi}}\sparen{\frac 1{2}\norm{K_\Phi c-y^\delta}^2_2+\norm{c}_{1,w}}.
\end{equation}

Apparently, the computational amount decreases by using the A-CSR framework instead of the CSR. However, according to Theorem \ref{thm:characterization of sparse regularized curvelet coeffs} the reconstruction quality is preserved. In Section \ref{sec:results}, we will present some practical experiments concerning these issues.

\begin{remark}
	Note that the characterization of Theorem \ref{thm:characterization of sparse regularized curvelet coeffs} may also be applied to the closed form formula \eqref{eq:sparse minimizer via BCD} by replacing the index of summation $\Ic$ by $\Ic_\Phi^\mathrm{visible}$. This yields a closed form solution for \eqref{eq:adapted csr}.
\end{remark}

\section{Discussion}
\label{sec:discussion}
This section is devoted to the discussion of the results which were presented in the previous section as well as their implications for the practical application of the curvelet sparse regularization.

\begin{description}
\item[General angular ranges.] So far, we have worked with a symmetric angular range $[-\Phi,\Phi]$ with $0<\Phi<\pi/2$ which was centered at $\Phi_0=0$. The results of Section \ref{sec:characterizations}, however, can be easily adapted to a more general situation, where the available angular range $[\Phi_0-\Phi,\Phi_0+\Phi]$ is centered around an angle $\Phi_0\in[-\pi,\pi]$. To this end, let $T_{\Phi_0}$ be the translation operator defined by $T_{\Phi_0}\Rc_\Phi f(\theta,s)=\Rc_\Phi f(\theta+\Phi_0,s)$. Then, the limited angle Radon transform with respect to a general angular range $[\Phi_0-\Phi,\Phi_0+\Phi]\times \R$ is given by $T_{\Phi_0}\Rc_\Phi$. Theorem \ref{thm:null space} and Theorem \ref{thm:characterization of sparse regularized curvelet coeffs} can be now applied to $T_{\Phi_0}\Rc_\Phi$, yielding a general index set of invisible curvelet coefficients
\begin{equation*}	
	\Ic_\Phi^{\mathrm{invisible}}=\sparen{(j,l,k)\in\Ic:\,\supp\hat\psi_{j,l,k}\cap R_{\Phi_0}W_\Phi=\emptyset},
\end{equation*}
where  $R_{\Phi_0}W_\Phi=\sparen{R_{\Phi_0}\xi:\; \xi\in W_\Phi}$ is a rotated version of $W_\Phi$.

\item[Computation of the system matrix.]
In Theorem \ref{thm:radon transform of curvelets} we have derived an expression for the Radon transform of a curvelet $\psi_{j,l,k}$. This expression can be used to compute the entries of the system matrix $\Phi$ analytically, if both, the angular window $V$ and the Fourier transform of the radial window $\widehat W$ are known analytically. This is useful for practical application since, in this case, the system matrix can be precomputed and needs not to be set up in every iteration of the minimization of the $\ell^1$-penalized Tikhonov functional. This may yield an additional speedup of the algorithm.

\item[Additional stabilization of the limited angle problem.]
Adapting the problem to the limited angular range has an additionally stabilizing effect. This comes from the fact that the reconstruction problem \eqref{eq:adapted problem} is formulated with respect to visible curvelet coefficients only. In this way, a big portion of the null space of the system matrix (limited angle Radon transform) is excluded from the formulation of the limited angle problem. Therefore, the condition number of the reduced system matrix improves which induces an additional stability.

We want to point out that the formulation of the adapted problem \eqref{eq:adapted problem} does only depend on null space analysis of the limited angle Radon transform in terms of curvelets (cf. Theorem \ref{thm:null space}). Thus, the adapted limited angle problem \eqref{eq:adapted problem} is not related to any reconstruction algorithm. Therefore, the additional stabilization will be present if any other method would be used for solving \eqref{eq:adapted problem}. The adapted formulation of the reconstruction problem \eqref{eq:adapted problem} can be therefore interpreted as preconditioning procedure.

\item[Related work.] In \cite{Frikel2010}, the adapted curvelet sparse regularization was introduced by using microlocal analysis. There, a qualitative characterization of visible curvelets was derived from the characterization of visible singularities of E. T. Quinto \cite{Quinto93} and the \enquote{the resolution of wavefront set property} of the continuous curvelet transform \cite{Candes_CCT_Wavefront_Set05}. These results were stated there without proofs. 

\end{description}

\section{Numerical experiments}
\label{sec:results}
This section is devoted to the illustration of our results which were presented in the previous sections. To this end, two types of numerical experiments were made. In the first part of our experiments we will illustrate the visibility of curvelets under the limited angle Radon transform and show how this leads to a dimensionality reduction in the limited angle reconstruction problem. In particular, these experiments are meant to illustrate Theorem \ref{thm:null space} and Theorem \ref{thm:characterization of sparse regularized curvelet coeffs}. The second part of our experiments is devoted to reconstructions via CSR, A-CSR and filtered backprojection (FBP). A comparison of these reconstructions will be presented in terms of execution times and reconstruction quality.

\subsection{Implementation of the minimization algorithm}
For the minimization of the $\ell^1$-penalized Tikhonov functional \eqref{eq:sparse regularization} we implemented a variant of the well known \emph{iterative soft-thresholding} algorithm \cite{Daubechies_Iterative_Thresholding_Algorithm_for_Inverse_Problems04,BrediesLorenz2008}. This algorithm is given as a fixed point iteration of the equation \eqref{eq:characterization of sparse solution}, namely 
\begin{equation}\label{eq:curvelet thresholding}
	c^{n+1}=\mathcal{S}_{\tau^n}\paren{c^n-s_nK^\ast(Kc^n-y^\delta)},
\end{equation}
where we have used $c^0\equiv 0$ as an initial guess. This procedure consists of a gradient descent step with a subsequent soft-thresholding with respect to the sequence $\tau^n=(\tau^n_{(j,l,k)})_{(j,l,k)\in\mathcal{I}}$.  The step length $s_n$ of the gradient descent step was chosen such that $0<\underline{s}\leq s_n\leq \bar s<2/\norm{K}^2$, \cite{BrediesLorenz2008}. Usually, the thresholding sequence $\tau^n$ is chosen as $\tau^n=s_n\odot w$, where $\odot$ denotes pointwise multiplication of the step length $s_n$ and the $\ell^1$-norm weight sequence $w$ (cf. \eqref{eq:sparse regularization}). This weight sequence is a free parameter and has to be selected appropriately because it affects the reconstruction quality. In general, there is no rule how to select such a weight sequence. In practice, this often done by trial and error.

In our implementation we got rid of this weight sequence by choosing the thresholding sequence $\tau_n$ adaptively and scale-dependent at each iteration $n$ via 
\begin{equation}
	\tau^n_{(j,l,k)} = 2^{3(j-J)/4}\sigma\sqrt{2\log_e N_{j,l}},
\end{equation}
where $\sigma$ denotes the standard deviation of the noise $\eta$, $N_{j,l}$ is the number of curvelet coefficients at scale $2^{-j}$ and at orientation $\theta_{j,l}$ and $J\in\N$ is the largest available scale parameter for the image size of interest. This thresholding strategy was initially presented in \cite[Sec. 6]{Candes_Recovering_Edges_in_Illposed_problems02}. Since it is based on the white noise model, we assumed throughout our experiments the noise to be white Gaussian. 

Moreover, we simulated a practical situation by assuming that the noise, and especially its standard deviation $\sigma$, is not known. In order to automatize the reconstruction procedure, we used the median absolute value (MAD) to estimate $\sigma$ (cf. \cite[p. 565]{Mallat_WaveletTour_SparseWay08}) by
\[\sigma \approx 1.4826\cdot \mathrm{MAD}(c^n_J).\] 
Above, $\mathrm{MAD}(c^n_J)$ is the median of the absolute values of the curvelet coefficient at the finest scale $2^{-J}$. A summarized description of our reconstruction procedure is given in the Algorithm \ref{alg1}.

\begin{algorithm}
\caption{Reconstruction algorithm}
\label{alg1}
\begin{algorithmic}[1]
	\REQUIRE $0<\underline{s}\leq s_n\leq \bar s<2/\norm{K}^2$;
	\STATE $J \leftarrow$ largest scale parameter of the curvelet decomposition;
	\STATE maxIter $\leftarrow$ maximum number of iterations;
	\STATE $c \leftarrow 0$;
	\STATE iter $\leftarrow$ 0;\\[1ex]
	\WHILE{(iter $\leq$ maxIter)}
		
		\STATE $\sigma\leftarrow 1.4826\cdot \mathrm{MAD}(c_J)$;\\[1ex]
		\FOR{ each $(j,l,k)\in\Ic$}
		\STATE $N_{j,l}$ = number of curvelet coefficients at scale $2^{-j}$ and orientation $\theta_{j,l}$;\\[1ex]
		\STATE $\tau^n_{(j,l,k)} \leftarrow 2^{3(j-J)/4}\sigma\sqrt{2\log_e N_{j,l}}$;\\[1ex]
		\ENDFOR\\[1ex]
		\STATE $c \leftarrow \mathcal{S}_{\tau^n}\paren{c-sK^\ast(Kc-y^\delta)}$;\\[1ex]
		\STATE iter  $\leftarrow$ iter + 1;
	\ENDWHILE
\end{algorithmic}
\end{algorithm}

In the following we will use this algorithm to compute CSR reconstructions as well as A-CSR reconstructions, i.e., in the formulation of Algorithms \ref{alg1}, $K$ may be the full or reduced system matrix.

Eventually, we would like to note that the reconstruction Algorithm \ref{alg1} is completely free of any parameter. The weight sequence $w$ appearing in the $\ell^1$-penalized Tikhonov functional \eqref{eq:sparse regularization} or \eqref{eq:adapted csr}, respectively, is set adaptively during the iteration.

\subsection{Visibility of curvelets \& Dimensionality reduction}
In our first experiment, we are going illustrate the visibility of curvelets under the limited angle Radon transform $\Rc_\Phi$ (cf. Theorem \ref{thm:null space}) for different values of $\Phi$. To this end we consider the  function 
\[f=\psi_1+\psi_2+\psi_3+\psi_4,\]
which is a linear combination of curvelets $\psi_i$, $i\in\sparen{1,2,3,4}$, at a fixed scale $2^{-4}$ and orientations $\theta_1=0^\circ$, $\theta_2=20^\circ$, $\theta_3=60^\circ$ and $\theta_4=90^\circ$, see Figure \ref{fig:linear combination of curvelets}. We computed the limited angle Radon transform $\Rc_\Phi f = \Rc_\Phi \psi_1+\Rc_\Phi \psi_2+\Rc_\Phi \psi_3+\Rc_\Phi \psi_4$ and its inverse using the angular range parameters $\Phi=35^\circ$ and $\Phi=80^\circ$. The results of this experiment are shown in Figure \ref{fig:radon transform of curvelets}. The first column shows the limited angle Radon transforms $\Rc_\Phi f$ of $f$ for different values of $\Phi$, whereas the second column shows the inverse Radon transforms $\Rc_\Phi^{-1}\Rc_\Phi f$ from the corresponding limited angle data. In the first row we see that only those curvelets are visible in the reconstruction which correspond to $\theta_1=0^\circ$ and $\theta_2=20^\circ$, i.e., 
\[\Rc_{35^\circ}^{-1}\Rc_{35^\circ} f = \psi_1+\psi_2.\] 
In the second row, we see that another curvelet $\psi_3$ (corresponding to $\theta_3=60^\circ$) becomes visible by enlarging the angular range, i.e., 
\[\Rc_{80^\circ}^{-1}\Rc_{80^\circ} f = \psi_1+\psi_2+\psi_3.\] 
To explain this effect we computed the set of invisible curvelet coefficients according to \eqref{eq:invisible curvelet indices}. As a result, we see that \[\Ic_{35^\circ}^\mathrm{invisible}=\sparen{3,4} \text{ and } \Ic_{80^\circ}^\mathrm{invisible}=\sparen{4}.\]
As a rule of thumb, we can conclude that curvelets having orientations within the available angular range $[-\Phi,\Phi]$ are visible for the angular range Radon transform. However, curvelets which correspond to missing directions are not visible for the limited angle Radon transform

\begin{figure}[H]
\centering
	\includegraphics[height=4.5cm]{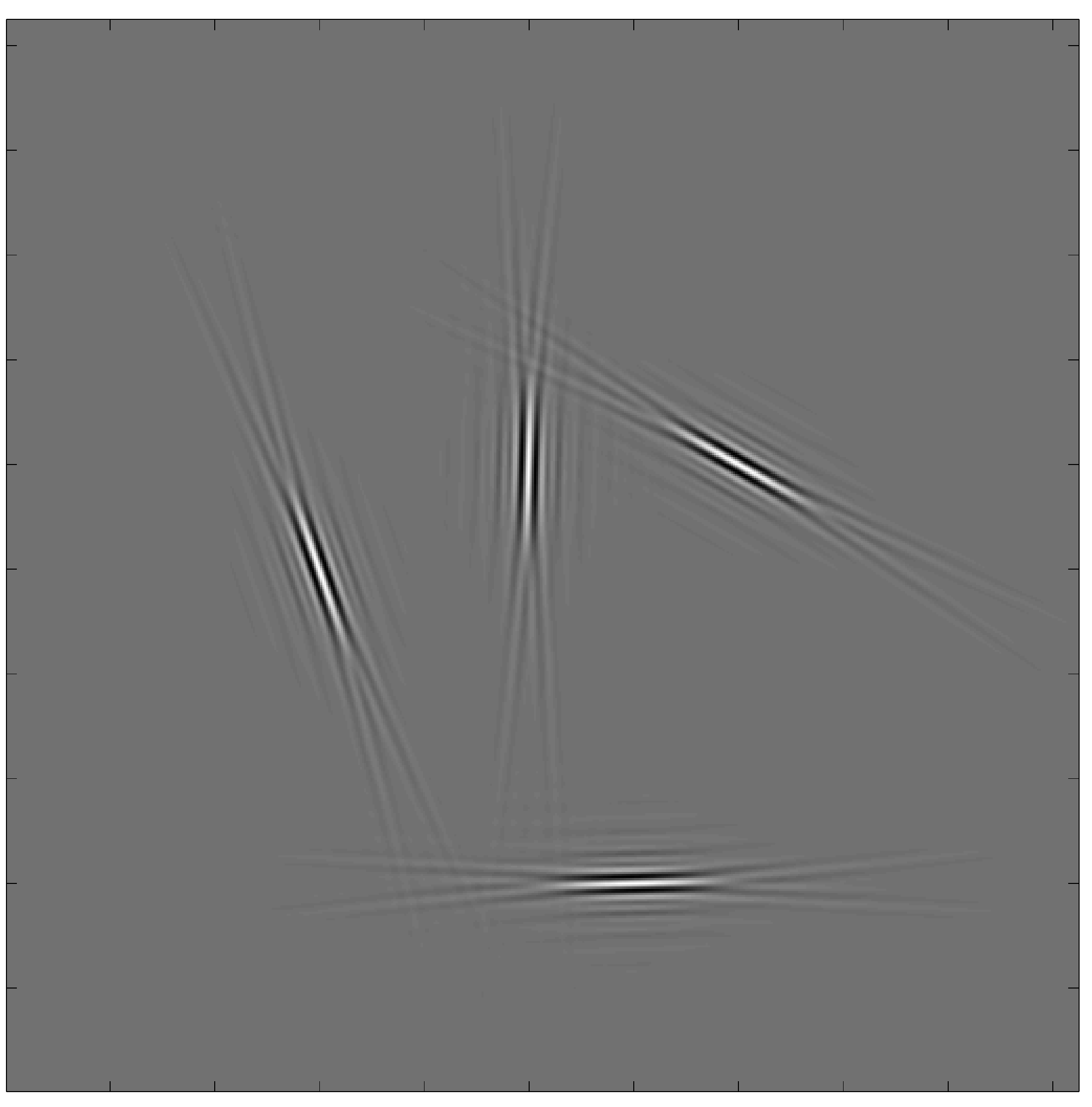}
	\caption{A Matlab generated image of a function which is given as a linear combination of curvelets at scale $2^{-4}$ with orientations $\theta\in\sparen{0^\circ,20^\circ,60^\circ,90^\circ}$.}
	\label{fig:linear combination of curvelets}
\end{figure}

\begin{figure}[H]
\centering
\[
\begin{array}{lcc}
	\toprule
	&\text{Limited angle Radon transform}  & \text{Inverse Radon transform}\\
	\toprule\\
	\Phi=35^\circ&\includegraphics[height=4.5cm]{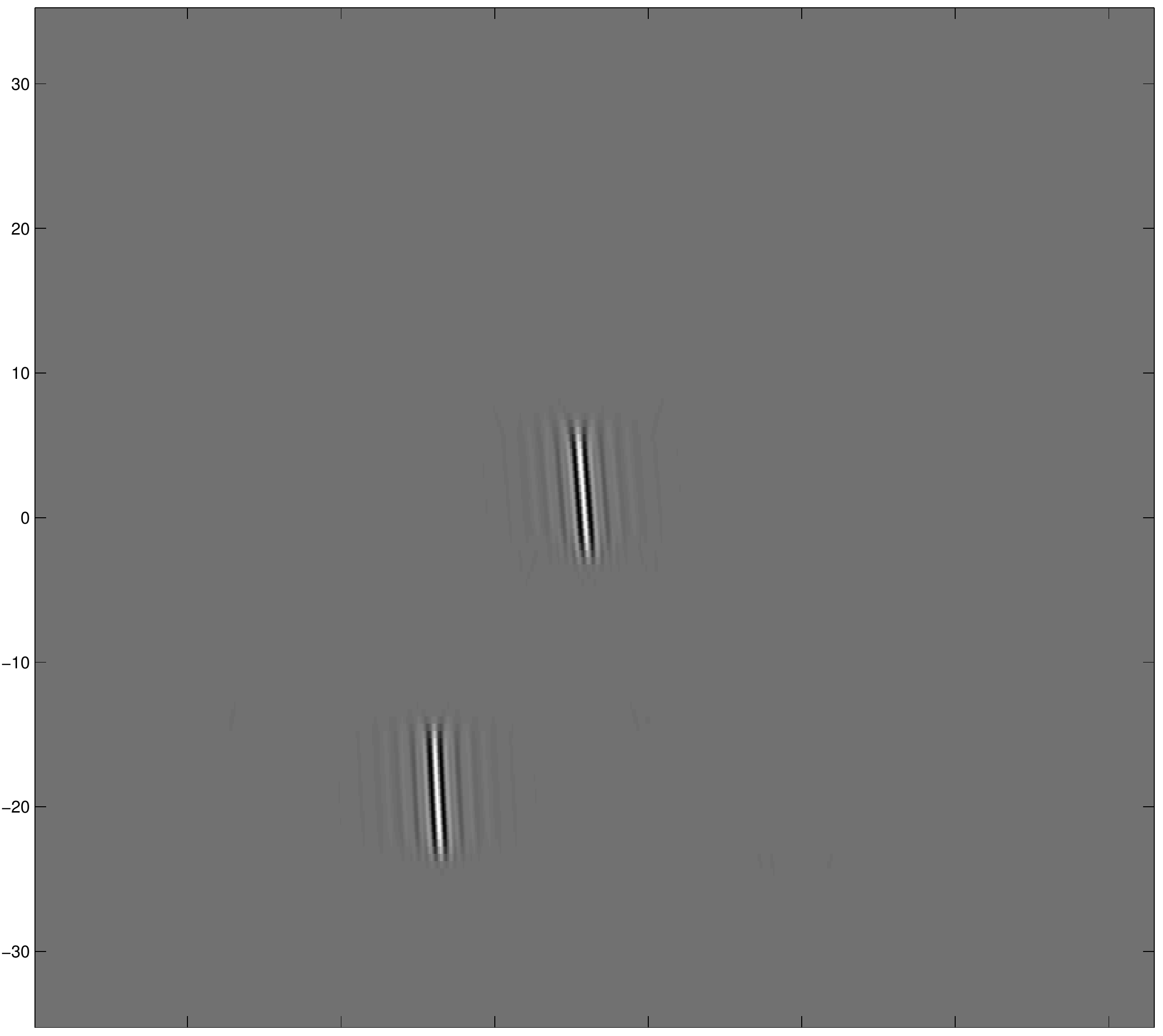}\quad
	&\includegraphics[height=4.5cm]{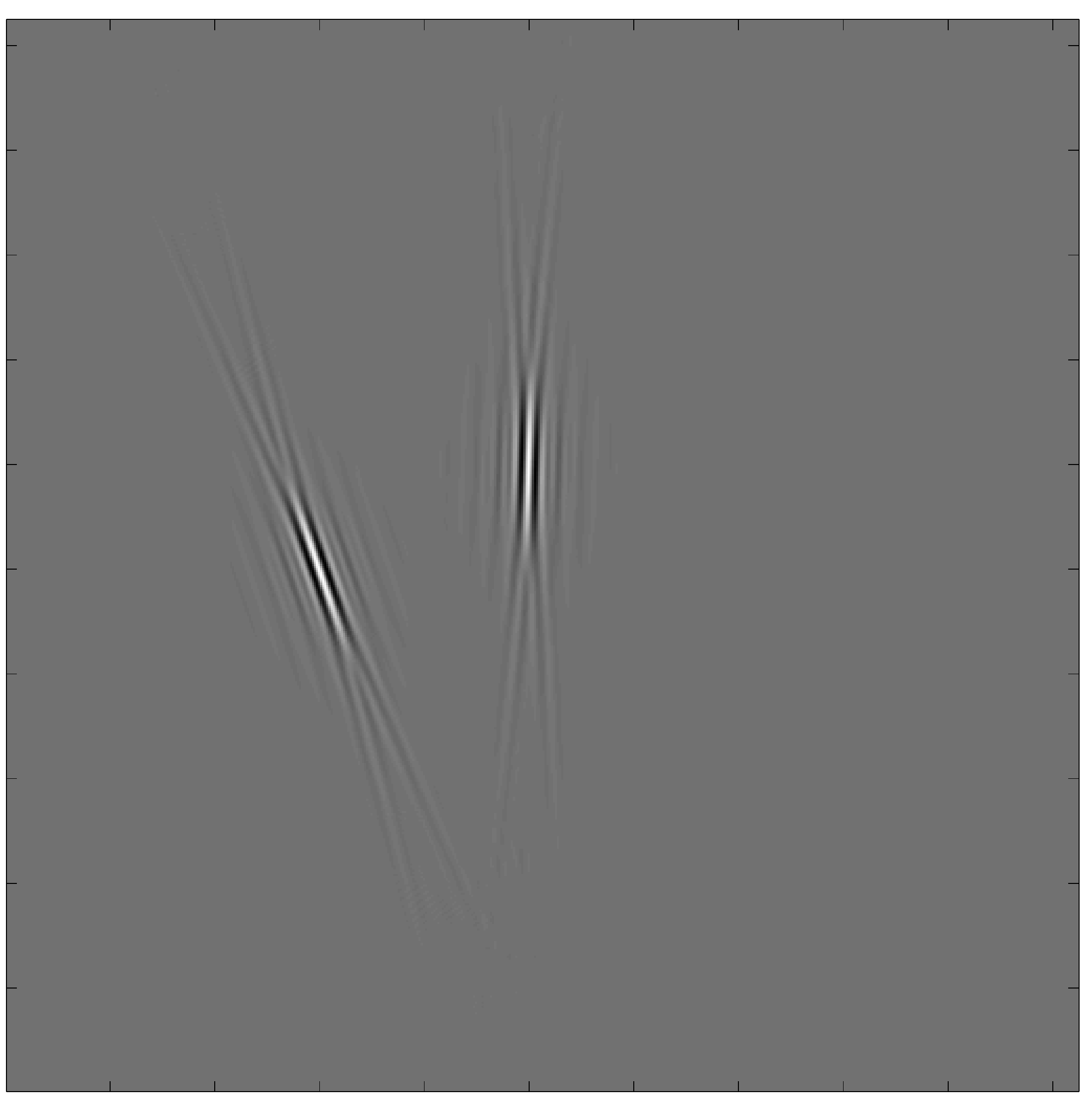}\\[2ex] 
	\midrule\\

	\Phi=80^\circ & \includegraphics[height=4.5cm]{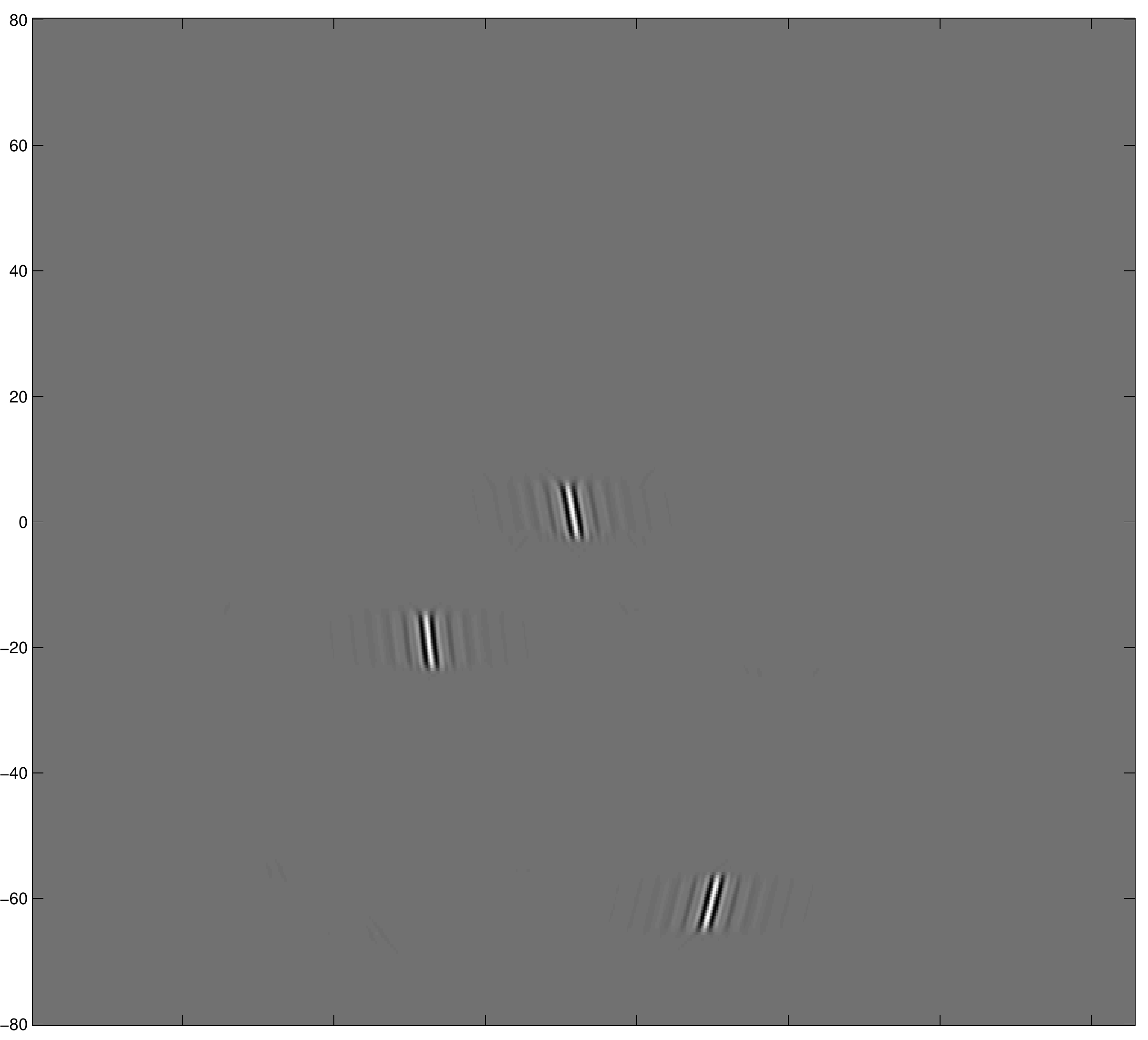}\quad
	& \includegraphics[height=4.5cm]{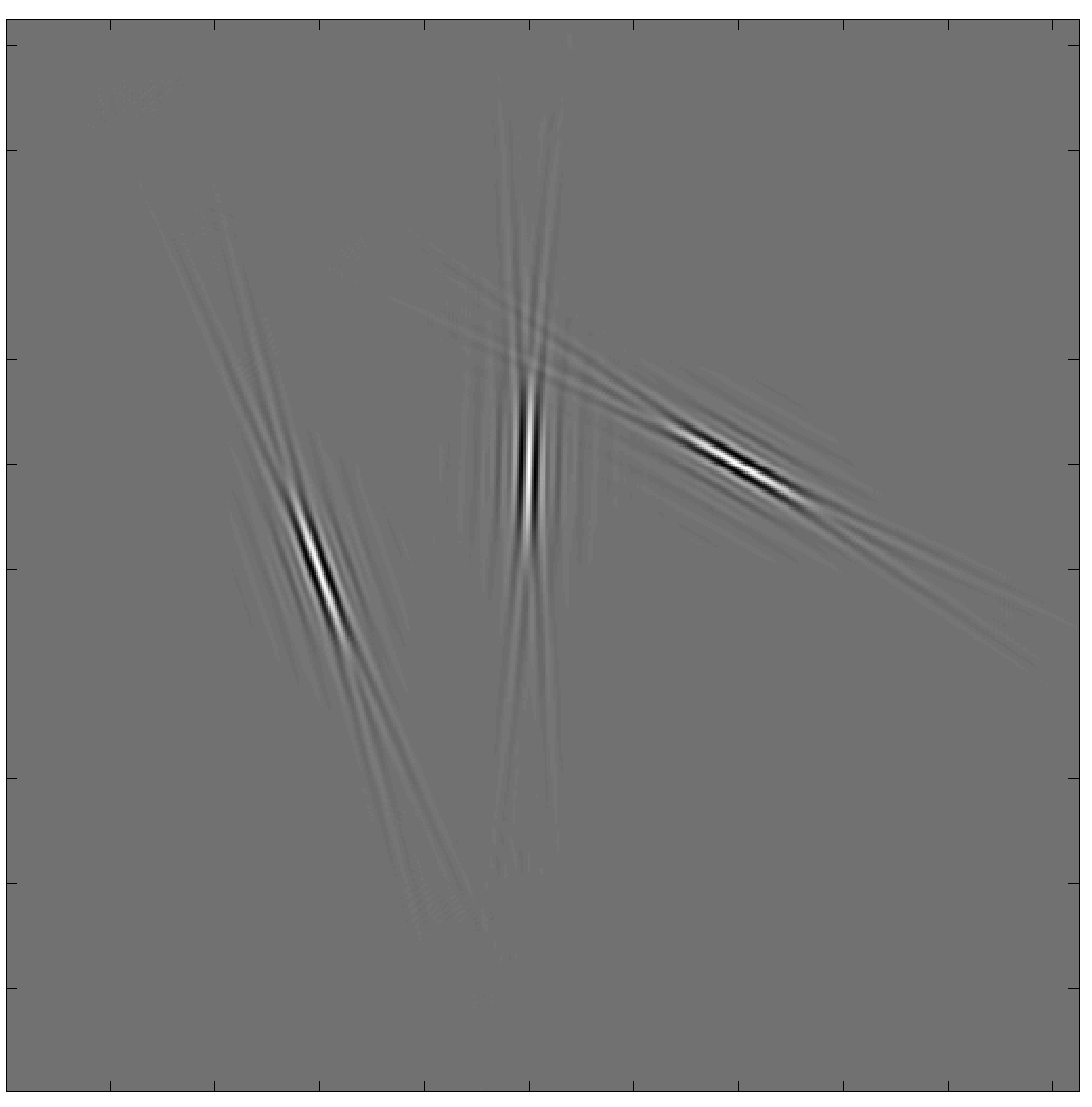}\\[2ex] 
	\bottomrule
\end{array}
\]

\caption{The first column shows the limited angle Radon transform (sinogram) of the image given in Figure \ref{fig:linear combination of curvelets} at an angular range $[-\Phi,\Phi]$ for $\Phi=35^\circ$ and $\Phi=80^\circ$. The second column shows the corresponding inverse Radon transform, i.e., $\Rc_\Phi^{-1}\Rc_\Phi f$. This figure illustrates the Theorem \ref{thm:null space}. We can observe that the curvelet with the orientation $\theta=90^\circ$ is invisible in both cases. Whereas, the curvelet with orientation $\theta=60^\circ$ is invisible for $\Phi=35^\circ$ (first row), i.e., lies in the kernel of the limited angle Radon transform, but visible for $\Phi=80^\circ$ (second row).}
\label{fig:radon transform of curvelets}
\end{figure}

Now it is obvious that if an arbitrary function $f$ is represented in terms of curvelet coefficients, we can seperate the visible and invisible parts of this function by \[f=\sum_{n\in\Ic_{\Phi}^\mathrm{visible}}c_n\psi_n + \sum_{n\in\Ic_{\Phi}^\mathrm{invisible}}c_n\psi_n.\]
This separation depends only on the parameter $\Phi$. The adapted dimension of the limited angle problem is then given by number of visible curvelets $\abs{\Ic_{\Phi}^\mathrm{visible}}$. In the next experiment we computed the full and the reduced dimension for an image $f$ of size $256\times 256$ using CurveLab version 2.1.2, \cite{CurveLab}. The results of this experiment are plotted in Figure \ref{fig:dimension}. The dimension of the non-adapted problem in the curvelet domain is constant for all angular ranges. However, the dimension of the adapted problem shows a strong dependence on the available angular range. We can observe a significant dimensionality reduction for any angular parameter satisfying $\Phi\leq150^\circ$. 

Moreover, we can observe a piecewise constant behavior of the reduced dimension. The dimension increases stepwise linearly as the angular range increases. The reason for this stepwise structure lies in the fact that curvelets remain visible as long as $\supp \hat\psi_{j,l,k}\cap W_\Phi\neq\emptyset$, see also Figure \ref{fig:idea of proof}. The length of one such step therefore corresponds to the length of the of the support of the angular window $V$ of curvelets at the finest scale $2^{-J}$, i.e., to $\abs{\supp V\paren{\frac{2^{\ceil{J/2}+1}}{\pi}\;\bdot}}$.

\begin{figure}[H]
\centering
\begin{tikzpicture}
	\begin{axis}[
		height=5cm, 
		width=13cm,
		xlabel={Angular range $\Theta$},
		xtick={0,10,20,30,40,50,60,70,80,90,100,110,120,130,140,150,160,170,180},
		xticklabels={$0^\circ$,,$20^\circ$,,$40^\circ$,,$60^\circ$,,$80^\circ$,,$100^\circ$,,$120^\circ$,,$140^\circ$,,$160^\circ$,,$180^\circ$},
		xmin = 0,
		xmax= 190,
		axis x line = bottom,
		axis y line = left,
		ylabel={dimension in $10^5$},
		ytick={100000,200000,300000},
		yticklabels={1,2,3},
		ymin=0,
		ymax=330000,
		scaled ticks=false,
		tick label style={font=\small},
		label style={font=\small},
		legend style={font=\small},
		cycle list name=black white,
		legend style={at={(0.98,0.07)},anchor=south east},
		]
		\addplot [thick, dashed] file {dim_full_256x256.dat};
		\addlegendentry{full dim}
		\label{plot:full dim}
		
		\addplot [thick] file {dim_reduced_256x256.dat};
		\addlegendentry{reduced dim}
		\label{plot:reduced dim}		
	\end{axis}%
\end{tikzpicture}%

\caption{Dimension of the full problem \eqref{eq:discrete problem}, \ref{plot:full dim}, and of the adapted problem \eqref{eq:adapted problem}, \ref{plot:reduced dim}, for an image of size $256\times 256$. The plot shows the dependence of the full and reduced dimension on the available angular range $[0,\Theta]$. Since the full problem is formulated in terms of all curvelet coefficients, its dimension is given by the number of all curvelet coefficients $\abs{\Ic}$. The adapted problem, however, is formulated only in terms of visible curvelet coefficients. Hence, the reduced dimension, given by $\abs{\Ic_\Phi^\mathrm{visible}}$, depends strongly on the available angular range.}
\label{fig:dimension}
\end{figure}
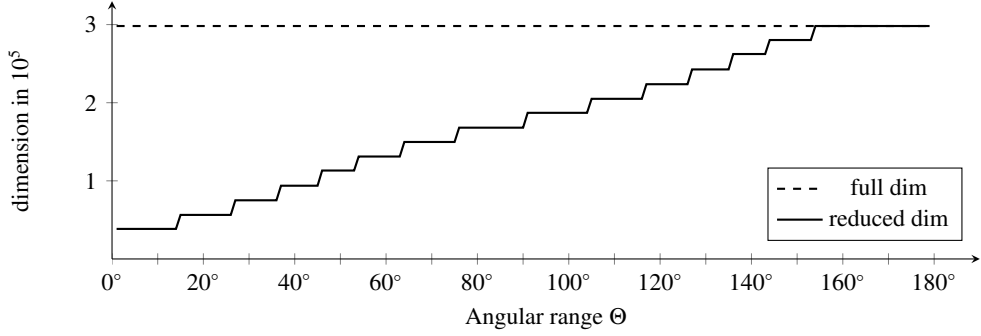

\subsection{CSR vs. A-CSR: Execution times \& reconstruction quality}

\begin{figure}[H]
\centering
	\subfloat[Shepp-Logan head phantom]{\includegraphics[height=4.5cm]{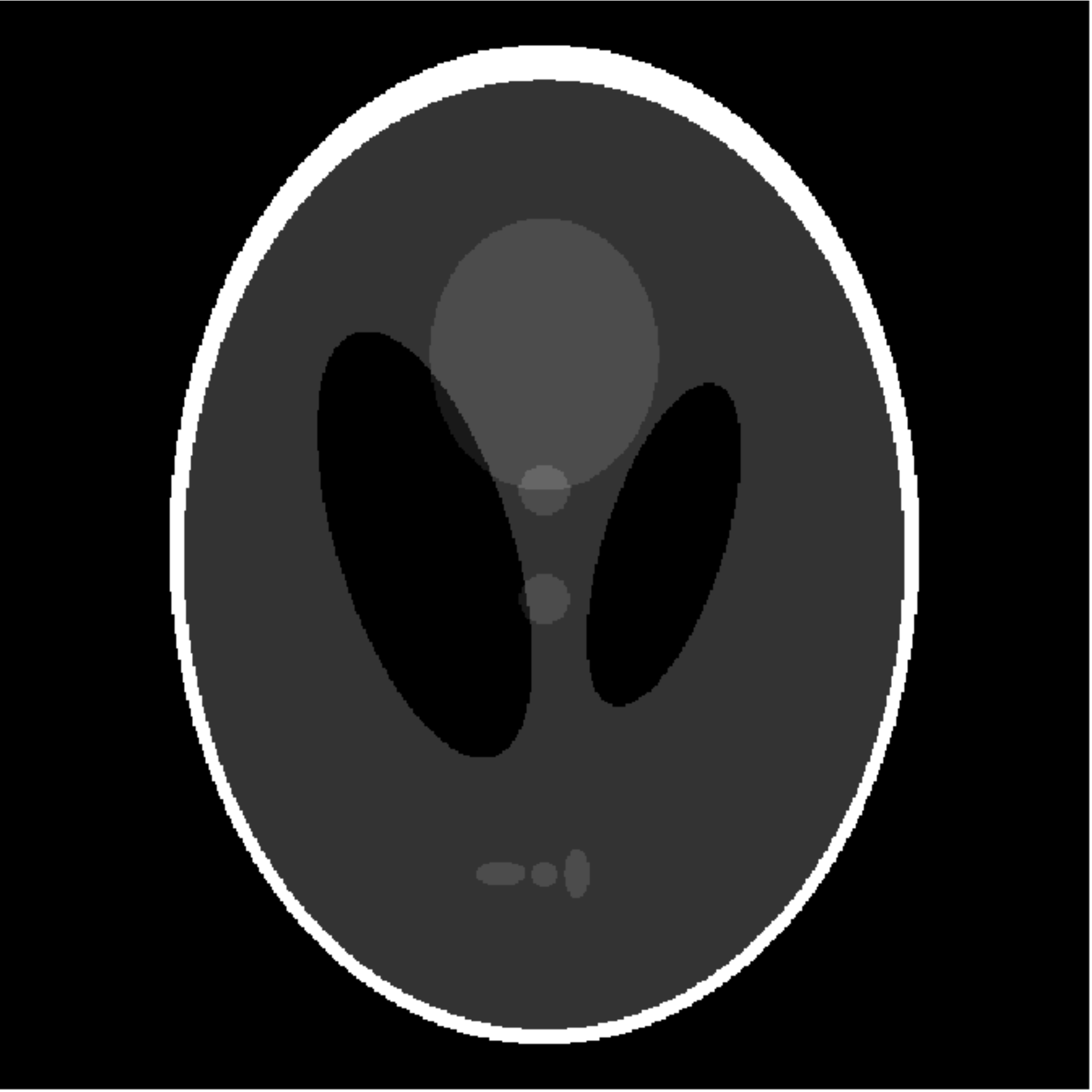}\label{subfig:phantom}}\quad
	\subfloat[Brainstem \cite{Brainstem_testimage}]{\includegraphics[height=4.5cm]{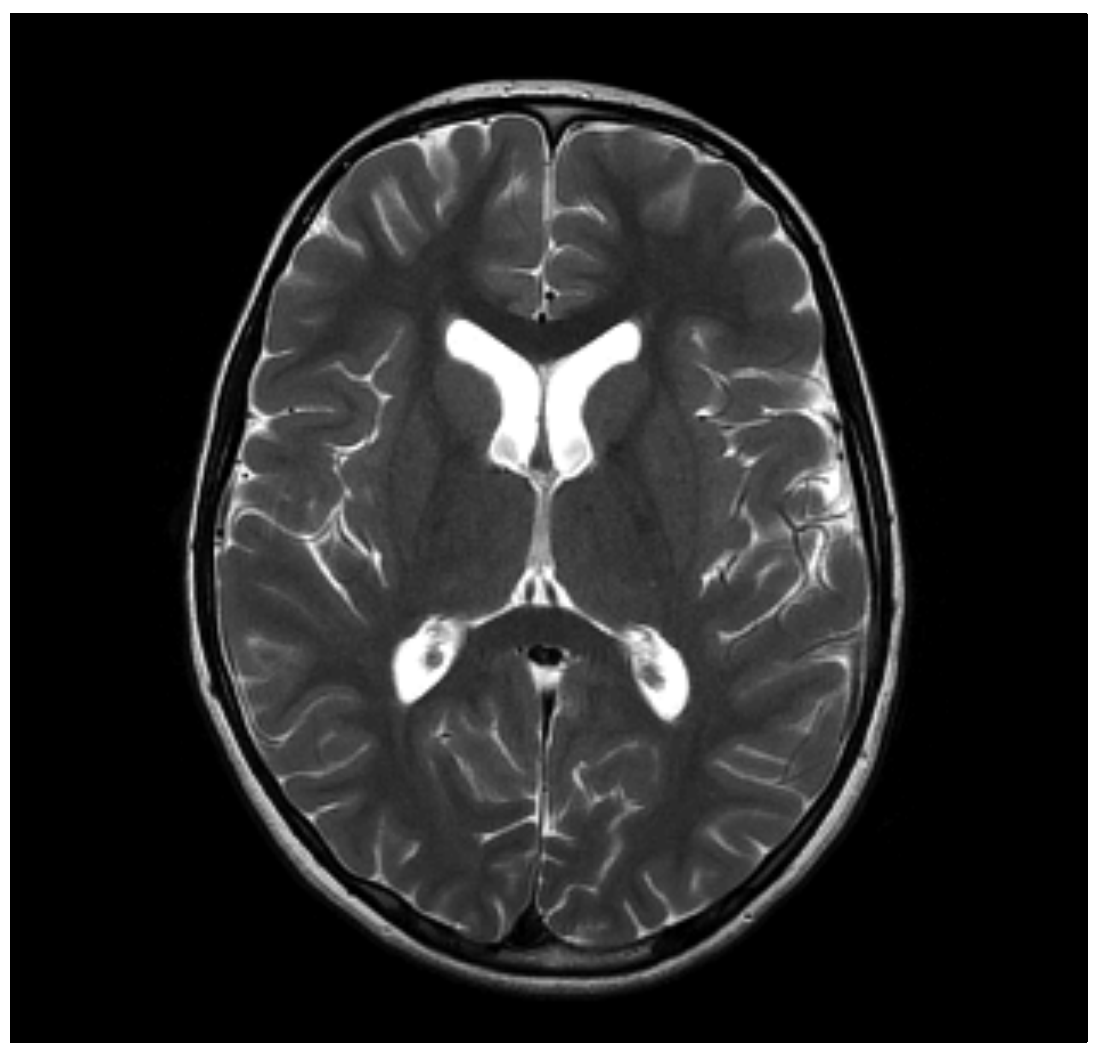}\label{subfig:brainstem}}\quad
	\subfloat[Radial pattern]{\includegraphics[height=4.5cm]{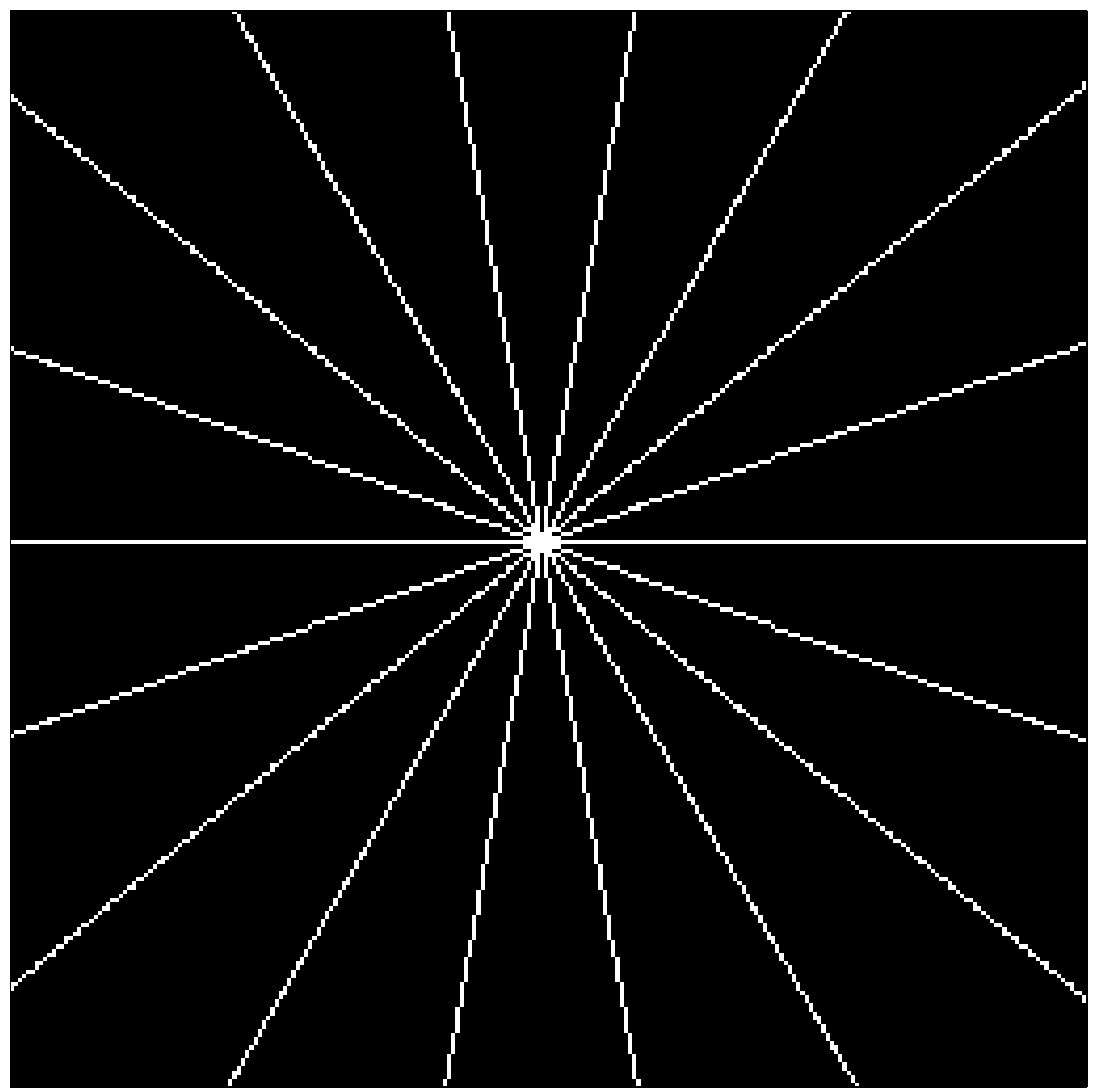}\label{subfig:radial pattern}}
	\caption{Original images.}
	\label{fig:test images}
\end{figure}

In the following experiments we are concerned with limited angle reconstructions obtained via the adapted and via the non-adapted curvelet sparse regularization. In particular, we are going to investigate these reconstructions in terms of execution time of the reconstruction procedure and the reconstruction quality. 

\subsubsection*{Experimental setup}
The limited angle Radon transform was computed for test images which are shown in Figure \ref{fig:test images}. To this end, we considered different angular ranges $[0,\Theta]$, where the parameter $\Theta$ was chosen to vary between $1^\circ$ and $180^\circ$, i.e., $\Theta\in\sparen{1^\circ,\dots,180^\circ}$. The generation of the limited angle data was done using the Matlab function \texttt{radon}. To simulate practical conditions, the generated data was corrupted by a white Gaussian noise, which was generated by the Matlab function \texttt{randn}. Having generated the limited angle data $y^\delta=\Rc_\Phi f + \eta$, we computed the CSR and A-CSR reconstructions using 100 iterations of the Algorithm \ref{alg1}. Instead of computing the system matrix directly and storing it in the memory, we implemented the transform $Kc=\Rc_\Phi T^\ast c$, \eqref{eq:discrete problem}, and its adapted version \eqref{eq:adapted problem} using the Matlab function \texttt{radon} and the CurveLab version 2.1.2, \cite{CurveLab}. Furthermore, we computed filtered backprojection (FBP) reconstructions using the Matlab function \texttt{iradon}.

\subsubsection*{Execution times}

We start by comparing the execution times of the CSR reconstructions to those of A-CSR reconstructions. The results of this experiment are plotted in Figure \ref{fig:execution times}. In this plot, the dotted line \ref{plot:execution time full} indicates the execution times of the CSR reconstructions, whereas the solid line \ref{plot:execution time adapted} shows the execution times of the adapted approach \eqref{eq:adapted csr}. The  dependence of the execution times on the available angular range shows in both cases a linear behavior. In particular, we can observe a significant speedup of the adapted procedure, especially for angular ranges $[0,\Theta]$ with $\Theta\leq 120^\circ$. The speedup exhibits a linear dependence on the available angular range which is due to the dimensionality reduction which was presented in the previous experiment, cf. Figure \ref{fig:dimension}. 

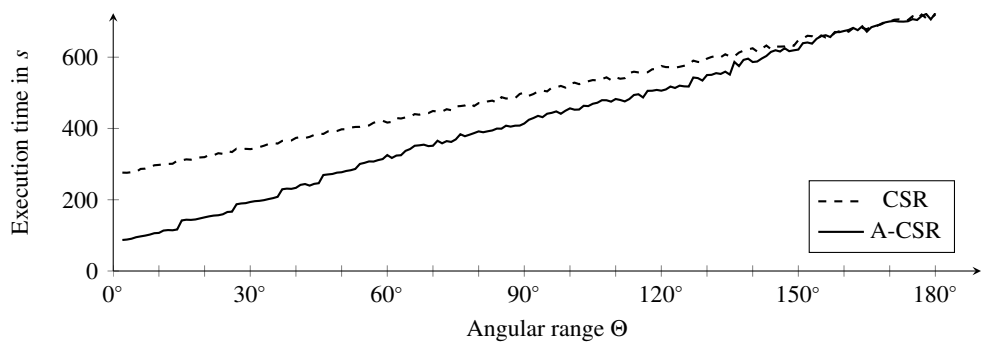
\begin{figure}[H]
\centering
\begin{tikzpicture}
	\begin{axis}[
		height=5cm, 
		width=13cm,
		xlabel={Angular range $\Theta$},
		xtick={0,10,...,180},
		xticklabels={$0^\circ$,,,$30^\circ$,,,$60^\circ$,,,$90^\circ$,,,$120^\circ$,,,$150^\circ$,,,$180^\circ$},
		xmin = 0,
		xmax= 190,
		axis x line = bottom,
		axis y line = left,
		ylabel={Execution time in $s$},
		ymin=0,
		scaled ticks=false,
		tick label style={font=\small},
		label style={font=\small},
		legend style={font=\small},
		legend style={at={(0.97,0.07)},anchor=south east},
		skip coords between index={0}{1},
		]
				
		\addplot  [thick,dashed] file {csr_time_iterations100.dat};
		\addlegendentry{CSR}
		\label{plot:execution time full}
		
		\addplot  [thick]  file {csr_adapted_time_iterations100.dat};
		\addlegendentry{A-CSR}
		\label{plot:execution time adapted}
	\end{axis}%
\end{tikzpicture}%
\caption{Execution times for CSR and A-CSR reconstruction using 100 iterations of Algorithm \ref{alg1}. Reconstruction of the Shepp-Logan head phantom of size $256\times 256$ (Figure \ref{fig:test images}\subref{subfig:phantom}) at different angular ranges $[0,\Theta]$, $\Theta\in\sparen{1^\circ,\dots,180^\circ}$. There is a significant speedup of the reconstruction procedure when using A-CSR.}
\label{fig:execution times}
\end{figure}

\subsubsection*{Reconstruction quality}

The results of the limited angle reconstruction are show in Figures \ref{fig:rec phantom} - \ref{fig:rec radial pattern} for angular ranges $[0,35^\circ]$ and $[0,160^\circ]$. The original images corresponding to these reconstruction are shown in Figure \ref{fig:test images}. We investigate the reconstruction quality, first, by considering the CSR and the A-CSR reconstructions of the Shepp-Logan head phantom. These reconstructions are shown in the first and in the second column of Figure \ref{fig:rec phantom}. By visually inspecting the images in each row separately, we can observe no difference in image quality. Inspecting reconstructions of the brainstem image (Figure \ref{subfig:brainstem}) and the radial pattern image (Figure \ref{subfig:radial pattern}) which are shown in Figures \ref{fig:rec brainstem} and \ref{fig:rec radial pattern}, we can again observe that there no difference in image quality between CSR and A-CSR reconstructions. Therefore, we infer that the CSR and the A-CSR produce reconstructions of the same visual quality. 

To make these observations independent of visual perception, we used the \emph{mean squared error} (MSE) as a quality measure. This is defined as 
\[\mathrm{MSE}(c^\mathrm{rec}) = \frac{1}{N}\sum_{n=1}^N\abs{c_n-c^\mathrm{rec}_n}^2,\]
where $c$ are the curvelet coefficients of the original image and $c^\mathrm{rec}$ denotes those curvelet coefficients which were obtained via CSR or A-CSR at different angular ranges. The resulting MSE values are plotted in Figure \ref{plot:mse}. As a function of the angular range parameter $\Phi$, MSE is decreasing for the non-adapted as well as for the adapted reconstruction method. However, the plots of the MSE values for CSR and A-CSR reconstructions again seem to be identical, cf. Figure \ref{plot:mse}. To refine our investigation we additionally consider the relative MSE,
\[\mathrm{MSE}(c^\mathrm{CSR},c^\mathrm{A-CSR}) = \frac{1}{N}\sum_{n=1}^N\abs{c^\mathrm{CSR}_n-c^\mathrm{A-CSR}_n}^2,\]
which compares the reconstructed curvelet coefficients obtained via CSR and those obtained via A-CSR. The plot of these values is shown in Figure \ref{plot:relative mse}. Here, we can observe how large the difference between the CSR and A-CSR reconstructions is in the case of Shepp-Logan head phantom reconstructions. Depending on the available angular range, the relative MSE values differ between $10^{-5}$ and $10^{-7}$.

As a result of the above discussion, we can conclude that the difference in the reconstruction quality of the CSR and the A-CSR reconstructions is very small. Visually, the reconstructions are not distinguishable. Therefore, the advantage of using the A-CSR approach consists in its significantly faster execution time. 

However, one might still ask where these differences come from? A possible explanation would be as follows: Since the reconstructed sequence of curvelet coefficient $c^\mathrm{CSR}$ contains invisible curvelet coefficients, these values may be not zero after a finite number of iterations and, hence, these values would contribute to the relative MSE. Such a behavior was observed during these numerical experiments, thought the values of the invisible curvelets were very small.

\begin{figure}[H]
\centering
\[
\begin{array}{lccc}
	\toprule
	&\mathrm{CSR}  & \mathrm{A-CSR} & \mathrm{FBP}\\
	\toprule\\
	\Theta=35^\circ
	&\includegraphics[height=4.5cm]{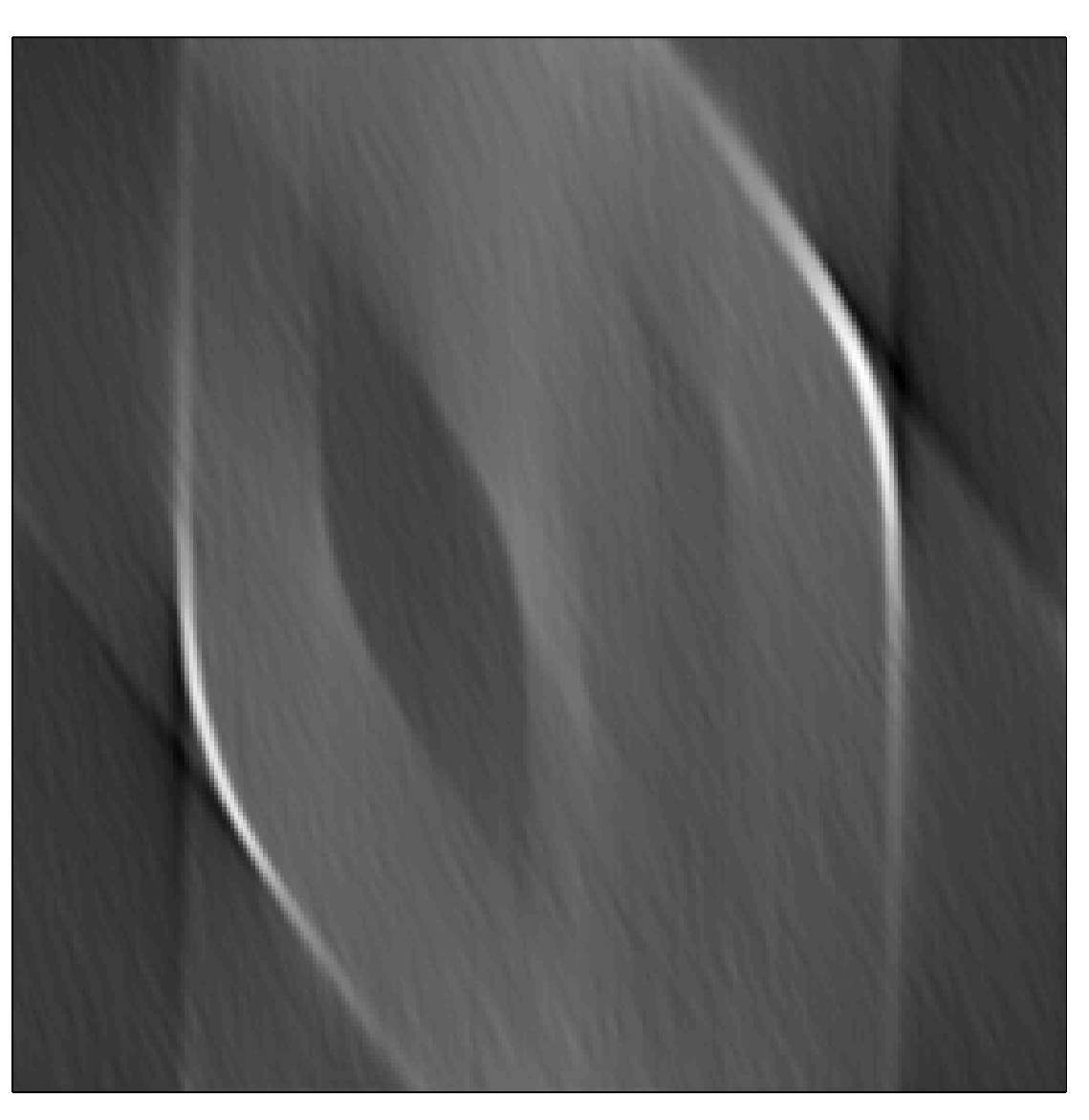}
	&\includegraphics[height=4.5cm]{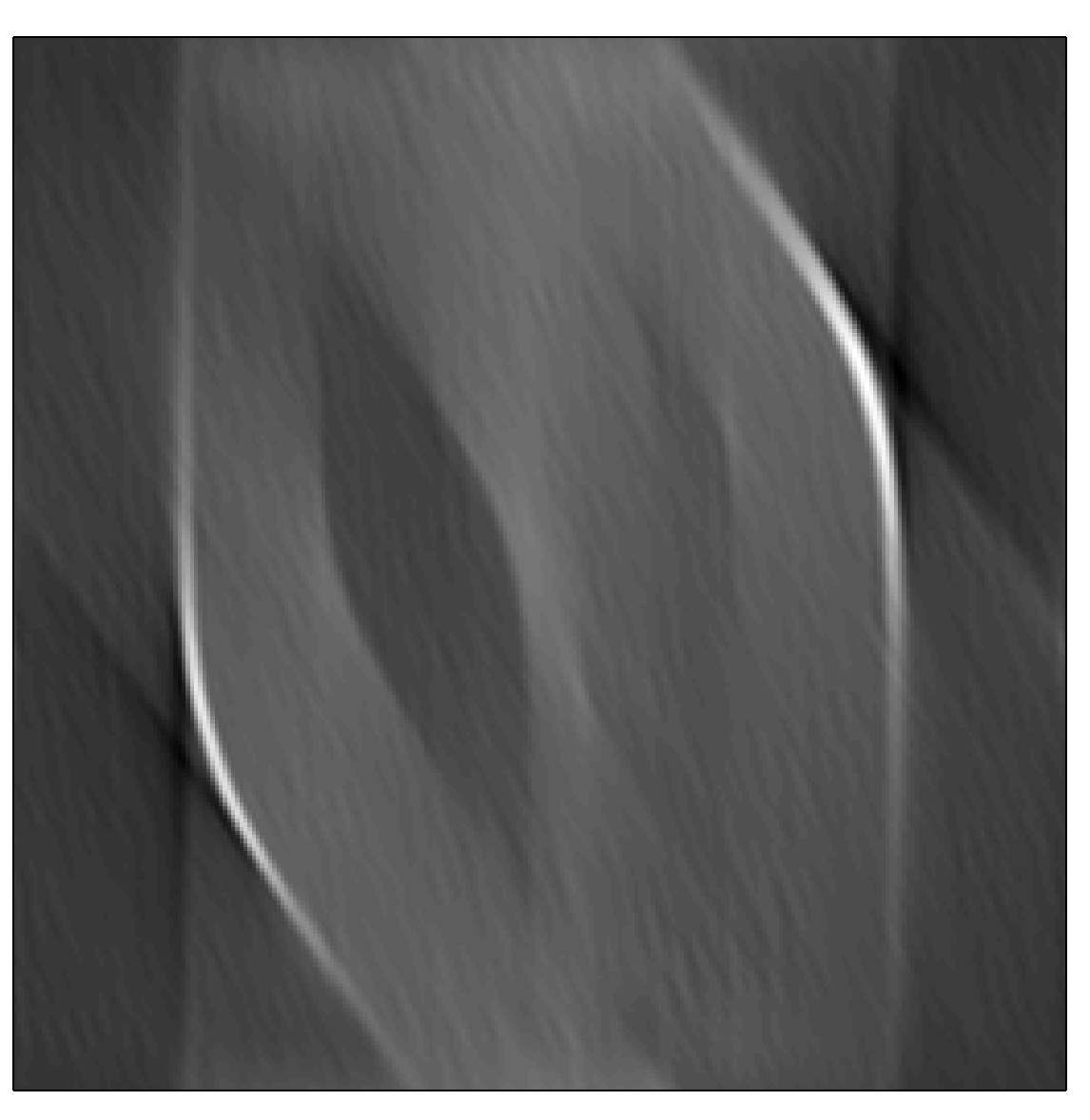}
	&\includegraphics[height=4.5cm]{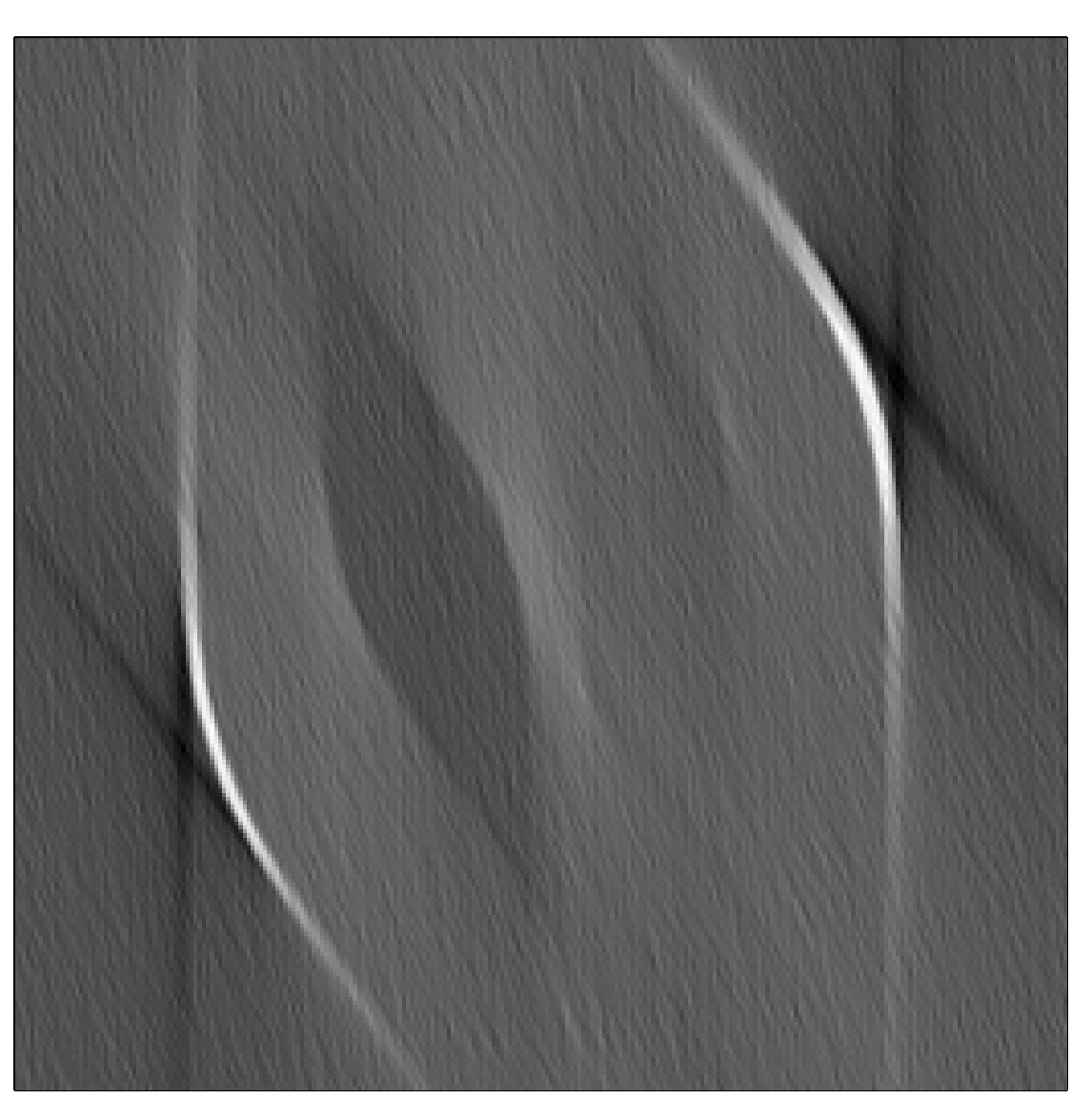}\\[2ex] 
	\midrule\\

	\Theta=160^\circ 
	&\includegraphics[height=4.5cm]{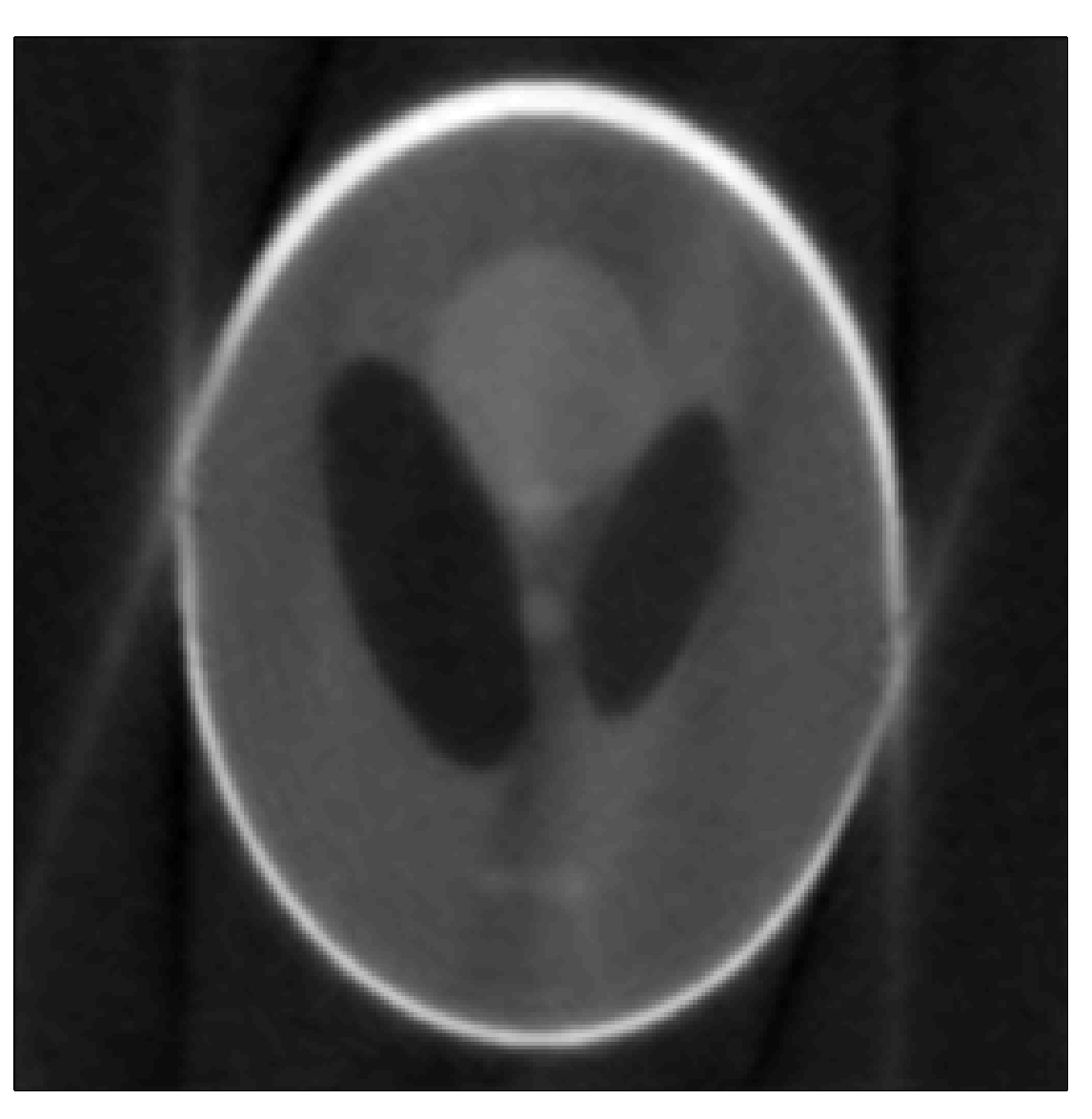}
	&\includegraphics[height=4.5cm]{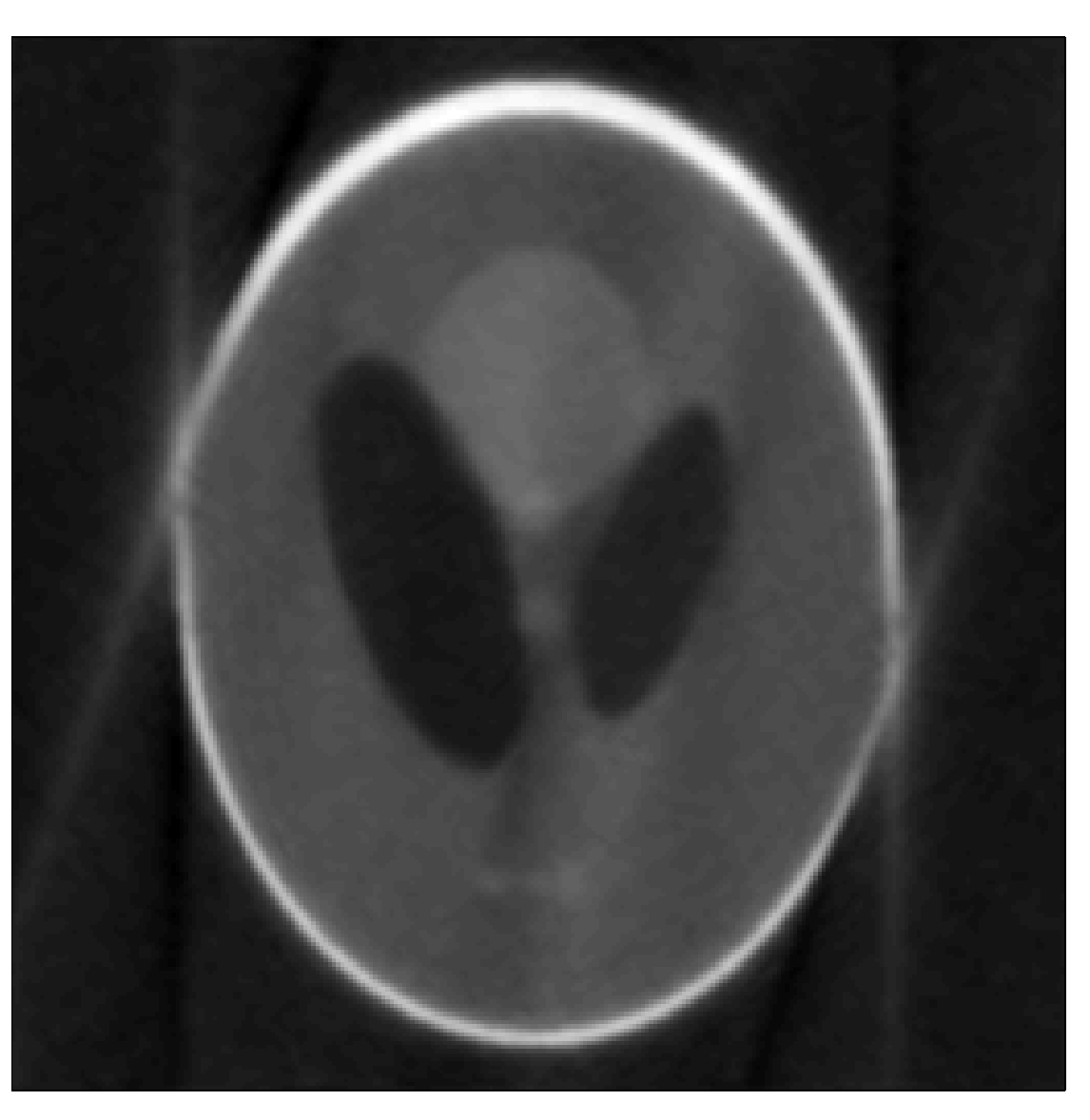}
	&\includegraphics[height=4.5cm]{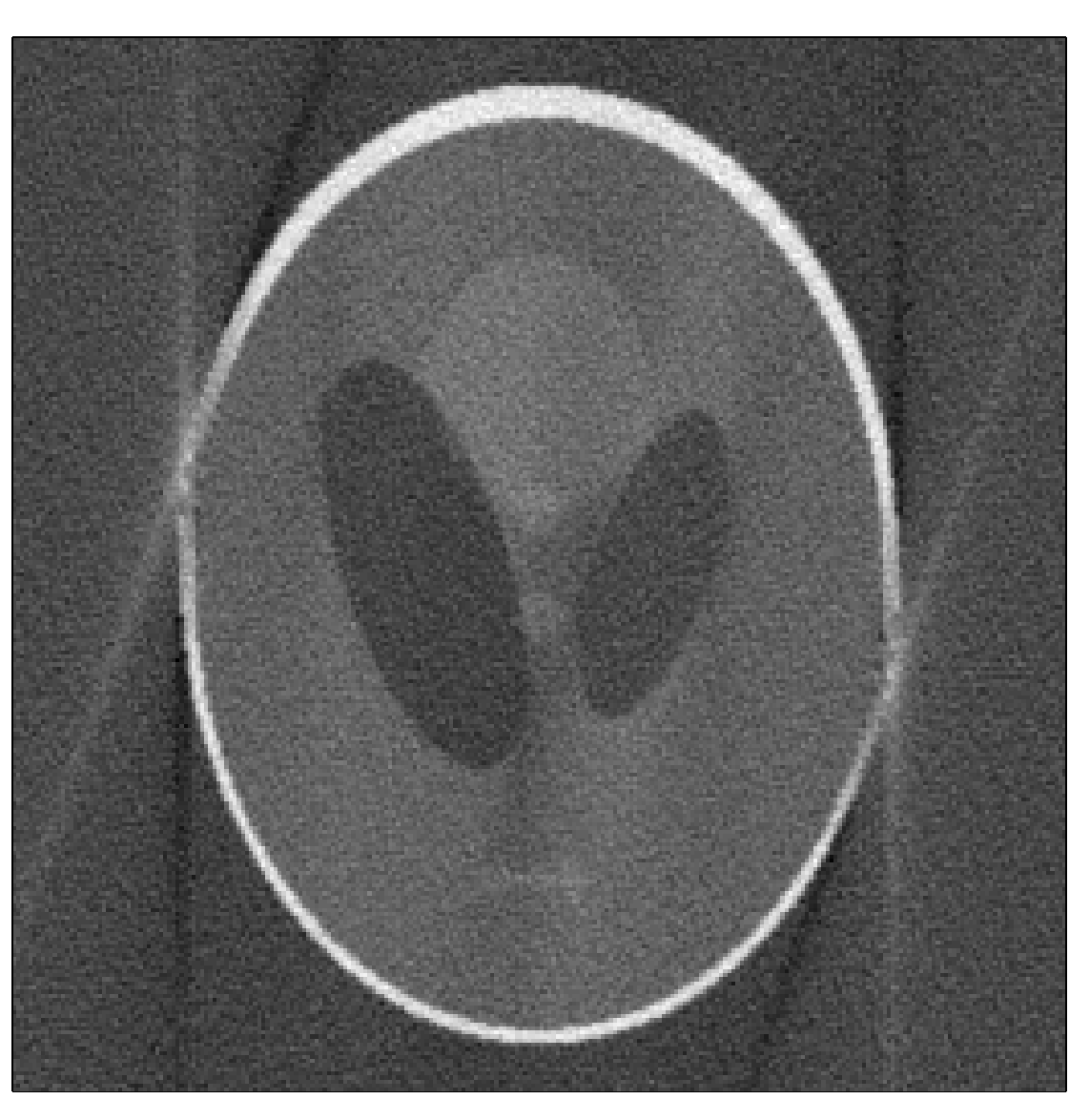}\\[2ex] 
	\bottomrule
\end{array}
\]
\caption{Reconstruction of the Shepp-Logan head phantom of size $256\times 256$ (Figure \ref{fig:test images}\subref{subfig:phantom})  at an angular range $[0,\Theta]$ and noiselevel $2\%$ by using CSR, A-CSR and FBP. In the above matrix of images, each row shows a reconstruction corresponding to the angular range parameter $\Theta\in\sparen{35^\circ,160^\circ}$. Visually, there is no difference between CSR and A-CSR reconstructions for any of the angular ranges. However, the A-CSR reconstructions were computed significantly faster. In contrast to FBP, CSR and A-CSR reconstructions appear less noisy. Though CSR and A-CSR reconstructions are slightly smoother than the FBP reconstructions, the edges are clearly visible.}
\label{fig:rec phantom}
\end{figure}

\begin{figure}[H]
\centering
\[
\begin{array}{lccc}
	\toprule
	&\mathrm{CSR}  & \mathrm{A-CSR} & \mathrm{FBP}\\
	\toprule\\
	\Theta=35^\circ
	&\includegraphics[height=4.2cm]{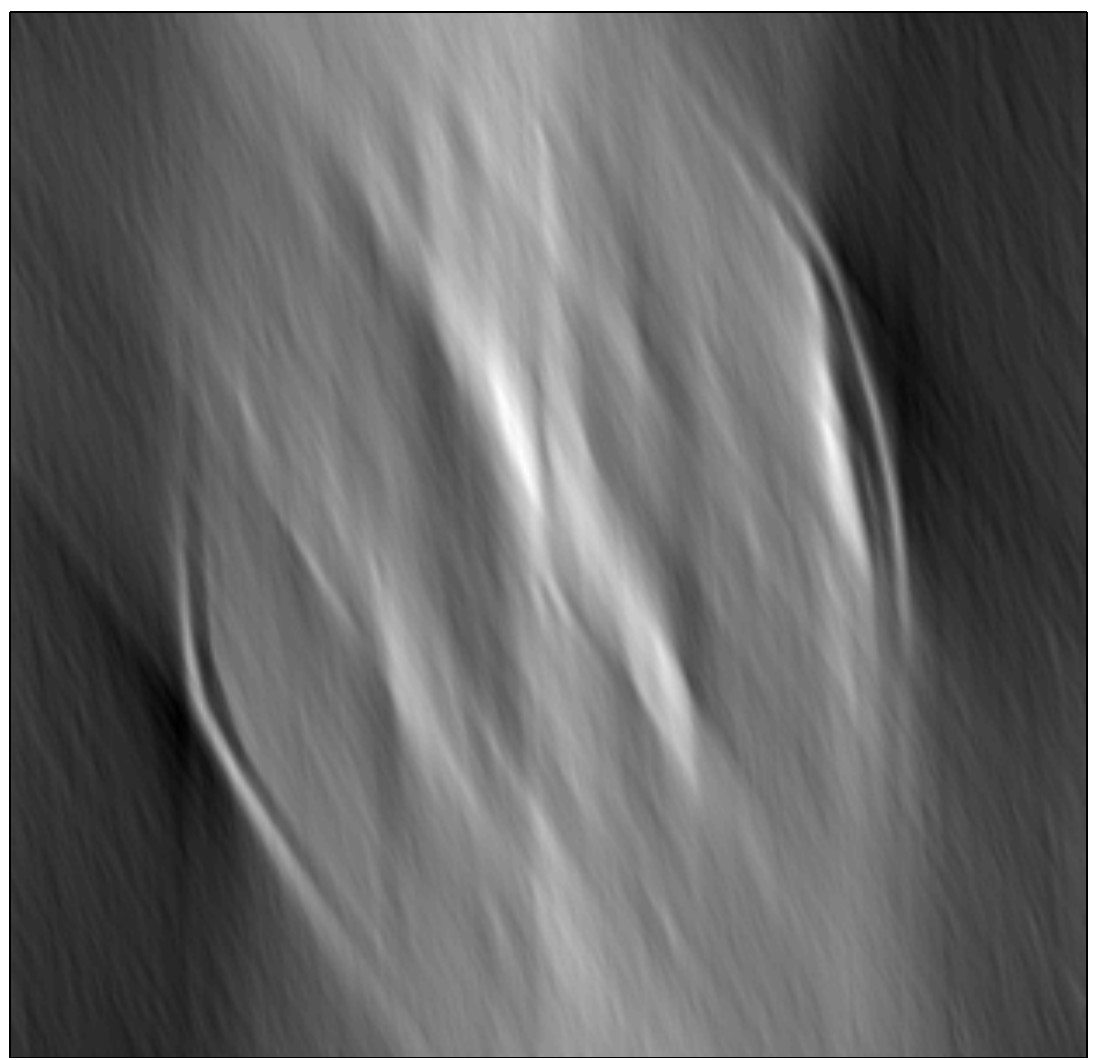}
	&\includegraphics[height=4.2cm]{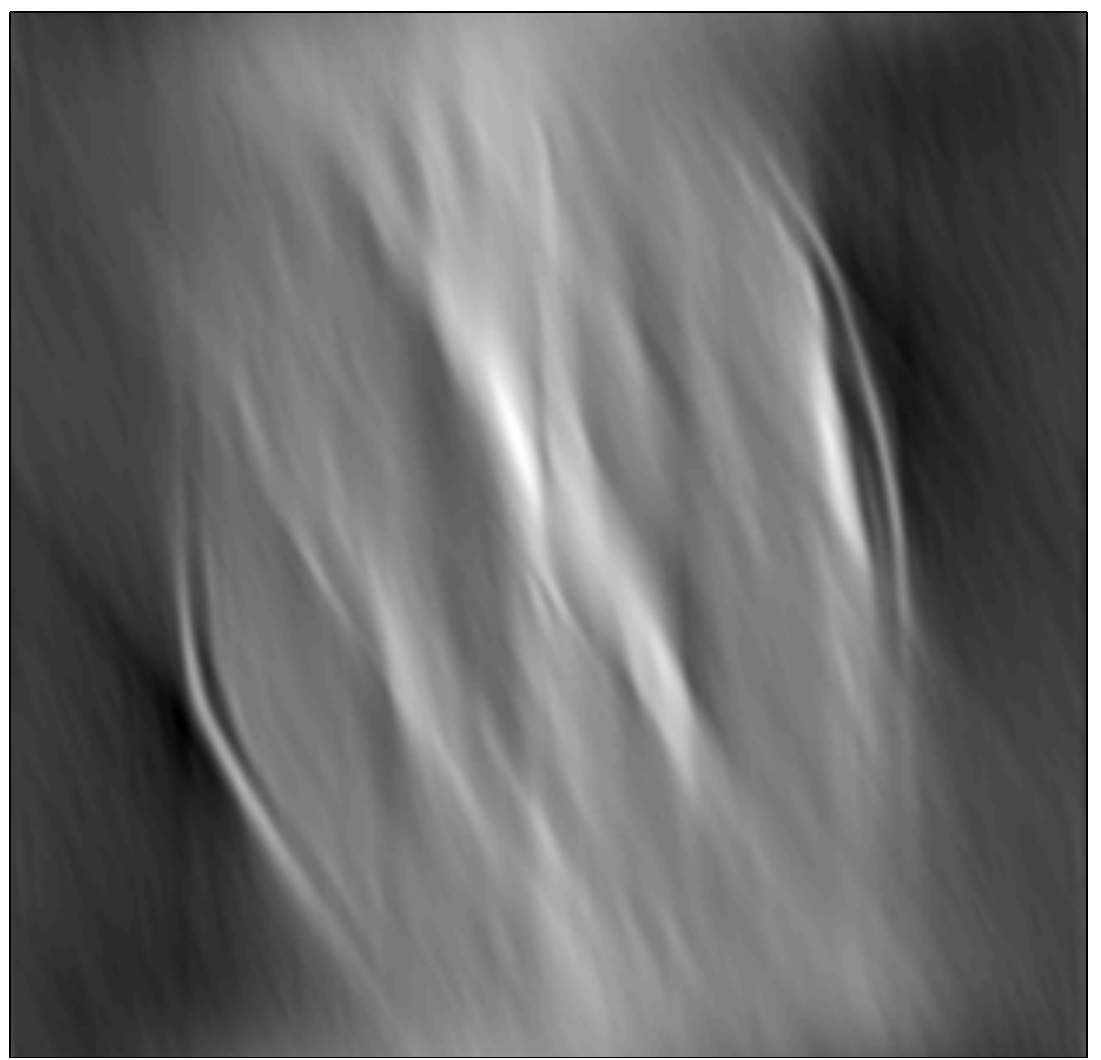}
	&\includegraphics[height=4.2cm]{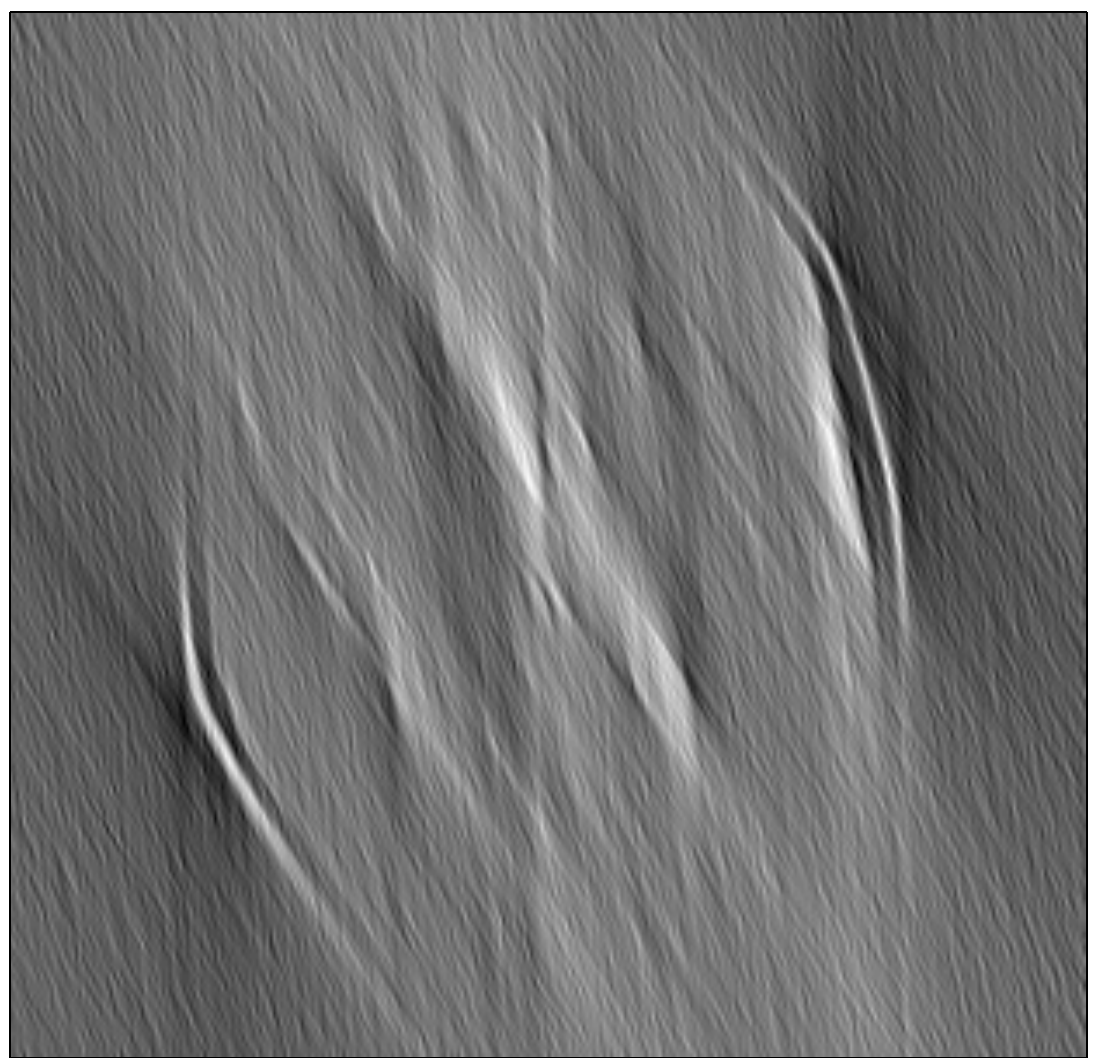}\\[2ex] 
	\midrule\\

	\Theta=160^\circ 
	&\includegraphics[height=4.2cm]{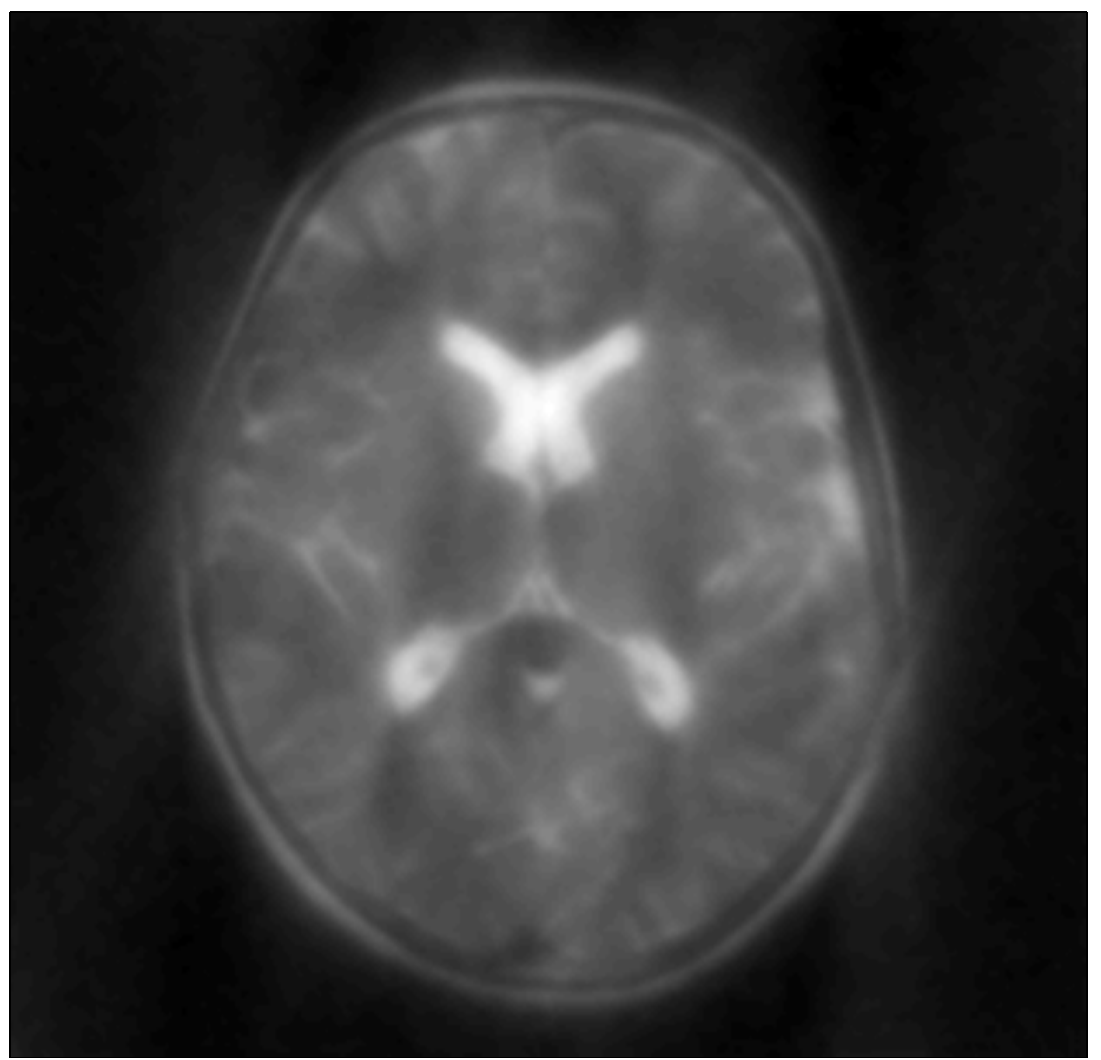}
	&\includegraphics[height=4.2cm]{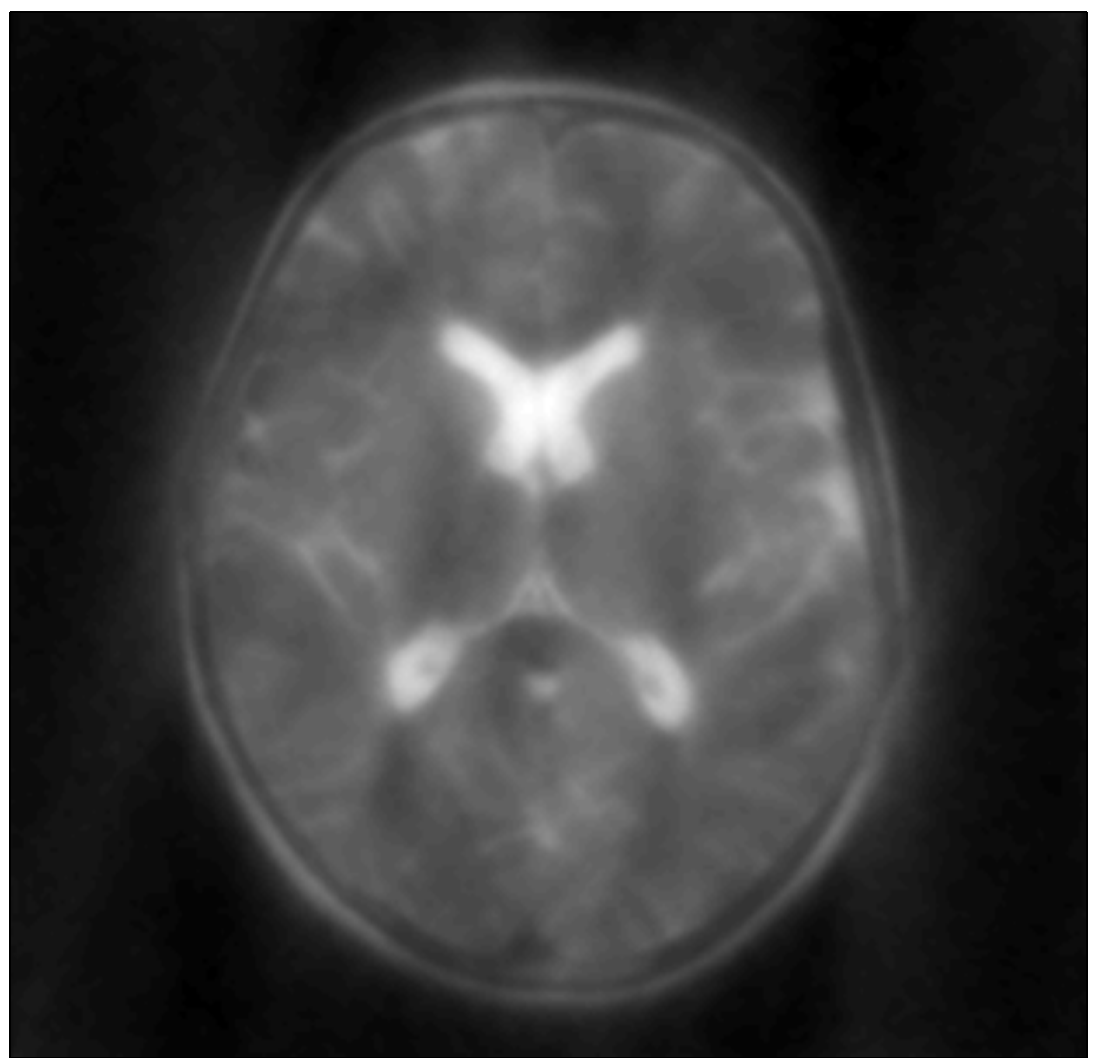}
	&\includegraphics[height=4.2cm]{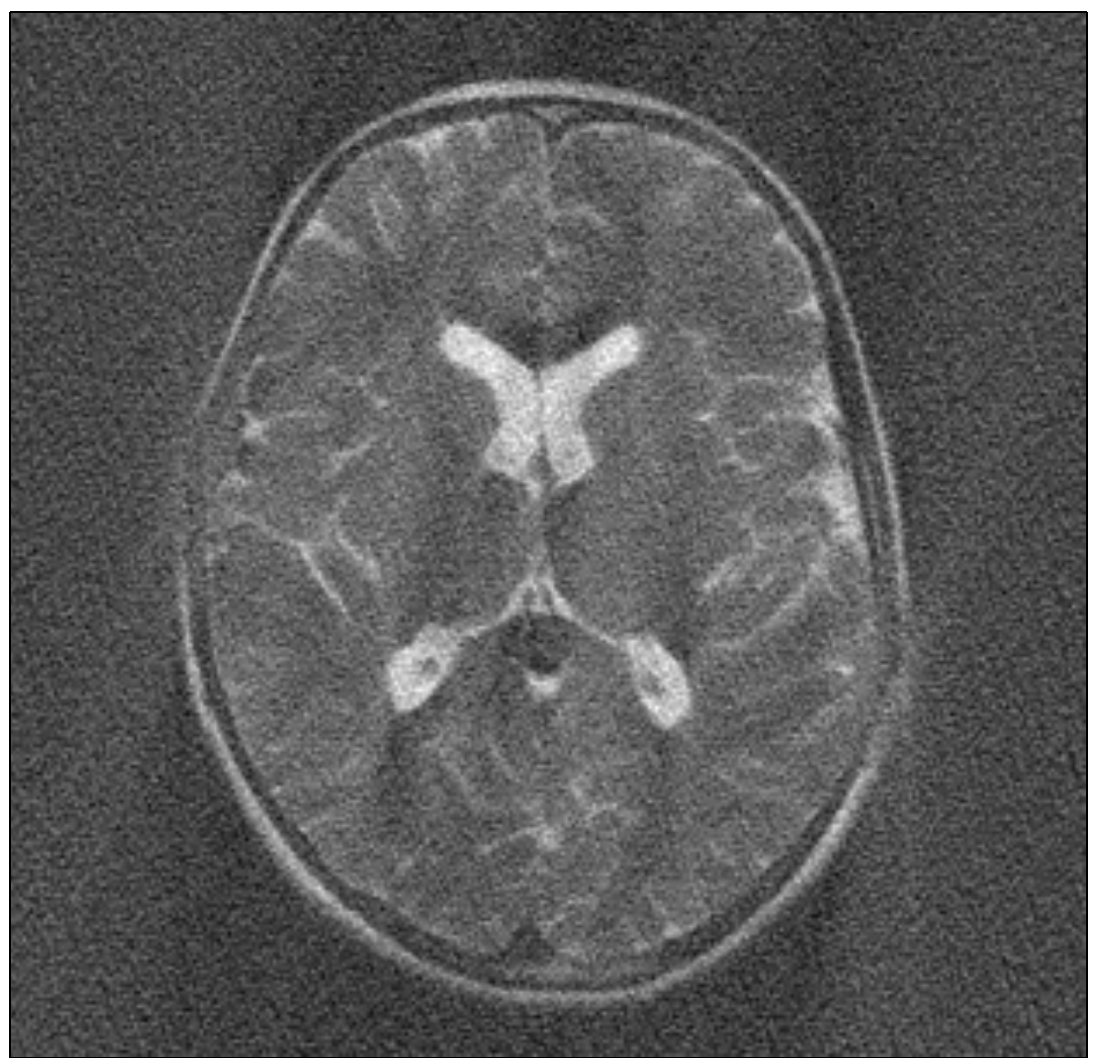}\\[2ex] 
	\bottomrule
\end{array}
\]
\caption{Reconstruction of a brainstem glioma of size $301\times 310$ (Figure \ref{fig:test images}\subref{subfig:brainstem}) at an angular range $[0,\Theta]$, noiselevel $2\%$ using CSR, A-CSR and FBP. Again, there is no visible difference between CSR and A-CSR reconstructions for any of the angular ranges. Though CSR and A-CSR reconstructions are slightly smoother than the FBP reconstructions, the overall image quality is better in the case of CSR and A-CSR.}
\label{fig:rec brainstem}
\end{figure}

\begin{figure}[H]
\centering
\[
\begin{array}{lccc}
	\toprule
	&\mathrm{CSR}  & \mathrm{A-CSR} & \mathrm{FBP}\\
	\toprule\\
	\Theta=35^\circ
	&\includegraphics[height=4.2cm]{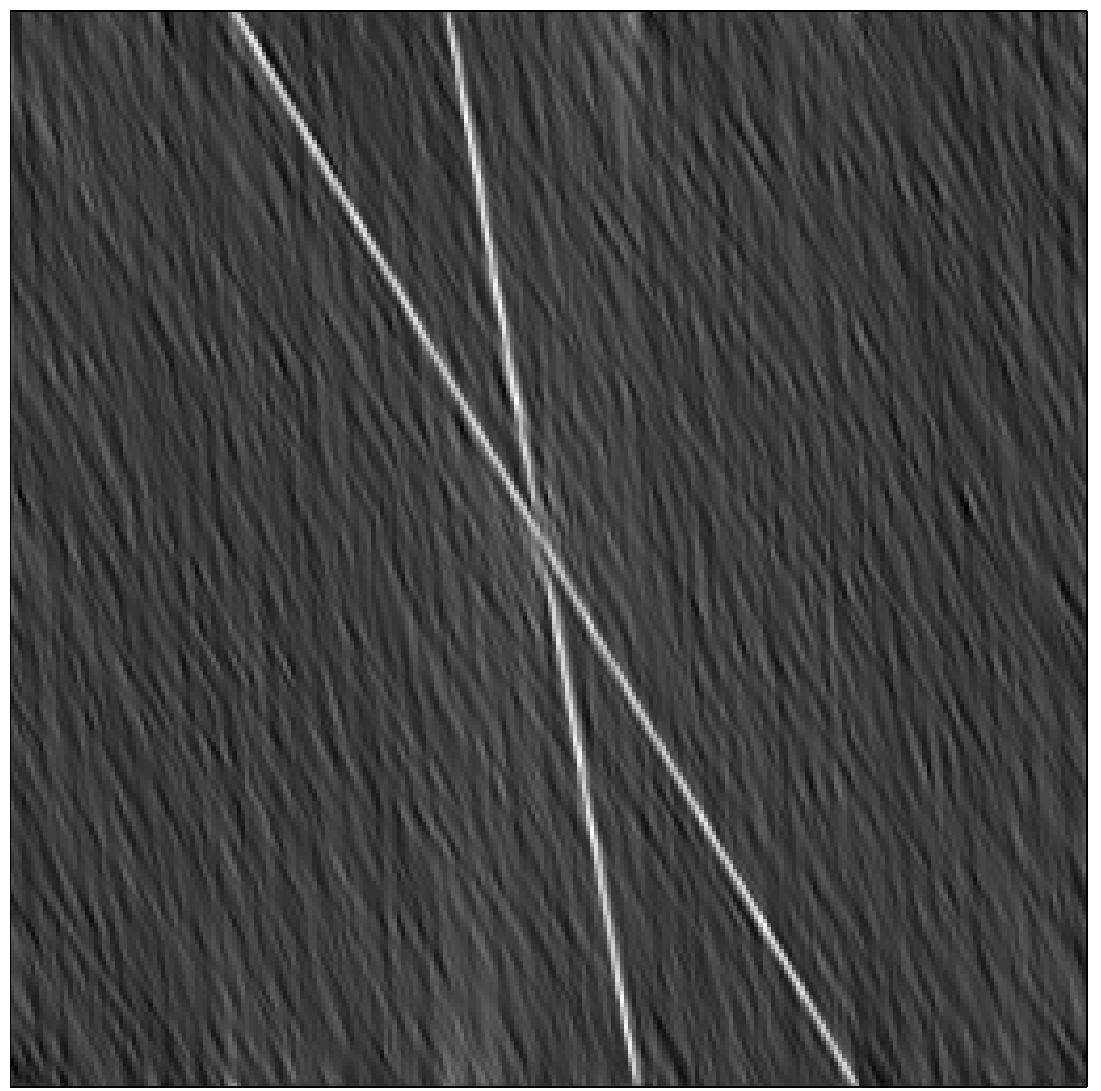}
	&\includegraphics[height=4.2cm]{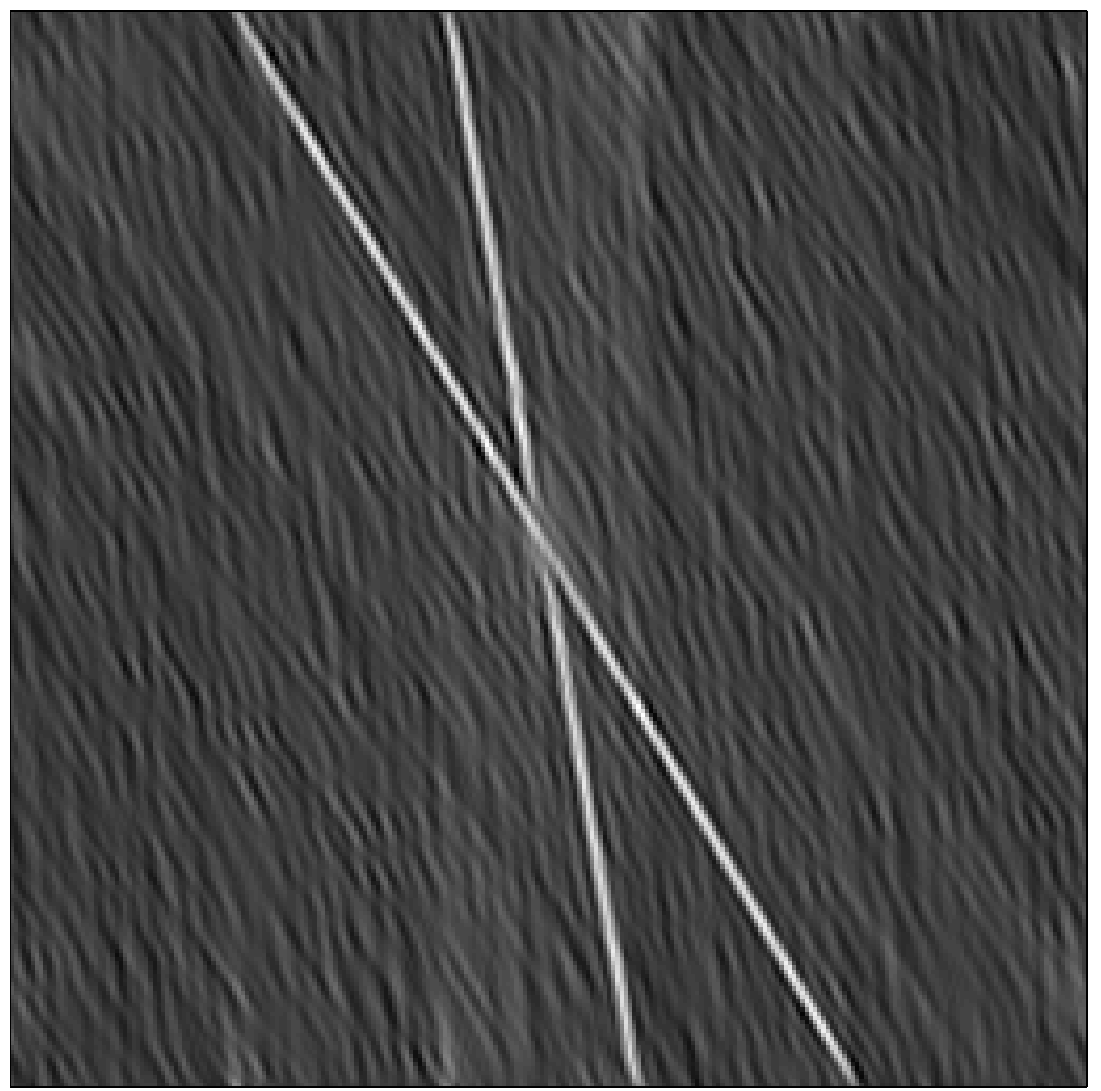}
	&\includegraphics[height=4.2cm]{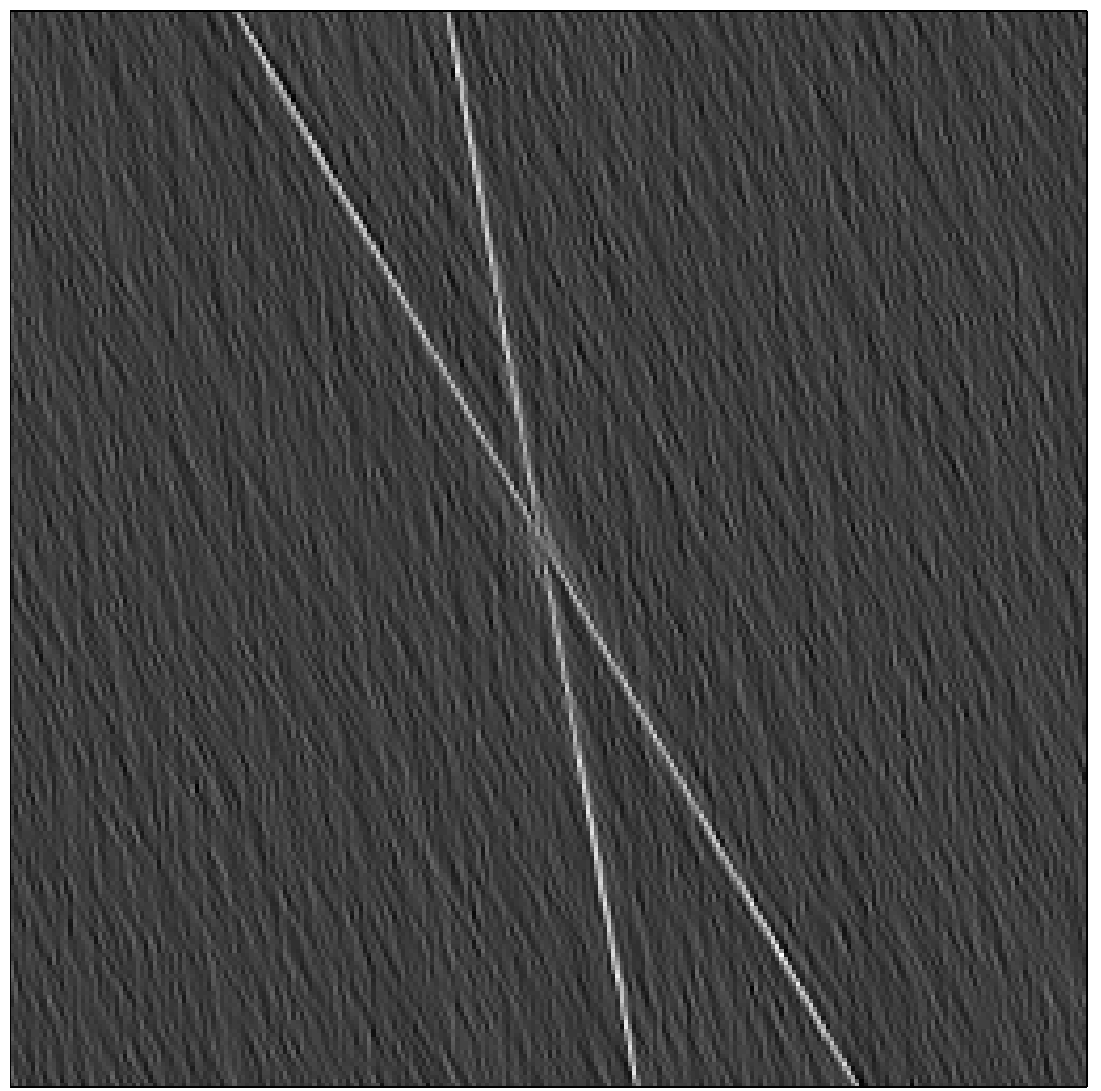}\\[2ex] 
	\midrule\\

	\Theta=160^\circ 
	&\includegraphics[height=4.2cm]{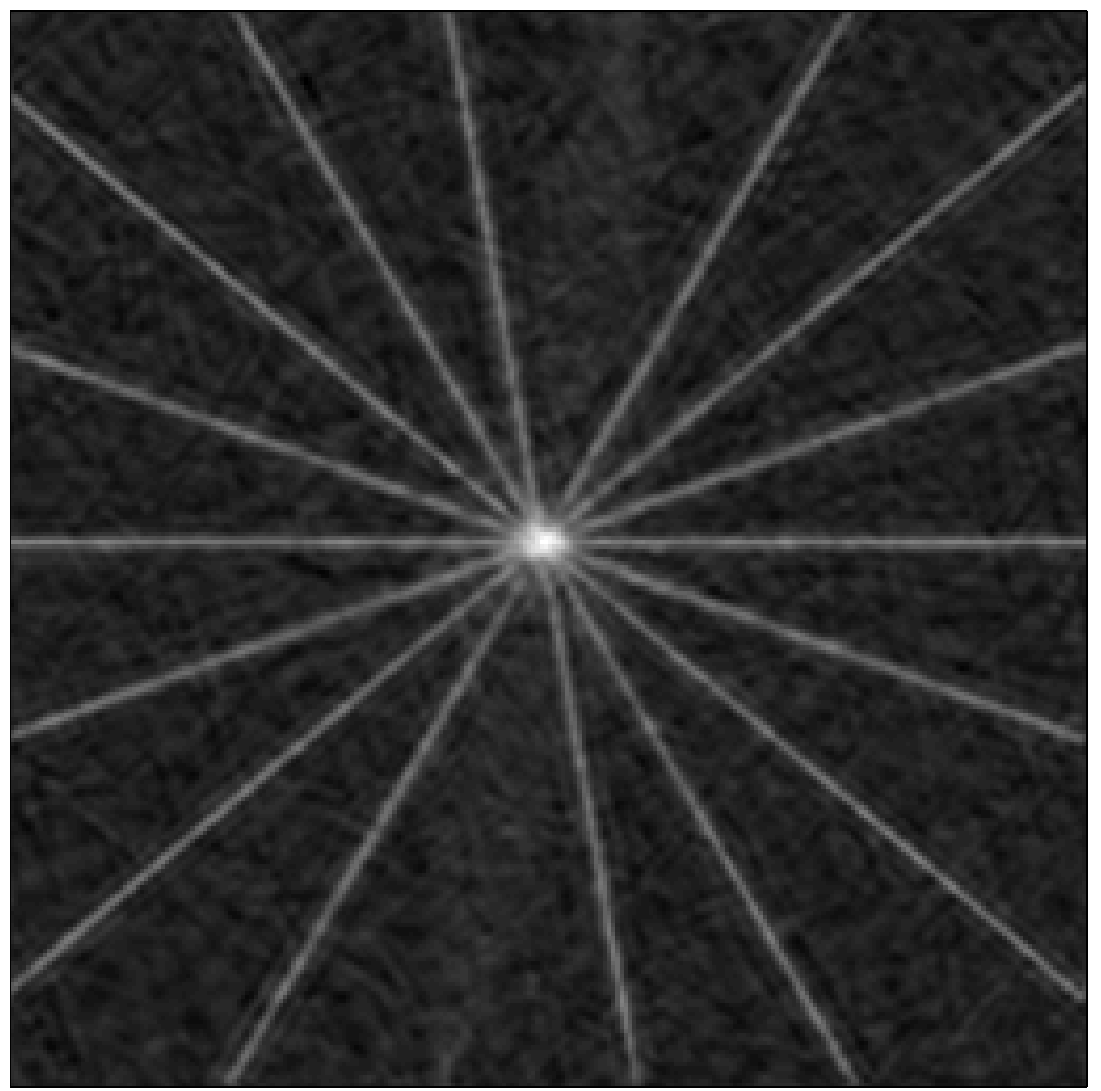}
	&\includegraphics[height=4.2cm]{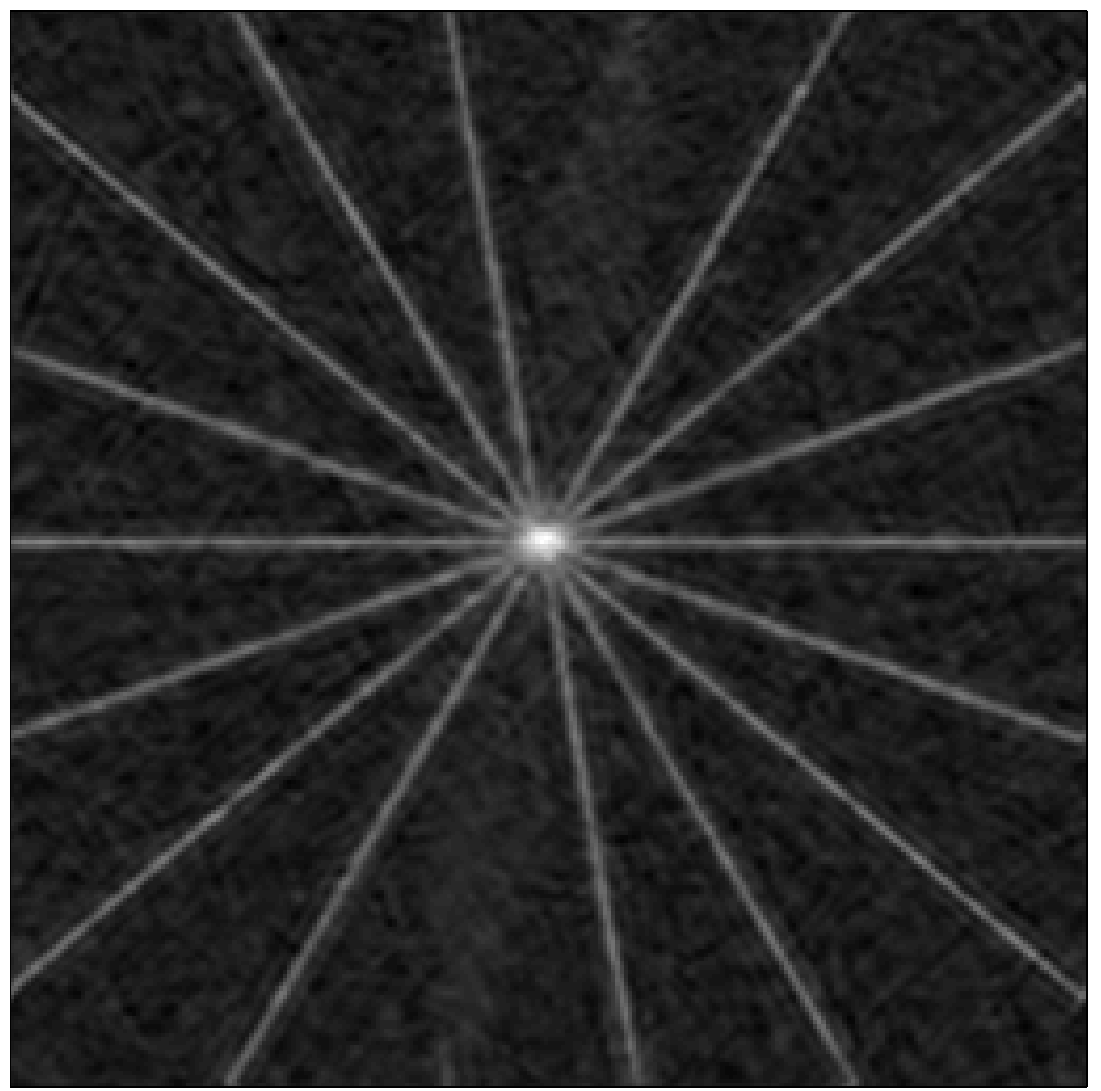}
	&\includegraphics[height=4.2cm]{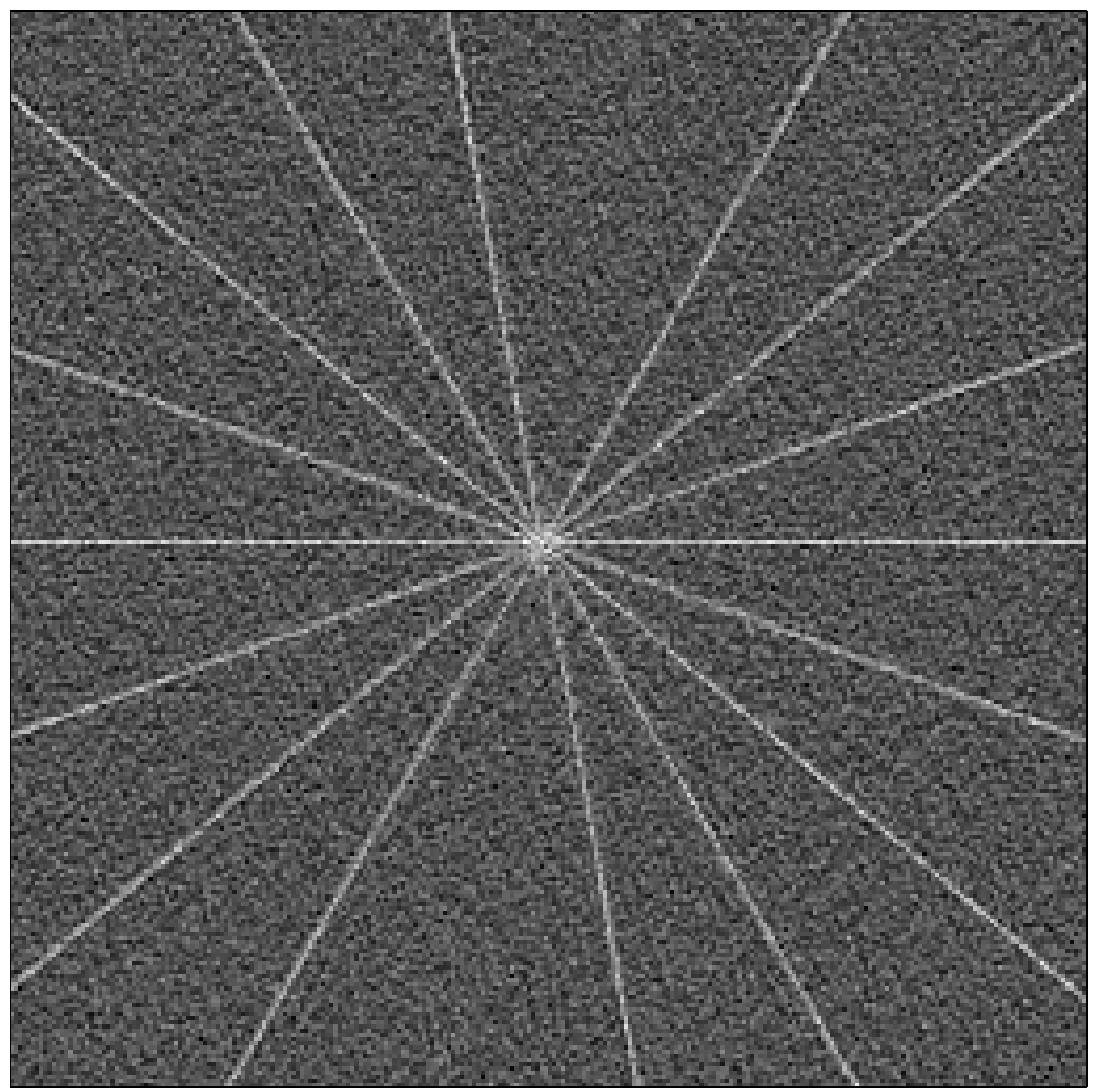}\\[2ex] 
	\bottomrule
\end{array}
\]
\caption{Reconstruction of a radial pattern of size $256\times 256$ (Figure \ref{fig:test images}\subref{subfig:radial pattern}) at an angular range $[0,\Theta]$, noiselevel $2\%$ using CSR, A-CSR and FBP. Same conclusions can be drawn as in the Figures \ref{fig:rec phantom} and \ref{fig:rec brainstem}. However, in this case we can additionally observe that only those lines are reconstructed whose normal directions are within the available angular range.}
\label{fig:rec radial pattern}
\end{figure}

\begin{figure}[H]
\centering
\begin{tikzpicture}
	\begin{axis}[
		height=5.2cm, 
		width=12.95cm,
		xmin = 0,
		xmax= 190,
		ymin=0,
		ymax=0.009,
		axis x line = bottom,
		axis y line = left,
		xlabel={Angular range $\Theta$},
		xtick={0,10,...,180},
		xticklabels={$0^\circ$,,,$30^\circ$,,,$60^\circ$,,,$90^\circ$,,,$120^\circ$,,,$150^\circ$,,,$180^\circ$},
		ylabel={MSE},
		y label style ={at={(-0.02,0.67)},anchor=south east},
		ytick={0.002,0.004,0.006,0.008},
		yticklabels={$2\cdot 10^{-3}$,$4\cdot 10^{-3}$,$6\cdot 10^{-3}$,$8\cdot 10^{-3}$},
		scaled ticks=false,
		tick label style={font=\small},
		label style={font=\small},
		legend style={font=\small},
		cycle list name=black white,
		legend style={at={(0.97,0.67)},anchor=south east},
		skip coords between index={0}{1},
		]
		\addplot []  file {csr_mse_OVERALL_coeffs_increment10.dat};
		\addlegendentry{CSR}
		
		\addplot [mark=star,thick,only marks]  file {csr_mse_OVERALLadapted_coeffs_increment10.dat};
		\addlegendentry{A-CSR}
	\end{axis}%
\end{tikzpicture}%

\caption{Mean squared error (MSE) of reconstructed curvelet coefficients, i.e., $\frac{1}{N}\sum_{n=1}^N\abs{c_n-c^\mathrm{rec}_n}^2$, where $c$ denotes the vector of curvelet coefficients of the original image and $c^\mathrm{rec}$ denotes the vector of reconstructed curvelet coefficients. This plot shows the results of the reconstruction of Shepp-Logan head phantom at an angular range $[0,\Theta]$ and noiselevel $2\%$.}
\label{plot:mse}
\end{figure}
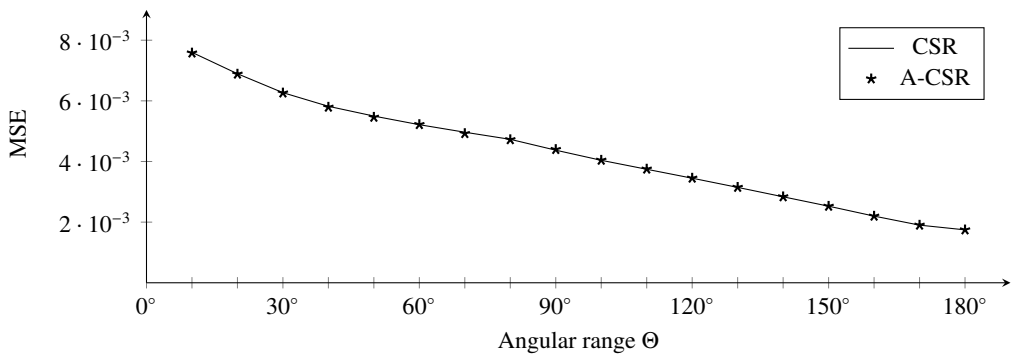

\begin{figure}[H]
\centering
\begin{tikzpicture}
	\begin{semilogyaxis}[
		height=5.2cm, 
		width=12.5cm,
		xmin = 0,
		xmax= 190,
		axis x line = bottom,
		axis y line = left,
		xlabel={Angular range $\Theta$},
		xtick={0,10,...,180},
		xticklabels={$0^\circ$,,,$30^\circ$,,,$60^\circ$,,,$90^\circ$,,,$120^\circ$,,,$150^\circ$,,,$180^\circ$},
		ylabel={relative MSE},
		scaled ticks=false,
		tick label style={font=\small},
		label style={font=\small},
		legend style={font=\small},
		cycle list name=black white,
		legend style={at={(0.97,0.67)},anchor=south east},
		skip coords between index={17}{19},
		skip coords between index={0}{1},
		]
		\addplot []  file {csr_relMSE_OVERALL_coeffs_increment10.dat};
	\end{semilogyaxis}%
\end{tikzpicture}%

\caption{Relative MSE of reconstructed curvelet coefficients, i.e., $\frac{1}{N}\sum_{n=1}^N\abs{c^\mathrm{CSR}_n-c^\mathrm{A-CSR}_n}^2$, where $c^\mathrm{CSR}$ and $c^\mathrm{A-CSR}$ denote the curvelet coefficient vector of the CSR reconstruction and A-CSR reconstruction, respectively. This plot shows the results of the reconstruction of Shepp-Logan head phantom at an angular range $[0,\Theta]$ and noiselevel $2\%$.}
\label{plot:relative mse}
\end{figure}
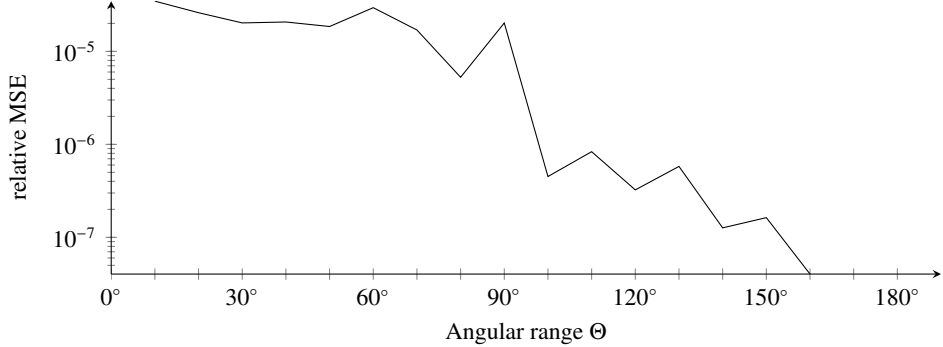

Eventually, we compare the reconstruction quality of CSR and A-CSR reconstruction to the quality of filtered backprojection (FBP) reconstructions. From Figures \ref{fig:rec phantom} - \ref{fig:rec radial pattern} we can observe that the FBP reconstructions contain more noise than reconstructions obtained through curvelet sparse regularization. Visually, the FBP reconstructions seem to be inferior to the CSR and A-CSR reconstructions. On the other hand, the visual impression of the CSR and A-CSR reconstruction appears to be quite good. Though CSR and A-CSR reconstructions are slightly smoother than the FBP reconstructions, all details are well preserved and the edges are still clearly visible. 

To verify the visual impressions, we computed the peak signal-to-noise-ratio (PSNR) of the normalized reconstructions\footnote{The gray values of the reconstructed images were normalized to the interval $[0,1]$} by
\begin{equation*}
	\mathrm{PSNR}(c^\mathrm{rec}) = 10\log\paren{\frac{1}{\mathrm{MSE}(c^\mathrm{rec})}}.
\end{equation*}
These values are shown in the Table \ref{tab:psnr values}. For each test image and each angular range, we can observe that the PSNR values of curvelet sparse regularizations (CSR and A-CSR) are considerably larger than those of the FBP reconstructions. Since larger PSNR values correspond to a better image quality, these results again confirm the visual impression. 

\begin{table}[H]
\centering
\subfloat[Shepp-Logan head phantom]{
	\begin{tabular}{lccc}
	\toprule
	& CSR & A-CSR & FBP \\
	\midrule
	$\Theta=35^\circ$   & 13.4 & 13.4 & 7.5   \\[1ex]
	$\Theta=160^\circ$ & 19.7 & 19.7 & 13   \\
	\bottomrule
	\end{tabular}
}\hskip4ex
\subfloat[Brainstem]{
	\begin{tabular}{lccc}
	\toprule
	& CSR & A-CSR & FBP \\
	\midrule
	$\Theta=35^\circ$   & 13.8 & 13.7 &  9.2   \\[1ex]
	$\Theta=160^\circ$ & 18.3 & 18.4 &  12.2  \\
	\bottomrule
	\end{tabular}
}\\[2ex]
\subfloat[Radial pattern]{
	\begin{tabular}{lccc}
	\toprule
	& CSR & A-CSR & FBP \\
	\midrule
	$\Theta=35^\circ$   & 14.5 & 14.5 & 11   \\[1ex]
	$\Theta=160^\circ$ & 16.2  & 16.2 & 9.3   \\
	\bottomrule
	\end{tabular}
}
\caption{PSNR values of normalized reconstructions.}
\label{tab:psnr values}
\end{table}

\subsection{Comments}
Our intention to perform these experiments was to give a practical proof of concept for our results. The implementation of the reconstruction algorithms is therefore very rudimental and, hence, there is much room for improvements or optimizations. For example, the execution times that are presented in Figure \ref{fig:execution times} may be improved by a more elaborate implementation of the Algorithm \ref{alg1}. Though, there are many other algorithms available in the literature, the reason to use the iterative soft-thresholding algorithm for our experiments was its simplicity.

\section{Summary \& Concluding remarks}
\label{sec:conclusion}
In this work we have introduced curvelet sparse regularization as a stable reconstruction method for the limited angle tomography. The stabilizing nature of the this method was demonstrated in numerical experiments. In comparison to the FBP reconstructions, curvelet sparse regularization reconstructions offered a superior reconstruction quality. Another issue, that was addressed by CSR is its ability to produce edge-preserving reconstructions. The numerical experiments confirmed that to some extent. Wa have seen that all details in the in the CSR reconstructions were well preserved and the edges were clearly visible. However, CSR reconstructions were found to be smoother (more blurry) than FBP reconstructions. We believe that an even better edge-preservation can be achieved by tuning the reconstruction procedure. 

The main part of this work was devoted to the characterization of curvelet sparse regularizations in limited angle tomography. In Section \ref{sec:characterizations}, we have given a characterization of limited angle CSR reconstructions in terms of visible and invisible curvelet coefficients. Based on this characterization, an adapted CSR method was formulated. The adaptivity of this approach results from the fact that, depending on the available angular range, the curvelet dictionary can be partitioned into a sub-dictionary of visible curvelets and a sub-dictionary of invisible curvelets. So, by formulating the reconstruction problem only with respect to the visible curvelet sub-dictionary, the problem becomes adapted to the limited angle geometry. Moreover, this entails a significant dimensionality reduction of the original reconstruction problem. This dimensionality reduction can be easily implemented in practice. A proof of this concept was given by numerical experiments. As a result, we found that the achieved dimensionality reduction is considerable, especially, when the available angular range is small. Consequently, a significant speedup of the reconstruction algorithms was observed. The reconstruction quality of the adapted approach, however, was found to be equal to that of the non-adapted method.  

Furthermore, we would like to note that the results of this work can be generalized to the three-dimensional setting. The ideas of this work carry over to this situation, even though, the analysis is more technical in this case. 

We conclude this article by emphasizing the role of curvelets in case of limited angle tomography and summarize the reasons why they were used in this work: On the one hand, curvelets provide a sparse representation of functions with an optimal encoding of edges. These properties qualify curvelets for the use in sparse regularization and give rise to an edge-preserving reconstruction. On the other hand, curvelets are highly directional. Therefore, they enable a separation of visible and invisible structures of a function which is imaged at a limited angular range. Because of this directionality, curvelets allow to adapt the problem the limited angle setting.

\section*{Acknowledgements}
The author gratefully acknowledges the support from GE Healthcare, Image Diagnost International, Munich. He especially thanks Peter Heinlein (GE Healthcare, Image Diagnost International, Munich) for his support during this work. 
The author also acknowledges the support of the TUM Graduate School’s Thematic Graduate Center ISAM at Technische Universit\"at Mü\"unchen, Germany.

\bibliographystyle{elsarticle-harv}
\bibliography{references}

\begin{thebibliography}{30}
\expandafter\ifx\csname natexlab\endcsname\relax\def\natexlab#1{#1}\fi
\expandafter\ifx\csname url\endcsname\relax
  \def\url#1{\texttt{#1}}\fi
\expandafter\ifx\csname urlprefix\endcsname\relax\def\urlprefix{URL }\fi

\bibitem[{Bredies et~al.(2010)Bredies, Kunisch, and Pock}]{Bredies:2010gt}
Bredies, K., Kunisch, K., Pock, T., 2010. {Total Generalized Variation}. SIAM
  Journal on Imaging Sciences 3~(3), 492 -- 526.

\bibitem[{Bredies and Lorenz(2008)}]{BrediesLorenz2008}
Bredies, K., Lorenz, D.~A., 2008. {Linear Convergence of Iterative
  Soft-Thresholding}. Journal of Fourier Analysis and Applications 14~(5-6),
  813--837.

\bibitem[{Cand\`{e}s et~al.(2008)Cand\`{e}s, Demanet, Donoho, and
  Ying}]{CurveLab}
Cand\`{e}s, E., Demanet, L., Donoho, D.~L., Ying, L., 2008. Curvelab-2.1.2.
  \texttt{http://www.curvelet.org/}.

\bibitem[{Cand{\`e}s and
  Donoho(2002)}]{Candes_Recovering_Edges_in_Illposed_problems02}
Cand{\`e}s, E.~J., Donoho, D.~L., 2002. Recovering edges in ill-posed inverse
  problems: optimality of curvelet frames. Ann. Statist. 30~(3), 784--842,
  dedicated to the memory of Lucien Le Cam.

\bibitem[{Cand{\`e}s and Donoho(2004)}]{Candes_New_tight_Frames_of_Curvelets04}
Cand{\`e}s, E.~J., Donoho, D.~L., 2004. New tight frames of curvelets and
  optimal representations of objects with piecewise {$C\sp 2$} singularities.
  Comm. Pure Appl. Math. 57~(2), 219--266.

\bibitem[{Cand{\`e}s and
  Donoho(2005{\natexlab{a}})}]{Candes_CCT_Wavefront_Set05}
Cand{\`e}s, E.~J., Donoho, D.~L., 2005{\natexlab{a}}. Continuous curvelet
  transform. {I}. {R}esolution of the wavefront set. Appl. Comput. Harmon.
  Anal. 19, 162--197.

\bibitem[{Cand{\`e}s and Donoho(2005{\natexlab{b}})}]{Candes_CCTII}
Cand{\`e}s, E.~J., Donoho, D.~L., 2005{\natexlab{b}}. Continuous curvelet
  transform. {II}. {D}iscretization and {F}rames. Appl. Comput. Harmon. Anal.
  19~(2), 198--222.

\bibitem[{Caselles et~al.(2007)Caselles, Chambolle, and
  Novaga}]{Caselles:2007hk}
Caselles, V., Chambolle, A., Novaga, M., 2007. {The discontinuity set of
  solutions of the TV denoising problem and some extensions}. Multiscale
  Modeling {\&} Simulation 6~(3), 879--894.

\bibitem[{Chui(1992)}]{Chui_IntroductionToWavelets}
Chui, C.~K., 1992. An introduction to wavelets. Vol.~1 of Wavelet Analysis and
  its Applications. Academic Press Inc., Boston, MA.

\bibitem[{Daubechies
  et~al.(2004)}]{Daubechies_Iterative_Thresholding_Algorithm_for_Inverse_Problems04}
Daubechies, I., et~al., 2004. An iterative thresholding algorithm for linear
  inverse problems with a sparsity constraint. Comm. Pure Appl. Math. 57~(11),
  1413--1457.

\bibitem[{Davison(1983)}]{Davison83}
Davison, M.~E., 1983. The ill-conditioned nature of the limited angle
  tomography problem. SIAM Journal on Applied Mathematics 43~(2), 428--448.

\bibitem[{Engl et~al.(1996)Engl, Hanke, and Neubauer}]{Regularization_Engl96}
Engl, H.~W., Hanke, M., Neubauer, A., 1996. Regularization of inverse problems.
  Vol. 375 of Mathematics and its Applications. Kluwer Academic Publishers
  Group, Dordrecht.

\bibitem[{Fadili and Peyr{\'e}(2011)}]{Fadili:cs}
Fadili, J.~M., Peyr{\'e}, G., 2011. {Total Variation Projection With First
  Order Schemes}. IEEE Transactions on Image Processing 20~(3), 657--669.

\bibitem[{Frikel(2010)}]{Frikel2010}
Frikel, J., April 2010. A new framework for sparse regularization in limited
  angle x-ray tomography. Biomedical Imaging: From Nano to Macro, 2010 IEEE
  International Symposium on, 824--827.

\bibitem[{Frikel(2011)}]{FrikelPAMM2011}
Frikel, J., May 2011. {Short communication: Dimensionality reduction of
  curvelet sparse regularizations in limited angle tomography}. Submitted to
  the Proceedings in Applied Mathematics and Mechanics~(2 pages).

\bibitem[{Griesse and
  Lorenz(2008)}]{Lorenz_Semismooth_Newton_for_SparseTikhonov08}
Griesse, R., Lorenz, D.~A., 2008. A semismooth newton method for tikhonov
  functionals with sparsity constraints. Inverse Problems 24~(3), 035007
  (19pp).
\newline\urlprefix\url{http://stacks.iop.org/0266-5611/24/035007}

\bibitem[{Hansen et~al.(2011)Hansen, Sidky, and Pan}]{Hansen:2011wf}
Hansen, P.~C., Sidky, E.~Y., Pan, X., may 2011. {Accelerated gradient methods
  for total-variation-based CT image reconstruction}. arXiv.org math.NA.

\bibitem[{Herman and Davidi(2008)}]{HermanDavidi2008}
Herman, G.~T., Davidi, R., 2008. Image reconstruction from a small number of
  projections. Inverse Problems 24~(4), 045011.
\newline\urlprefix\url{http://stacks.iop.org/0266-5611/24/i=4/a=045011}

\bibitem[{Kolehmainen et~al.(2003)}]{Kolehmainen03}
Kolehmainen, V., et~al., 2003. Statistical inversion for medical x-ray
  tomography with view radiographs: {II}. {A}pplication to dental radiology.
  Phys. Med. Biol. 48, 1465--1490.

\bibitem[{Lorenz and Trede(2008)}]{LorenzTrede2008}
Lorenz, D.~A., Trede, D., 2008. Optimal convergence rates for tikhonov
  regularization in besov scales. Inverse Problems 24~(5), 055010.
\newline\urlprefix\url{http://stacks.iop.org/0266-5611/24/i=5/a=055010}

\bibitem[{Mallat(2009)}]{Mallat_WaveletTour_SparseWay08}
Mallat, S., 2009. A wavelet tour of signal processing, 3rd Edition.
  Elsevier/Academic Press, Amsterdam, the sparse way, With contributions from
  Gabriel Peyr{\'e}.

\bibitem[{Natterer(1986)}]{Natterer86}
Natterer, F., 1986. The mathematics of computerized tomography. B. G. Teubner,
  Stuttgart.

\bibitem[{Quinto(1993)}]{Quinto93}
Quinto, E.~T., 1993. Singularities of the {X}-ray transform and limited data
  tomography in {$\mathbb{R}^2$} and {$\mathbb{R}^3$}. SIAM J. Math. Anal.
  24~(5), 1215--1225.

\bibitem[{Quinto(2006)}]{Quinto06}
Quinto, E.~T., 2006. An introduction to {X}-ray tomography and {R}adon
  transforms. In: The Radon transform, inverse problems, and tomography.
  Vol.~63 of Proc. Sympos. Appl. Math. Amer. Math. Soc., Providence, RI, pp.
  1--23.

\bibitem[{Radiopedia.org(2010)}]{Brainstem_testimage}
Radiopedia.org, 2010.
\newline\urlprefix\url{http://radiopaedia.org/cases/brainstem-glioma}

\bibitem[{Rantala et~al.(2006)}]{Rantala2006}
Rantala, M., et~al., February 2006. Wavelet-based reconstruction for limited
  angle x-ray tomography. IEEE Transactions on Medical Imaging 25~(2),
  210--217.

\bibitem[{Ring(2000)}]{Ring:2000vq}
Ring, W., 2000. {Structural properties of solutions to total variation
  regularization problems}. Mathematical modelling and numerical analysis
  34~(4), 799--810.

\bibitem[{Rockafellar(1970)}]{Rockafellar1970}
Rockafellar, R.~T., 1970. Convex analysis. Princeton Mathematical Series, No.
  28. Princeton University Press, Princeton, N.J.

\bibitem[{Scherzer et~al.(2009)Scherzer, Grasmair, Grossauer, Haltmeier, and
  Lenzen}]{VariationalMethodsScherzer2009}
Scherzer, O., Grasmair, M., Grossauer, H., Haltmeier, M., Lenzen, F., 2009.
  Variational methods in imaging. Vol. 167 of Applied Mathematical Sciences.
  Springer, New York.

\bibitem[{Stein and Weiss(1971)}]{SteinWeiss1971}
Stein, E.~M., Weiss, G., 1971. Introduction to {F}ourier analysis on
  {E}uclidean spaces. Princeton University Press, Princeton, N.J., princeton
  Mathematical Series, No. 32.

\end{thebibliography}

\end{document}